\documentclass[10pt]{article}
\usepackage{authblk}

\usepackage{bbm}
\usepackage{xcolor}
\definecolor{azure}{rgb}{0.25, 0.41, 0.88}
\usepackage{times}
\usepackage{mathtools}
\usepackage[
  colorlinks   = true,   
  linkcolor    = azure,  
  citecolor    = azure, 
  urlcolor     = azure,  
  linktoc      = page    
]{hyperref}
\usepackage[shortlabels]{enumitem}
\usepackage{mathrsfs}
\usepackage{mhequ}
\usepackage{graphicx}
\usepackage{subcaption}

\RequirePackage{color}
\RequirePackage[a4paper,top=1.6in,bottom=1.6in,left=1.6in,right=1.6in]{geometry}
\RequirePackage{lastpage}
\usepackage[english]{babel}
\RequirePackage{amsmath,amsthm,amssymb,mathtools}
\RequirePackage{enumitem}
\RequirePackage{placeins}

\usepackage{fancyhdr}
\usepackage{tikz}
\usetikzlibrary{arrows,decorations.pathmorphing,backgrounds,positioning,fit,petri}
\usetikzlibrary{shapes}

\usepackage[font=small]{caption}
\usepackage{titlesec}
\titleformat{\section}
  {\normalfont\fontsize{14}{15}\bfseries}{\thesection}{1em}{}
\titleformat{\subsection}
  {\normalfont\fontsize{12}{15}\bfseries}{\thesubsection}{1em}{}

\newcommand{\expected}{\mathbb{E}}
\newcommand{\prob}{\mathbb{P}}
\newcommand{\eps}{\varepsilon}

\def\scal#1{\langle#1\rangle} 

\makeatletter
\DeclareFontFamily{OMX}{MnSymbolE}{}
\DeclareSymbolFont{MnLargeSymbols}{OMX}{MnSymbolE}{m}{n}
\SetSymbolFont{MnLargeSymbols}{bold}{OMX}{MnSymbolE}{b}{n}
\DeclareFontShape{OMX}{MnSymbolE}{m}{n}{
    <-6>  MnSymbolE5
   <6-7>  MnSymbolE6
   <7-8>  MnSymbolE7
   <8-9>  MnSymbolE8
   <9-10> MnSymbolE9
  <10-12> MnSymbolE10
  <12->   MnSymbolE12
}{}
\DeclareFontShape{OMX}{MnSymbolE}{b}{n}{
    <-6>  MnSymbolE-Bold5
   <6-7>  MnSymbolE-Bold6
   <7-8>  MnSymbolE-Bold7
   <8-9>  MnSymbolE-Bold8
   <9-10> MnSymbolE-Bold9
  <10-12> MnSymbolE-Bold10
  <12->   MnSymbolE-Bold12
}{}

\let\llangle\@undefined
\let\rrangle\@undefined
\DeclareMathDelimiter{\llangle}{\mathopen}
                     {MnLargeSymbols}{'164}{MnLargeSymbols}{'164}
\DeclareMathDelimiter{\rrangle}{\mathclose}%
                     {MnLargeSymbols}{'171}{MnLargeSymbols}{'171}
\makeatother

\DeclareMathOperator*{\mymax}{\text{\normalfont max}}
\DeclareMathOperator*{\mymin}{\text{\normalfont min}}
\DeclareMathOperator*{\mysup}{\text{\normalfont sup}}
\DeclareMathOperator*{\myinf}{\text{\normalfont inf}}
\DeclareMathOperator*{\mylim}{\text{\normalfont lim}}
\DeclareMathOperator*{\mylog}{\text{\normalfont log}}
\DeclareMathOperator*{\mylimsup}{\text{\normalfont limsup}}
\DeclareMathOperator*{\myliminf}{\text{\normalfont liminf}}

\DeclareMathOperator*{\de}{\hspace{-.1em}\text{\normalfont d\hspace{-.15em}}}
\DeclareMathOperator{\Var}{Var}
\DeclareMathOperator{\Dir}{Dir}

\numberwithin{equation}{section}
\newtheorem{theorem}{Theorem}[section]
\newtheorem{corollary}[theorem]{Corollary}
\newtheorem{lemma}[theorem]{Lemma}
\newtheorem{proposition}[theorem]{Proposition}
\theoremstyle{definition}
\newtheorem{definition}[theorem]{Definition}
\newtheorem{remark}[theorem]{Remark}
\newtheorem*{acknowledgements}{Acknowledgements}

\DeclareFontEncoding{LS1}{}{}
\DeclareFontSubstitution{LS1}{stix}{m}{n}
\DeclareSymbolFont{stixletters}{LS1}{stix}{m}{it}
\DeclareMathAccent{\cev}{\mathord}{stixletters}{"91}
\DeclareMathAccent{\vec}{\mathord}{stixletters}{"92}

\makeatletter
\newcommand\@affiliationlist{}

\renewcommand\affil[2][]{%
  \g@addto@macro\@affiliationlist{%
    \item[\textsuperscript{\scriptsize #1}] #2%
  }%
}

\newlength{\affillabelwidth}

\renewcommand{\maketitle}{
  \vspace*{0em}%
  \begin{center}{\LARGE\bfseries \@title \par}\end{center}
  \vspace{1em}%
  \begin{center}{\small \scshape\@date \par}\end{center}\par 
  \vspace{4em}
  \begin{list}{}{\setlength{\leftmargin}{1.5em}\setlength{\rightmargin}{1.5em}}
    \item[] {\raggedright\large\@author \par}
  \end{list}

  \settowidth{\affillabelwidth}{\textsuperscript{\scriptsize 8}} 
  \begin{list}{}{%
      \setlength{\leftmargin}{2.2em}
      \setlength{\rightmargin}{1.5em}
      \setlength{\labelwidth}{\affillabelwidth}
      \setlength{\labelsep}{0.2em}
      \setlength{\itemsep}{.2em}
      \setlength{\parsep}{0pt}
      \setlength{\topsep}{0pt}
  }
    {\raggedright\small \@affiliationlist}
  \end{list}\vspace{0.5em}
}

\renewenvironment{abstract}
  {\par                               
   \begin{list}{}{                        
      \setlength{\leftmargin}{1.5em}
      \setlength{\rightmargin}{1.5em}}
   \item[]                                   
   {\bfseries\abstractname}\par               
   \small\noindent\ignorespaces}             
  {\end{list}}

\begin{document}

\allowdisplaybreaks

\date{\today}

\title{Fluctuations of the Boundary-Driven Symmetric Zero-Range  Process from the NESS}

\author{\bfseries Patrícia Gonçalves\textsuperscript{1}, Adriana Neumann\textsuperscript{2}, Maria Chiara Ricciuti\textsuperscript{3}
}

\affil[1]{Center for Mathematical Analysis, Geometry and Dynamical Systems, Instituto Superior Técnico, Universidade de Lisboa, Av. Rovisco Pais, 1049-001 Lisboa, Portugal, Email: \href{mailto:pgoncalves@tecnico.ulisboa.pt}{\texttt{pgoncalves@tecnico.ulisboa.pt}}}
\affil[2]{Mathematics and Statistics Institute, Universidade Federal do Rio Grande do Sul, Av. Bento Gonçalves, 9500, Porto Alegre, RS, Brazil, Email: \href{mailto:aneumann@mat.ufrgs.br}{\texttt{aneumann@mat.ufrgs.br}}}
\affil[3]{Department of Mathematics, Imperial College London, 180 Queen's Gate, London SW7 2AZ, United Kingdom, Email: \href{mailto:maria.ricciuti18@imperial.ac.uk}{\texttt{maria.ricciuti18@imperial.ac.uk}}}

\clearpage
\maketitle
\thispagestyle{empty}

\begin{abstract}
We study the non-equilibrium stationary fluctuations of a symmetric zero-range process on the discrete interval $\{1, \ldots, N-1\}$ coupled to reservoirs at sites $1$ and $N-1$, which inject and remove particles at rates proportional to $N^{-\theta}$ for any value of $\theta\in\mathbb{R}$. We prove that, if the jump rate is bounded and under diffusive scaling, the fluctuations converge to the solution of a generalised Ornstein-Uhlenbeck equation with characteristic operators that depend on the stationary density profile. The limiting equation is supplemented with boundary conditions of Dirichlet, Robin, or Neumann type, depending on the strength of the reservoirs. We also introduce two notions of solutions to the corresponding martingale problems, which differ according to the choice of test functions.
\end{abstract}

%%%%%%%%%%%%%%%%%%%%%%%%%%%%%%%%%%%%%%%%%%%%%%%%%%
%%%CONTENTS%%%%%%%%%%%%%%%%%%%%%%%%%%%%%%%%%%%%%%%%
%%%%%%%%%%%%%%%%%%%%%%%%%%%%%%%%%%%%%%%%%%%%%%%%%%
\setcounter{tocdepth}{1}
\tableofcontents

%%%%%%%%%%%%%%%%%%%%%%%%%%%%%%%%%%%%%%%%%%%%%%%%%%
%%%INTRODUCTION%%%%%%%%%%%%%%%%%%%%%%%%%%%%%%%%%%%
%%%%%%%%%%%%%%%%%%%%%%%%%%%%%%%%%%%%%%%%%%%%%%%%%%
\section{Introduction}

Understanding the macroscopic evolution of conserved quantities in interacting particle systems from their microscopic dynamics is a cornerstone of nonequilibrium statistical mechanics. In recent decades, substantial progress has been made in establishing hydrodynamic limits, that is, laws of large numbers describing the  (typically deterministic) evolution of macroscopic quantities such as particle densities at diffusive or superdiffusive scales. Once the macroscopic behaviour has been characterised, a natural next question concerns the behaviour of the random fluctuations around this deterministic limit. The study of these fluctuations provides insight into the emergence of macroscopic stochastic partial differential equations (SPDEs) and, more broadly, into universal features of the microscopic models.

Among the classical systems in this area, the \textit{zero-range process} plays a distinguished role. In this system, each lattice site $x$ may host an arbitrary number $\eta(x)$ of indistinguishable particles, and a particle at site $x$ jumps to a site $y$ with probability $p(x,y)$ scaled by a factor $g(\eta(x))$ (with the function $g$ satisfying mild assumptions to ensure well-defined dynamics). The model was introduced by Spitzer in 1970~\cite{spi70} and has since served as a canonical example of a conservative interacting particle system with unbounded state-space, and has been extensively studied both on the discrete $d$-dimensional torus and on $\mathbb{Z}^d$.

Another widely-studied model is the \textit{exclusion process.} In this system, particles are constrained by the exclusion rule, meaning that each site may contain at most a finite number of particles. The dynamics are defined by allowing a particle at site $x$ to attempt a jump to a site $y$ at rate $p(x,y)$, with the jump being suppressed if $y$ has already attained its maximum value. Despite its simple microscopic definition, this process displays a remarkably rich macroscopic behaviour, and it has served as a fundamental testing ground for the development of techniques in hydrodynamic limits, fluctuation theory and large deviations. Owing to its simple structure and its well-behaved analytical properties, the exclusion process has been extensively analysed and remains one of the most thoroughly understood interacting particle systems.

When such a microscopic system evolves on a finite interval $I_N := \{1, \ldots, N-1\}$ and is coupled to stochastic reservoirs at the boundary, particles are injected and removed at random, breaking the conservation of mass and inducing a \textit{nonequilibrium steady state (NESS)}. For both the exclusion and the zero-range process evolving on the torus or $\mathbb{Z}$, stationary measures are translation invariant and of product form. However, when coupled with reservoirs, the NESS of the exclusion process becomes substantially more intricate. In general, the invariant measure is no longer of product form and exhibits non-trivial spatial correlations, even in one dimension. These long-range correlations, which persist uniformly in the system size, reflect the presence of a particle current induced by the boundary reservoirs, and except for some special cases, the explicit characterisation of the NESS remains unknown; see, for example, \cite{dehp93, fc23}. On the other hand, remarkably, for the boundary-driven zero-range process, the NESS is still of product form, but its parameters vary spatially, encoding the macroscopic density gradient generated by the boundary dynamics; see~\cite{lms05}.

As previously mentioned, the hydrodynamic limit of such boundary-driven systems describes the deterministic evolution of the macroscopic density profile. This problem has been extensively studied for the exclusion process (see, for example, \cite{bmns17, bgj19, gon19,  bgj21, gs22, bcgs23, gjmn20}). In this context, the interplay between bulk dynamics and boundary reservoirs leads to a rich collection of limiting behaviours, with the parameter $\theta$ governing the effective strength of the boundary interaction in the macroscopic limit. Depending on the regime, the hydrodynamic equation exhibits different boundary conditions: ranging from strong reservoirs imposing Dirichlet boundary conditions, to critical regimes yielding Robin-type conditions, and to weak reservoirs resulting in Neumann boundary conditions. These results reveal how microscopic injection and removal mechanisms translate into macroscopic boundary laws, and they play a crucial role in the analysis of stationary profiles, nonequilibrium fluctuations, and large-deviation principles for the associated empirical measures.

For the zero-range process, the hydrodynamic limit has long been well-known both on the discrete torus and on $\mathbb{Z}$, and in particular, under diffusive scaling, the empirical density evolves according to the nonlinear diffusion equation $\partial_t \rho_t(u) = \Delta\Phi(\rho_t(u))$, where the flux function $\Phi$ is determined by the expectation of the jump rate $g$ under the invariant measure. In the boundary-driven case, instead, the hydrodynamic behaviour was partially analysed in~\cite{fmn21} and then only recently rigorously established in~\cite{agnr25} for the symmetric, nearest-neighbour case (that is, for $p(x,y)=\boldsymbol{1}_{\{x\pm1\}}(y)$) in the regimes $\theta\ge1$ and under some mild assumptions on the jump rate $g$. Therein, it was shown that the empirical density evolves according to the same nonlinear diffusion equation (but where the flux function $\Phi$ is now given by the expectation of $g$ under the NESS) supplemented with boundary conditions of Robin type for $\theta=1$ and Neumann type for $\theta>1$. The result for $\theta<1$ remains an open problem, but it is conjectured to be described by the same equation with boundary conditions of Dirichlet type.

In the present article, we answer a question left open in~\cite{agnr25} and analyse the fluctuations around the stationary state of this system. Specifically, we study the sequence of processes describing the spacetime fluctuations of the empirical density field under the stationary measure in the diffusive scaling. Our main result shows that, when the jump rate $g$ is bounded from above and from below by a positive constant, these fluctuations converge, as $N\to\infty$, to the solution of the \textit{generalised Ornstein-Uhlenbeck equation} $\partial_t\mathscr{Y}_t=\Phi'(\bar\rho_\theta)\Delta \mathscr{Y}_t +\nabla \mathscr{\dot W}_t$, where $\bar\rho_\theta$ is the stationary density profile and $\mathscr{\dot W}_t$ denotes a space-time white noise. This limiting equation is supplemented with boundary conditions whose nature -- Dirichlet, Robin, or Neumann -- depends on the strength of the boundary reservoirs, namely on the value of $\theta$. Our result covers all values $\theta\in\mathbb{R}$, and thus provides a complete description -- in the case of $g$ bounded -- of the stationary fluctuations of the boundary-driven zero-range process in all regimes of boundary intensity. 

The study of fluctuations for boundary-driven systems has a long history. For the symmetric exclusion process, Gaussian fluctuations around the stationary state were established in~\cite{fgn17,fgn19}, yielding Ornstein-Uhlenbeck limits with boundary conditions depending on the reservoir strength. Beyond the stationary regime, the time-dependent nonequilibrium fluctuations have also been analysed (see~\cite{gjmn20}), revealing spacetime covariance structures different from their equilibrium counterparts. These works highlight the sensitivity of fluctuation behaviour to the scaling of the boundary dynamics and the resulting macroscopic boundary conditions. In contrast, for the zero-range process, much less is known. On the torus or infinite lattice, the equilibrium fluctuations have long been well understood (see~\cite{kl99}), but results in the open-boundary setting were lacking. Our work fills this gap by providing the first rigorous characterisation of equilibrium fluctuations in the boundary-driven setting, including the identification of the limiting Ornstein-Uhlenbeck process and its boundary conditions.

\vspace{1em}
\textbf{Structure of the Proof.}
Our strategy combines elements from the equilibrium fluctuation theory of~\cite{kl99} with adaptations to the non-translation-invariant, boundary-driven setting. The starting point is the construction of an appropriate martingale decomposition for the fluctuation field, obtained via Dynkin’s formula applied to chosen test functions. A key point is the identification of this space of test functions, which needs to be chosen in such a way that the boundary terms of the equation do not diverge and all terms of the martingale equation are well defined, yet rich enough to yield uniqueness of the martingale problem. In the case $\theta\ge0$, our chosen space of test functions is indeed large enough, and we establish uniqueness by making use of results about regular Sturm-Liouville problems; see Appendix~\ref{sec:app_uniqueness}. For $\theta<0$, instead, we need to impose that derivatives of all orders of these test functions vanish at the boundary. While this space is too small to guarantee uniqueness, similar to what was done in~\cite{bgjs22}, we prove two additional conditions on the fluctuation field which then guarantee uniqueness; see Proposition~\ref{prop:uniqueness_neg}. This means that we obtain two martingale problems in the case of Dirichlet boundary conditions.

We are also able to provide a new martingale problem for the Robin and Neumann cases, for a large space of test functions that do not satisfy any boundary condition. This new martingale problem contains an additional term with respect to the aforementioned one, reflecting the fact that no boundary conditions are imposed on the test functions. This manifests itself at the martingale level as an extra term appearing in the form of an additive functional, which depends on the function $H$ through its value at the boundary. Since the space of test functions in this case is larger than the previous one, which already guarantees uniqueness, we immediately conclude uniqueness in law for the solution of this new martingale problem. We also observe that this new finding is not restricted to the zero-range process, but also appears in other models, such as the exclusion process, although a proof of this fact has never been established. To the best of our knowledge, this is the first time that a new formulation of the martingale problem of the Robin-Neumann regimes is provided, which may be useful for both linear and nonlinear particle systems.

In all the regimes, one of the main difficulties of the proof consists in establishing a \textit{Boltzmann-Gibbs principle}, namely Theorem~\ref{thm:boltzmann_gibbs}, which allows us to close the martingale equation in terms of the density fluctuation field. Our method follows the approach of the original result of~\cite{br84} (see also~\cite{kl99}), but requires substantial adaptations due to the lack of reversibility and translation invariance of the stationary measure. Indeed, due to the non-reversibility, when localising the generator to finite boxes, these are not mean-zero: to get around this, we introduce an appropriate correction of the generator which makes it locally centred with respect to the stationary measure, and then bound this correction via the Cauchy-Schwarz inequality. The non-translation invariance, on the other hand, means that we cannot automatically reduce the proof to a single finite-size box. To this end, we define an appropriate weighted translation operator and then apply a trick inspired by the Boltzmann-Gibbs principle of~\cite{lab18}, which takes advantage of the fact that the stationary measure is ``smooth in space". This result (and its local version at the boundary for $\theta<0$) is the only instance where we need the boundedness assumption for the jump rate $g$; everything else holds in full generality, as long as $g$ satisfies standard assumptions to guarantee well-defined dynamics. 

Another difficulty in the proof is the derivation of the boundary conditions in the martingale problems. We are able to replace the boundary contributions by macroscopic averages and, from these, deduce the corresponding macroscopic terms in the martingale problem. As a consequence, we obtain two different martingale problems in each regime of $\theta$, both of which admit unique (in law) solutions. 

\vspace{1em}
\textbf{Future Work.} Several natural extensions of the present work remain open. First, it would be of interest to relax the boundedness assumption on the jump rate $g$, which is only used in the proof of the Boltzmann-Gibbs principle, and to develop techniques allowing the argument to be carried out under more general growth conditions.

Another interesting question is whether our results, in the specific case $g(k)=\boldsymbol{1}_{[1, \infty)}(k)$, can be transferred to the exclusion process, for instance by means of coupling arguments.

Finally, an important direction for future research is the study of the weakly asymmetric version of the model and the derivation of the crossover to the stochastic Burgers equation, as in~\cite{gj14}. Moreover, since our analysis is already performed at the level of the NESS, it would be of interest to extend these results to the non-equilibrium scenario, following the approach of~\cite{jm18}.

\vspace{1em}
\textbf{Outline.} In Section~\ref{sec:framework_results}, we define the model and its stationary measure, recall the result of~\cite{agnr25}, and introduce the spaces of test functions. Our main result, Theorem~\ref{thm:fluctuations}, establishes the convergence of the stationary fluctuation field to the solution of the corresponding generalised Ornstein-Uhlenbeck equation. The proof of this theorem is presented in Section~\ref{sec:proof}, which is divided into several parts. In Section~\ref{subsec:dynkin}, we derive the martingale representation and identify the limiting martingale equation. Section~\ref{subsec:bg} states and proves a Boltzmann–Gibbs principle tailored to our setting, together with a local version needed to handle the case $\theta<0$. In Section~\ref{subsec:replacement}, we prove a replacement lemma at the boundary, which allows us to control boundary terms for all regimes $\theta\in\mathbb{R}$. Section~\ref{subsec:tightness} establishes tightness of the sequence of density fluctuation fields, while Section~\ref{subsec:characterisation} proves the additional conditions required for $\theta<0$ and characterises limit points. In Section~\ref{sec:martingale_new} we derive the new formulation  of the martingale problems for the Robin and Neumann regimes. Finally, Appendix~\ref{sec:app_uniqueness} provides a proof of uniqueness for the martingale problems that characterise the limiting fluctuation fields.

\begin{acknowledgements} \small{The research of P.G. is partially funded by Fundação para a Ciência e Tecnologia (FCT), Portugal, through grant No. UID/4459/2025 and also the ERC/FCT SAUL project. M.C.R. gratefully acknowledges support from the Dean's PhD Scholarship at Imperial College London. 
A.N. acknowledges support from the National Council for Scientific and Technological Development (CNPq, Brazil) through a productivity research grant with reference  307606/2025-2.}
\end{acknowledgements}

%%%%%%%%%%%%%%%%%%%%%%%%%%%%%%%%%%%%%%%%%%%%%%%%%%
%%%FRAMEWORK AND RESULTS%%%%%%%%%%%%%%%%%%%%%%%%%%
%%%%%%%%%%%%%%%%%%%%%%%%%%%%%%%%%%%%%%%%%%%%%%%%%%
\section{Framework and Results}\label{sec:framework_results}

%%%%%%%%%%%%%%%%%%%%%%%%%%%%%%%%%%%%%%%%%%%%%%%%%%
%%%The Model%%%%%%%%%%%%%%%%%%%%%%%%%%%%%%%%%%%%%%
%%%%%%%%%%%%%%%%%%%%%%%%%%%%%%%%%%%%%%%%%%%%%%%%%%
\subsection{The Model}

Given a positive integer $N$, let $I_N:=\{1,\ldots, N-1\}$. We consider a system of indistinguishable particles moving on the set $I_N$, which we call \textit{bulk}, and we will denote by $\eta(x)$ the occupation variable at $x\in I_N$. The system we consider is known as the \textit{zero-range process} and it imposes no restriction on the total number of particles allowed per site, so that our state space is the set $\Omega_N:=\mathbb{N}^{I_N}$. The rates of the process are defined via a function $g:\mathbb{N}\to [0, \infty)$.

\vspace{1em}
\textbf{Assumptions on the Jump Rate $\boldsymbol{g}$}. Throughout this work, we assume that $g$ satisfies $g(0)=0$ (which reflects the fact that, if there are no particles at a site, then the jump rate from that site is null), and that it is strictly positive on the set of positive integers. Moreover, we assume that it has bounded variation, in the sense that
\begin{equation}\label{eq:function_g}
    g^*:=\mysup_{k\ge1}|g(k+1)-g(k)|<\infty.
\end{equation}
The condition above ensures that the process is well defined; see~\cite{lig05} for details. In addition, the proof of our main result will require the following boundedness condition for the jump rate $g$:
\begin{equation}\label{assumption:g_bdd}
0 < m_g := \myinf_{k\ge 1} g(k) \le \mysup_{k\ge 1} g(k) =: M_g < \infty.
\end{equation}
In fact, this condition will be used exclusively in Theorems~\ref{thm:boltzmann_gibbs} and \ref{thm:local_boltzmann_gibbs}, and we will indicate precisely where it is required in their proofs; no other result relies on this hypothesis.

\vspace{1em}
We consider the Markov process whose infinitesimal generator $\mathcal{L}^N$ is given by
\begin{equation*}
    \mathcal{L}^N := \mathcal{L}^N_l + \mathcal{L}^N_0 + \mathcal{L}^N_r,
\end{equation*}
with $\mathcal{L}^N_l, \mathcal{L}^N_0$ and $\mathcal{L}^N_r$ acting on functions $f:\Omega_N\to\mathbb{R}$ via
\begin{align*}
    &\mathcal{L}^N_lf(\eta) := \frac{\alpha}{N^\theta}\left[f(\eta^{1+})-f(\eta)\right]+\frac{\lambda g(\eta(1))}{N^\theta}\left[f(\eta^{1-})-f(\eta)\right],
    \\&\mathcal{L}^N_0f(\eta) := \sum_{\substack{x, y\in I_N \\ |x-y|=1}}g(\eta(x))\left[f(\eta^{x, y})-f(\eta)\right],
    \\&\mathcal{L}^N_rf(\eta) := 
    \frac{\beta}{N^\theta}\left[f(\eta^{(N-1)+})-f(\eta)\right]+\frac{\delta g(\eta(N-1))}{N^\theta}\left[f(\eta^{(N-1)-})-f(\eta)\right].
\end{align*}
To ease the notation, we only index these generators in $N$; however, note that the boundary generators $\mathcal{L}^N_l$ and $\mathcal{L}^N_r$ also depend on the parameters $\alpha, \lambda, \theta$ and $\beta, \delta, \theta$, respectively.

Here, $\eta^{x, y}$ denotes the configuration obtained from $\eta$ after a particle at $x$ jumps to $y$, namely
\begin{equation*}
    \eta^{x, y}(z):=\begin{cases} 
    \eta(z)-1 & \mbox{if } z=x,
    \\ \eta(z)+1 & \mbox{if } z=y, 
    \\\eta(z) & \mbox{if } z\ne x, y,
    \end{cases}
\end{equation*}
and $\eta^{x\pm}$ denotes the one obtained after a particle is added or removed at site $x$, namely
\begin{equation*}
    \eta^{x\pm}(z):=\begin{cases}
    \eta(z)\pm1 & \mbox{if } z=x,
    \\\eta(z) & \mbox{if } z\ne x.
    \end{cases}
\end{equation*}
The dynamics can be informally described as follows. In the bulk, particles jump from a site $x$ to a nearest-neighbour site $y$ at a rate that depends only on the occupation variable $\eta(x)$ through the function $g$ defined above, namely $g(\eta(x))$. At the boundary sites $x=1$ and $x=N-1$, particles are removed from the system at rate $g(\eta(x))$, while particles are injected at a constant rate. The parameters $\alpha, \lambda, \beta, \delta\ge 0$ represent the density of the reservoirs, while the parameter $\theta\in\mathbb{R}$ tunes the strength of the boundary dynamics; see Figure~\ref{fig:dynamics}. Throughout, we will denote by $\{\eta_t=\eta_t^N, t\ge0\}$ the continuous-time Markov process on $\Omega_N$ with generator $N^2\mathcal{L}^N$. 
%%%%%%%%%%%%%%%%%%%%%%%%%%%%%%%%%%%%%%%%%%%%%%%%%%
\begin{figure}[h!]
\begin{center}
\begin{tikzpicture}[thick, scale=0.8][h!]    
    \draw[step=1cm,gray,very thin] (-6,0) grid (6,0);   
    \foreach \y in {-6,...,6}{
    \draw[color=black,very thin] (\y,-0.2)--(\y,0.2);}    
    \draw (-6,-0.5) node [color=black] {$\scriptstyle{1}$};
    \draw (-5,-0.5) node [color=black] {$\scriptstyle{2}$};
    \draw (-4,-0.5) node [color=black] {$\scriptstyle{3}$};
    \draw (-2,-0.5) node [color=black] {$\scriptstyle{...}$};
    \draw (-1,-0.5) node [color=black] {$\scriptstyle{x-1}$};
    \draw (0,-0.5) node [color=black] {$\scriptstyle{x}$};
    \draw (1,-0.5) node [color=black] {$\scriptstyle{x+1}$};
    \draw (3,-0.5) node [color=black] {$\scriptstyle{...}$};
    \draw (5,-0.5) node [color=black] {$\scriptstyle{N\!-\!2}$};
    \draw (6,-0.5) node [color=black] {$\scriptstyle{N\!-\!1}$};
        
    \node[ball color=black!30!, shape=circle, minimum size=0.5cm] at (-6,1.2) {};
    \node[ball color=black!30!, shape=circle, minimum size=0.5cm]  at (-6.,0.4) {};
    \node[shape=circle,minimum size=0.63cm] (A) at (-6.,2) {};
    \node[shape=circle,minimum size=0.5cm] (C) at (-6.,1.3) {};
    \node[shape=circle,minimum size=0.5cm] (D) at (-4.9,1.3) {};
    \path [->] (C) edge[bend left=60] node[above right] {{\footnotesize{$g(\eta(1))$}}} (D);
    \node[shape=circle,minimum size=0.63cm] (OO) at (-7.1,1.3) {};
    \path [<-] (OO) edge[bend left=60] node[above] {{\footnotesize{$\frac{\lambda}{N^\theta}g(\eta(1))$}}} (C);	
    \node[shape=circle,minimum size=0.63cm] (OO0) at (-7.1,-0.4) {};
    \node[shape=circle,minimum size=0.63cm] (C0) at (-6,-0.4) {};		
    \path [<-] (C0) edge[bend left=60] node[below ] {{\footnotesize{$\frac{\alpha}{N^\theta}$}}} (OO0);
    
    \node[ball color=black!30!, shape=circle, minimum size=0.5cm]  at (6.,0.4) {};
    \node[ball color=black!30!, shape=circle, minimum size=0.5cm] at (6.,1.2) {};
    \node[ball color=black!30!, shape=circle, minimum size=0.5cm] at (6.,2) {};
    \node[shape=circle,minimum size=0.5cm] (E) at (6.,2.1) {};
    \node[shape=circle,minimum size=0.5cm] (F) at (4.8,2.1) {};
    \path [->] (E) edge[bend right =60] node[above left] {{\footnotesize{$g(\eta(N-1))$}}} (F);
    \node[ball color=black!30!, shape=circle, minimum size=0.5cm]  at (5.,0.4) {};
    \node[shape=circle,minimum size=0.5cm] (H) at (7.2,2.1) {};
    \path [->] (E) edge[bend left=60] node[above] {
    {{\footnotesize{$\qquad\frac{\delta }{N^\theta}g(\eta(N-1))$}}}}(H);		
    \node[shape=circle,minimum size=0.5cm] (H0) at (7.1,-0.4) {}; 
    \node[shape=circle,minimum size=0.5cm] (E0) at (6,-0.4) {};		
    \path [->] (H0) edge[bend left=60] node[below ] {{\footnotesize{$\frac{\beta}{N^\theta}$}}} (E0);
    
    \node[ball color=black!30!, shape=circle, minimum size=0.5cm] (O) at (-4,0.4) {};
    \node[ball color=black!30!, shape=circle, minimum size=0.5cm]  at (-3.,0.4) {};
    \node[ball color=black!30!, shape=circle, minimum size=0.5cm]  at (-3.,1.2) {};
    \node[ball color=black!30!, shape=circle, minimum size=0.5cm]  at (-3.,2) {};
    \node[ball color=black!30!, shape=circle, minimum size=0.5cm]  at (-1.,0.4) {};
    \node[ball color=black!30!, shape=circle, minimum size=0.5cm]  at (-1.,1.2) {};		
    \node[ball color=black!30!, shape=circle, minimum size=0.5cm]  at (0.,0.4) {};
    \node[ball color=black!30!, shape=circle, minimum size=0.5cm]  at (0.,1.2) {};
    \node[ball color=black!30!, shape=circle, minimum size=0.5cm]  at (0.,2) {};
    \node[ball color=black!30!, shape=circle, minimum size=0.5cm]  at (0.,2.8) {};		
    \node[ball color=black!30!, shape=circle, minimum size=0.5cm]  at (1.,0.4) {};
    \node[ball color=black!30!, shape=circle, minimum size=0.5cm]  at (1.,1.2) {};		
    \node[ball color=black!30!, shape=circle, minimum size=0.5cm]  at (2.,0.4) {};		
    \node[ball color=black!30!, shape=circle, minimum size=0.5cm]  at (-1,0.4) {};
        
    \node[shape=circle,minimum size=0.63cm] (G1) at (0,3) {};
    \node[shape=circle,minimum size=0.63cm] (H1) at (1,2.8) {};
    \path [->] (G1) edge[bend left =60] node[above] {{\footnotesize{$\quad g(\eta(x))$}}} (H1);		
    \node[shape=circle,minimum size=0.63cm] (J1) at (-1,2.8) {};
    \path [<-] (J1) edge[bend left =60] node[above] {{\footnotesize{$g(\eta(x))\quad$}}} (G1);
\end{tikzpicture}
\caption{Microscopic dynamics of the boundary-driven symmetric zero-range process.}\label{fig:dynamics}
\end{center}
\end{figure}
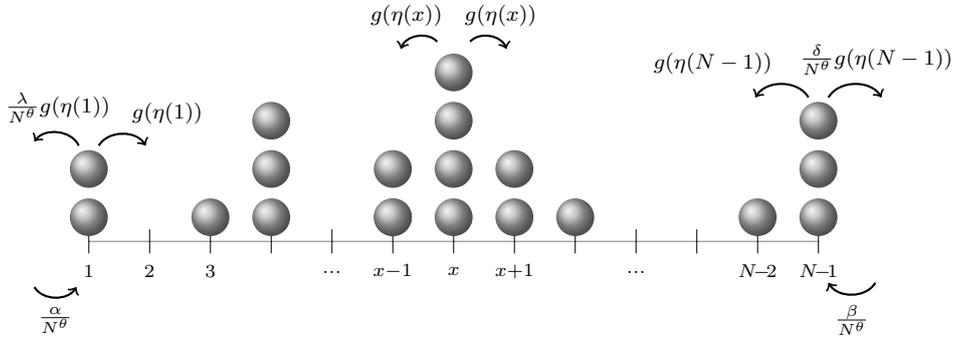
%%%%%%%%%%%%%%%%%%%%%%%%%%%%%%%%%%%%%%%%%%%%%%%%%%

%%%%%%%%%%%%%%%%%%%%%%%%%%%%%%%%%%%%%%%%%%%%%%%%%%
%%%The Invariant Measure%%%%%%%%%%%%%%%%%%%%%%%%%%
%%%%%%%%%%%%%%%%%%%%%%%%%%%%%%%%%%%%%%%%%%%%%%%%%%
\subsection{The Invariant Measure}

A remarkable property of the zero-range process is that its invariant measure can be explicitly computed even when in contact with stochastic reservoirs, and  it is of product form. In particular, let
\begin{equation}\label{eq:nu}
    \bar\nu_N\{\eta(x)=k\}:=\frac{1}{Z(\bar\varphi_N(x))}\frac{\bar\varphi_N(x)^k}{g(k)!},
\end{equation}
where $g(k)!:=\prod_{j=1}^k g(j)$, $\bar\varphi_N:I_N\to[0, \infty)$ is a map to be determined, and $Z$ is the partition function
\begin{equation}\label{eq:Z}
    Z(\varphi):=\sum_{j=0}^\infty \frac{\varphi^j}{g(j)!}.
\end{equation}
Let $\varphi^*$ denote the radius of convergence of $Z$. The following result characterises the stationary measure of $\{\eta_t^N, t\ge0\}$ and can be found in~\cite[Equation 20]{lms05} (or \cite[Lemma~2.1]{agnr25}).

\begin{lemma}\label{lemma:inv_measure} Define the map $\bar\varphi_N:I_N\to[0, \infty)$ as 
\begin{equation}
   \bar\varphi_N(x):=\frac{-(\alpha\delta-\beta\lambda)(x-1)+\alpha\delta(N-2)+(\alpha+\beta)N^\theta}{\lambda\delta(N-2)+(\lambda+\delta)N^\theta}.
\end{equation}
For $\alpha, \beta, \lambda, \delta\ge0$, $\theta\in\mathbb{R}$ and $N\in\mathbb{N}$ satisfying
\begin{equation}\label{eq:params_condition}
    \mymax\left\{\frac{\alpha\delta(N-2)+(\alpha+\beta)N^\theta}{\lambda\delta(N-2)+(\lambda+\delta)N^\theta}, \frac{\beta\lambda(N-2)+(\alpha+\beta)N^\theta}{\lambda\delta(N-2)+(\lambda+\delta)N^\theta}\right\}<\varphi^*,
\end{equation}
the measure \eqref{eq:nu} with fugacity profile $\bar\varphi_N$ is the unique invariant measure of $\{\eta_t^N, t\ge0\}$.
\end{lemma}

The condition \eqref{eq:params_condition} comes from the fact that, in order for the measure \eqref{eq:nu} to make sense, we need to have $\bar\varphi_N(x)<\varphi^*$ uniformly in $x\in I_N$ and $N\ge1$.

\begin{remark}\label{remark:limit_fugacity}
Note that, if $\alpha, \beta, \lambda, \delta>0$, the asymptotic fugacity profile $\bar\varphi_\theta:[0, 1]\to[0, \infty)$ is given by:
\begin{enumerate}[i)]
    
    \item ${\theta<1}$: $\bar\varphi_\theta(u)=\frac{-(\alpha\delta-\beta\lambda)u+\alpha\delta}{\lambda\delta}$;
    
    \item $\theta=1$: $\bar\varphi_\theta(u)=\frac{-(\alpha\delta-\beta\lambda)u+\alpha\delta+\alpha+\beta}{\lambda\delta+\lambda+\delta}$;
    
    \item $\theta>1$: $\bar\varphi_\theta(u)=\frac{\alpha+\beta}{\lambda+\delta}$;

\end{enumerate}
see Figure~\ref{fig:varphi_theta}. In fact, it is easy to see that not all parameters $\alpha, \beta, \lambda, \delta$ need to be strictly positive in order for the discrete and asymptotic fugacity profile to be well-defined; however, for simplicity, throughout we will assume $\alpha, \beta, \lambda, \delta>0$.
\end{remark}

\begin{figure}[htbp]
    \centering
    \includegraphics[width=0.5\textwidth]{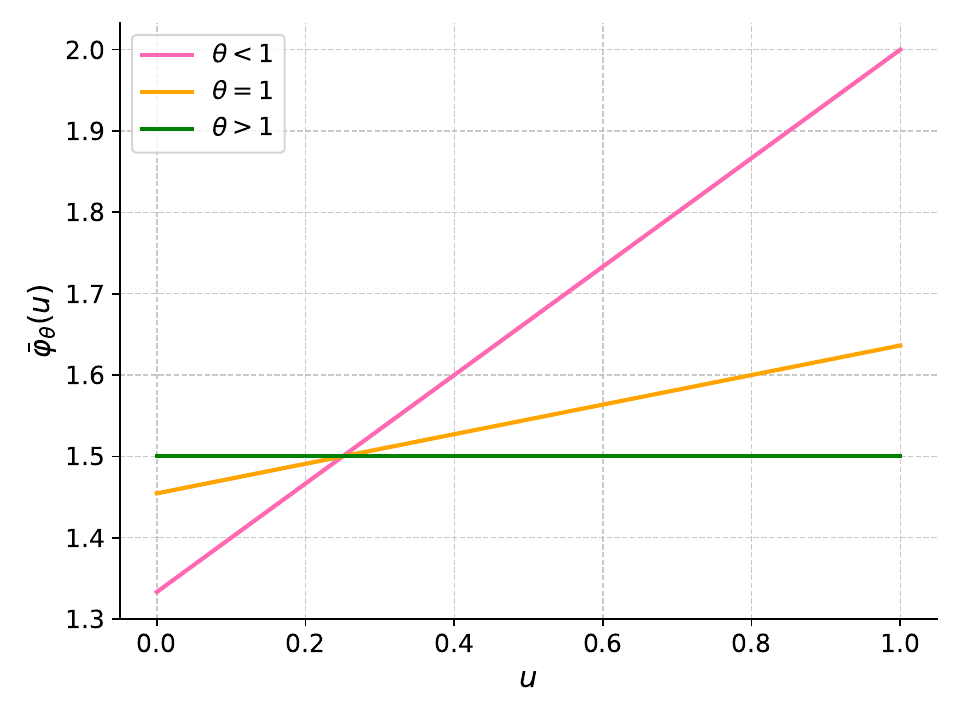}
    \caption{Plot of the asymptotic fugacity profile $\bar\varphi_\theta$ for the choice $\alpha=2, \beta=1, \lambda=\frac{3}{2}$ and $\delta=\frac{1}{2}$.}
    \label{fig:varphi_theta}
\end{figure}

Let now $R:[0,\varphi^*)\to[0, \infty)$ be the function defined as the power series
\begin{equation}\label{eq:R}
    R(\varphi ):= \frac{1}{Z(\varphi)}\sum_{j=0}^\infty \frac{j\varphi^j}{g(j)!},
\end{equation}
where $\varphi^{*}$ is again the radius of convergence of the partition function $Z$. Then, denoting by $E_{\nu}[\,\cdot\,]$ the expectation with respect to a measure $\nu$ on $\Omega_N$, we can rewrite the expected value of the occupation variable at site $x$ under $\bar{\nu}_N$ as
\begin{equation*}
    E_{\bar\nu_N}[\eta(x)] = R(\bar\varphi_N(x)), \qquad x\in I_N.
\end{equation*}

\begin{remark}\label{remark:limit_density} Under the assumptions of Remark \ref{remark:limit_fugacity}, the limit density profile $\bar\rho_\theta:[0, 1]\to[0, \infty)$ is given $\bar\rho_\theta(u)=R(\bar\varphi_\theta(u))$. In particular, a linear fugacity profile does \textit{not}, in general, imply a linear density profile. 
\end{remark}

The following result can be found in~\cite[Lemma~3.6]{agnr25}.
\begin{lemma}\label{lemma:bounded_moments}
Let $\bar{\nu}_N$ denote the invariant measure from Lemma~\ref{lemma:inv_measure}. For each positive integer $\ell$, we have that
\begin{equation*}
    \mysup_{N}\mysup_{x\in I_N}E_{\bar{\nu}_N}[(\eta(x))^\ell]<\infty.
\end{equation*}
\end{lemma}

Throughout, we will assume that $\mylim_{\varphi\uparrow\varphi*}Z(\varphi)=\infty$. Then, as highlighted in~\cite[Section~2.3]{kl99}, the range of $R$ is $[0, \infty)$, and $R$ is strictly increasing; then the map 
\begin{equation}\label{eq:Phi}
    \Phi:=R^{-1}:[0, \infty)\to[0, \varphi^*)
\end{equation} is well defined, and a strictly increasing function. Calling $\rho_N(x):=E_{\bar\nu_N}[\eta(x)]$ for $x\in I_N$, a simple computation shows that 
\begin{equation}\label{eq:Phi=varphi}
    E_{\bar\nu_N}[g(\eta(x))]=\Phi(\rho_N(x))=\bar\varphi_N(x).
\end{equation}
In fact, by~\cite[Lemma~3.3]{agnr25}, the following holds.
\begin{lemma}\label{lemma:g_moments}
For $\ell=1, 2$, we have that
\begin{equation}\label{eq:g_moments}
    \mysup_N\mysup_{x\in I_N} E_{\bar\nu_N}[g(\eta(x))^\ell]<\infty.
\end{equation}
\end{lemma}

\begin{remark}{Although we will only require \eqref{eq:g_moments} for $\ell=1$ or $2$, it is straightforward to verify that the proof of~\cite[Lemma~3.3]{agnr25} can be easily extended to any positive integer $\ell$, so that \eqref{eq:g_moments} actually holds for all $\ell$.}
\end{remark}

Finally, for each $\theta\in\mathbb{R}$, define $\chi_\theta:[0, 1]\to[0, \infty)$ via 
\begin{equation}\label{eq:chi}
    \chi_\theta(u):=\mylim_{N\to\infty} \Var_{\bar\nu_N}(\eta(Nu)),
\end{equation}
where $\Var_\nu(\cdot)$ denotes the variance under a measure $\nu$ on $\Omega_N$. Then, $\chi_\theta$ is well defined and continuous, which follows easily from Lemma~\ref{lemma:bounded_moments} and the convergence of the fugacity profile $\bar\varphi_N(\frac{\cdot}{N})$ to a continuous profile $\bar\varphi_\theta$.

%%%%%%%%%%%%%%%%%%%%%%%%%%%%%%%%%%%%%%%%%%%%%%%%%%
%%%Hydrodynamic Limit%%%%%%%%%%%%%%%%%%%%%%%%%%%%%
%%%%%%%%%%%%%%%%%%%%%%%%%%%%%%%%%%%%%%%%%%%%%%%%%%
\subsection{Hydrodynamic Limit}

The hydrodynamic limit of this model for $\theta\ge1$ was studied in~\cite{agnr25} under the following assumptions on the jump rate $g$.
\begin{enumerate}

    \item[] \textsc{Assumption 1}: $g$ is non decreasing;
    
    \item[] \textsc{Assumption 2}: denoting by $\text{gap}(j, \ell)$ the spectral gap of the generator of a symmetric zero-range process on $\{1, \ldots, \ell\}$ with $j$ particles and jump rate $g$, there exists $C=C(g)>0$ such that $\text{gap}(\ell, j)\ge \frac{C}{\ell^2(1+\frac{j}{\ell})^2}$ for all  $\ell\ge2$ and $j\ge0$. 

\end{enumerate}

We note that the hydrodynamic limit in~\cite{agnr25} requires additional assumptions on $g$. In this work, we do not assume any of the above conditions, but instead impose the purely technical assumption \eqref{assumption:g_bdd}, whose removal is left for future work.

Let $C^{m,n}([0, T]\times[0, 1])$ denote the set of continuous functions on $[0, T]\times[0, 1]$ which are $m$ times differentiable in the first variable and $n$ in the second, with continuous derivatives, where $m$ and $n$ are positive integers. Moreover, let $\mathcal{H}^1$ denote the Sobolev space $\mathcal H^1(0, 1)$, equipped with the norm
\begin{equation*}
    \|\cdot\|_{\mathcal{H}^1}^2:=\|\cdot\|^2_{L^2((0, 1))}+\|\nabla\cdot \|^2_{L^2((0, 1))},
\end{equation*}
and by $L^2([0, T], \mathcal{H}^1)$ the set of measurable functions $f:[0, T]\to\mathcal{H}^1$ with $\int_0^T\|f_t\|_{\mathcal{H}^1}^2\de t<\infty$. Let $\gamma:[0,1]\to[0, \infty)$ be a measurable and integrable function, and consider the non-linear PDE with boundary conditions 
\begin{equation}\label{eq:pde_zr}
    \left\{ 
    \begin{array}{ll}
    \partial_t \rho_t(u) = \Delta \Phi(\rho_t(u)) &\mbox{ for } u \in (0,1) \mbox{ and } t \in (0,T],
    \\\partial_u \Phi (\rho_t(0)) = \boldsymbol{1}_{\{1\}}(\theta)(\lambda \Phi(\rho_t(0))-\alpha) &\mbox{ for } t \in (0,T],
    \\\partial_u \Phi (\rho_t(1)) = \boldsymbol{1}_{\{1\}}(\theta)(\beta-\delta\Phi(\rho_t(1))) &\mbox{ for } t \in (0,T],
    \\\rho_0(u) = \gamma(u) &\mbox{ for } u \in [0,1],
    \end{array}
    \right.
\end{equation}
where the function $\Phi$ was defined in \eqref{eq:Phi}. 

\begin{definition}\label{def:weak_solutions}
 We say that $\rho:[0, T]\times [0, 1]\to[0, \infty)$ is a weak solution of the PDE \eqref{eq:pde_zr} if:
\begin{enumerate}[i)]
    
    \item $\rho \in L^2([0, T]\times [0, 1])$,
    
    \item $\Phi(\rho)\in L^2([0,T], \mathcal H^1)$,
    
    \item for any $t\in[0,T]$ and $G\in C^{1,2}([0,T]\times [0,1])$,
    \begin{equation*} 
        \begin{split}
        \langle\rho_t, G_t\rangle &- \langle\gamma, G_0\rangle - \int_0^t\{\langle \rho_s, \partial_s G_s\rangle + \langle\Phi(\rho_s), \Delta G_s\rangle\}\de s
        \\&+\int_0^t \{\Phi(\rho_s(1))\partial_u G_s(1)-\Phi(\rho_s(0))\partial_u G_s(0) \}\de s
        \\&-\boldsymbol{1}_{\{1\}}(\theta)\int_0^t \big\{(\beta-\delta\Phi(\rho_s(1)))G_s(1) - (\lambda \Phi(\rho_s(0))-\alpha)G_s(0) \big\}\de s = 0,
        \end{split}
    \end{equation*}
    where $\langle \cdot, \cdot\rangle$ denotes the inner product $\langle f, h\rangle:=\int_0^1 f(u)h(u)\de u$ on the set of measurable functions in $L^2([0,1])$.

\end{enumerate}
\end{definition}

By~\cite[Lemma~2.10]{agnr25}, there exists a unique weak solution of \eqref{eq:pde_zr} in the sense of the definition above. Given a configuration $\eta\in\Omega_N$, we associate to it the empirical measure on $[0,1]$ defined by
\begin{equation*}
    \pi^{N}_t(\de u) := \frac{1}{N} \sum_{x\in I_N} \eta_t(x) \delta_{\frac{x}{N}}(\de u),
\end{equation*}
where $\delta_u$ stands for the Dirac measure at $u\in[0,1]$. Given a measurable function $H:[0,1]\to\mathbb R$, let us denote its integral with respect to the measure $\pi_t^N$  by
\begin{equation*}
    \langle \pi_t^N, H\rangle := \frac{1}{N}\sum_{x\in I_N} \eta_t(x) H\Big(\frac xN\Big).
\end{equation*} 
For a measure $\mu^N$ on $\Omega^N$, we denote by $\prob_{\mu^N}$ the probability on the Skorokhod space of càdlàg trajectories $\mathcal{D}([0,T], \Omega_N)$ corresponding to the jump process $\{\eta_t: t\in[0, T]\}$ with initial distribution $\mu^N$. 

\begin{definition}\label{def:associated_profile}
A sequence $\{\mu^N\}_N$ of probability measures on $\Omega_N$ is said to be \textit{associated} to a measurable profile $\gamma:[0,1]\to\mathbb[0, \infty)$ if,  for any $\eps>0$ and any continuous function $H:[0,1]\to\mathbb{R}$, the following limit holds:
\begin{equation*}
    \mylim_{N\to\infty}
    \mu^N \left\{\eta\in\Omega_N: \left|\frac{1}{N} \sum_{x=1}^{N-1} H\left(\frac{x}{N}\right)\eta(x)
    - \int_0^1 H(u)\gamma(u)\de u \right| > \eps\right\} = 0.
\end{equation*}
\end{definition}

Given two measures $\mu, \nu$ on $\Omega_N$, we denote by  $H(\mu | \nu)$ their \textit{relative entropy}
\begin{equation*}
    H(\mu | \nu):=\int_{\Omega_N }\mylog\left(\frac{\mu}{\nu}\right)\text{d}\mu,
\end{equation*} 
and we say that $\mu_1\le\mu_2$ if $\int f \de \mu_1\le \int f\de \mu_2$ for all monotone\footnote{We say that a function $f:\Omega_N\to\mathbb{R}$ is \textit{monotone} if $f(\eta)\le f(\xi)$ for all $\eta\le\xi$, where this last inequality means that  $\eta(x)\le \xi(x)$ for every $x\in I_N$.} functions $f:\Omega_N\to\mathbb R$. 

\begin{theorem}\label{thm:hl}
Let $\gamma:[0,1]\to[0, \infty)$ be a  measurable function and let $\{\mu^N\}_N$ be a sequence of probability measures on $\Omega_N$ associated to $\gamma$ and satisfying the two conditions 
\begin{equation*}
    H(\mu^N|\bar\nu_N)\lesssim N \quad \text{and} \quad \mu^N\le \bar\nu_N.
\end{equation*}
Then, for $\theta\ge1$ and for every $t\in [0,T]$, the sequence of empirical measures $\{\pi^{N}_t\}_N$ converges in probability to the absolutely continuous measure $\rho(t, u)\de u$, that is, for any $\eps>0$ and for any $H\in C^2(\mathbb{T})$, we have that
\begin{equation*}
    \mylim_{N\to\infty}\prob_{\mu^N} \left\{\left|\langle \pi^{N}_{t},H\rangle-\int_0^1\rho(t,u)H(u)\de u \right|>\eps\right\} =0,
\end{equation*}
where $\rho(t,u)$ is the unique weak solution of \eqref{eq:pde_zr}.   
\end{theorem}

The hydrodynamic limit of the model for $\theta<1$ remains an open problem, but it is conjectured that it should be described by the PDE \eqref{eq:pde_zr} with Dirichlet boundary conditions given by
\begin{equation*}
    \left\{ 
    \begin{array}{ll}
    \Phi(\rho_t(0))=\frac{\alpha}{\lambda} &\mbox{ for } t \in (0,T],
    \\\Phi(\rho_t(1))=\frac{\beta}{\delta} &\mbox{ for } t \in (0,T].
    \end{array}
    \right.
\end{equation*}

%%%%%%%%%%%%%%%%%%%%%%%%%%%%%%%%%%%%%%%%%%%%%%%%%%
%%%The Space of Test Functions%%%%%%%%%%%%%%%%%%%%
%%%%%%%%%%%%%%%%%%%%%%%%%%%%%%%%%%%%%%%%%%%%%%%%%%
\subsection{The Space of Test Functions}

Throughout, for $f\in C^\infty([0, 1])$ and $k$ non-negative integer, we denote by $\partial_u^kf(0)$ the right derivative $\partial_u^k f(0^+)$, and by $\partial_u^k f(1)$ the left derivative $\partial_u^k f(1^-)$; we will also use the notation $\nabla f$ for $\partial_u f$ and $\Delta f$ for $\partial_u^2 f$. 

\begin{definition}\label{def:S_theta} Let $\mathfrak{A}_\theta$ denote the operator acting on smooth functions $f\in C^\infty([0, 1])$ via $\mathfrak{A}_\theta f(u):=\Phi'(\bar\rho_\theta(u))\Delta f(u)$, with $\Phi$ defined in \eqref{eq:Phi} and $\bar\rho_\theta$ given in Remark \ref{remark:limit_density}. We define the set $\mathcal{S}_\theta$ as follows: 
\begin{enumerate}[i)]
    
    \item $\theta<0$:
    \begin{equation*}
        \mathcal{S}_\theta:=\left\{f\in C^\infty([0, 1]): \begin{array}{ll}\partial_u^k f(0)=0, \\ \partial_u^k f(1)=0\end{array}  \text{for all integers} \ k\ge0\right\};
    \end{equation*}
    
    \item $0\le \theta<1$: 
    \begin{equation*}
        \mathcal{S}_\theta:=\left\{f\in C^\infty([0, 1]): \begin{array}{ll}\mathfrak{A}_\theta^k f(0)=0, \\ \mathfrak{A}_\theta^k f(1)=0\end{array}  \text{for all integers} \ k\ge0\right\};
    \end{equation*}
    
    \item  $\theta\ge1$: 
    \begin{equation*}
        \mathcal{S}_\theta\hspace{-.1em}:=\hspace{-.1em}\left\{f\in C^\infty([0, 1]): \hspace{-.5em}\begin{array}{ll}\partial_u(\mathfrak{A}_\theta^k f)(0)=\boldsymbol{1}_{\{1\}}(\theta)\lambda \mathfrak{A}_\theta^k f(0), \\ \partial_u(\mathfrak{A}_\theta^k f)(1)=-\boldsymbol{1}_{\{1\}}(\theta)\delta \mathfrak{A}_\theta^k f(1)\end{array}  \hspace{-.1em}\text{for all integers} \ k\ge0\right\}.
    \end{equation*}
    
\end{enumerate}
\end{definition}
For $\theta\ge0$, our space of test functions can be seen as a generalisation of the space of test functions used to study the fluctuations of the boundary-driven symmetric simple exclusion; see~\cite{fgn17, fgn19, gjmn20}. 

\begin{remark}\label{remark:Phi_smooth} Note that the map $\Phi$ is indeed differentiable. In fact, it is actually smooth, as it is the inverse of the strictly increasing, analytic function $R$ given in \eqref{eq:R}.
\end{remark}

\begin{proposition}\label{prop:frechet} For each $\theta\in\mathbb{R}$, $\mathcal{S}_\theta$ is a nuclear Fréchet space.
\end{proposition}
\begin{proof} The space $C^\infty([0, 1])$ equipped with the family of semi-norms
\begin{equation}\label{eq:seminorms}
    \|H\|_k:=\mysup_{u\in[0, 1]}\left|\partial^k_u H\right|, \quad k\in\mathbb{N}
\end{equation}
is a nuclear Fréchet space. For each $\theta\in\mathbb{R}$, let $\mathfrak{A}_\theta$ be as in Definition \ref{def:S_theta} and, for $k\in\mathbb{N}$ and $b\in\{0, 1\}$, define the map $G_{k, b}^\theta:C^\infty([0, 1])\to\mathbb{R}$ via
\begin{equation*}
    G_{k, b}^\theta(H):=\begin{cases} 
    {\partial_u^k H(b)} & \mbox{if } \theta<0,
    \\\mathfrak{A}_\theta^k H(b) & \mbox{if } 0\le\theta<1,
    \\ \partial_u(\mathfrak{A}_\theta^k H)(b)-\lambda \mathfrak{A}_\theta^k H(b) & \mbox{if } \theta=1 \mbox{ and } b=0,
    \\ \partial_u(\mathfrak{A}_\theta^k H)(b)+\delta \mathfrak{A}_\theta^k H(b) & \mbox{if } \theta=1 \mbox{ and } b=1,
    \\ \partial_u(\mathfrak{A}_\theta^k H)(b) & \mbox{if } \theta\ge1.
    \end{cases}
\end{equation*}
This map is continuous, as the point evaluation of any derivative is continuous for the topology coming from the semi-norms \eqref{eq:seminorms}. Thus, for each $\theta, k$ and $b$, the kernel of $G_{k, b}^\theta$ is a closed subspace of $C^\infty([0, 1])$. Now, the space of test functions $\mathcal{S}_\theta$ can be written as
\begin{equation*}    \mathcal{S}_\theta=\bigcap_{k\in\mathbb{N}, b\in\{0, 1\}} \text{ker}\left(G_{k, b}^\theta\right).
\end{equation*}
But then, $\mathcal{S}_\theta$ is the intersection of a countable family of closed subspaces, and thus it is a closed subspace of a nuclear Fréchet space, which makes it a nuclear Fréchet space.
\end{proof}

For each $\theta\in\mathbb{R}$, we will denote by $\mathcal{S}_\theta'$ the topological dual of $\mathcal{S}_\theta$ with respect to the topology generated by the semi-norms \eqref{eq:seminorms}, so that $\mathcal{S}_\theta'$ consists of all real-valued linear functionals on $\mathcal{S}_\theta$ which are continuous with respect to $\|\cdot\|_k$ for each $k\in\mathbb{N}$.

We also introduce the following inner product, which depends on the model parameters $\alpha, \lambda, \beta, \delta$ and $\theta\in\mathbb{R}$: given $H, G:[0, 1]\to\mathbb{R}$, let
\begin{align*}
    \scal{H, G}_{L^{2, \theta}}:=\;2\int_0^1 \Phi(\bar\rho_\theta(u))H(u)G(u)\de u
   & +\left\{\frac{\alpha}{\lambda^2}+\frac{\Phi(\bar\rho_\theta(0))}{\lambda}\right\}H(0)G(0)\boldsymbol{1}_{\{1\}}(\theta)
    \\&+\left\{\frac{\beta}{\delta^2}+\frac{\Phi(\bar\rho_\theta(1))}{\delta}\right\}H(1)G(1)\boldsymbol{1}_{\{1\}}(\theta),
\end{align*}
with corresponding norm
\begin{equation}\label{def:theta_norm}
    \| H\|_{L^{2, \theta}}^2:=\scal{H, H}_{L^{2, \theta}}.
\end{equation}
Then, $L^{2, \theta}([0, 1])$ denotes the space of functions $H:[0, 1]\to\mathbb{R}$ with $\| H\|_{L^{2, \theta}}^2<\infty$.

%%%%%%%%%%%%%%%%%%%%%%%%%%%%%%%%%%%%%%%%%%%%%%%%%%
%%%The Martingale Problem%%%%%%%%%%%%%%%%%%%%%%%%%
%%%%%%%%%%%%%%%%%%%%%%%%%%%%%%%%%%%%%%%%%%%%%%%%%%
\subsection{The Martingale Problems}\label{subsec:mg_problems}

\begin{definition}[Generalised Ornstein-Uhlenbeck Process] Let $\mathfrak{A}_\theta:\mathcal{S}_\theta\to\mathcal{S}_\theta$, $\mathfrak{B}_\theta:\mathcal{S}_\theta\to L^{2, \theta}([0, 1])$ {and let $\{\mathscr{W}_t, t\in[0, T]\}$ be an $\mathcal{S}_\theta'$-valued Brownian motion.} We say that an $\mathcal{S}_\theta'$-valued stochastic process $\{\mathscr{Y}_t, t\in[0, T]\}$ with continuous trajectories is a \textit{solution to the generalised Ornstein-Uhlenbeck equation}
\begin{equation}\label{eq:general_OU}
    \partial_t\mathscr{Y}_t=\mathfrak{A}_\theta \mathscr{Y}_t +\mathfrak{B}_\theta \mathscr{\dot W}_t
\end{equation}
if, for any $H\in\mathcal{S}_\theta$,
the processes $\{\mathscr{M}_t(H), t\in[0, T]\}$ and $\{\mathscr{N}_t(H), t\in[0, T]\}$ defined by
\begin{align}
    &\mathscr{M}_t(H):=\mathscr{Y}_t(H)-\mathscr{Y}_0(H)-\int_0^t\mathscr{Y}_s(\mathfrak{A}_\theta H)\de s,\label{eq:general_M}
    \\&\mathscr{N}_t(H):=\mathscr{M}_t(H)^2-t\|\mathfrak{B}_\theta H\|_{L^{2, \theta}}^2\nonumber
\end{align}
are $\mathcal{F}_t$-martingales, where $\mathcal{F}_t:=\sigma\{\mathscr{Y}_s(H): s\le t, H\in\mathcal{S}_\theta\}$.
\end{definition}

\begin{remark} Note that, for $\theta \ge 0$, the space of test functions $\mathcal{S}_\theta$ was chosen so that, whenever $H \in \mathcal{S}_\theta$, one also has $\mathfrak{A}_\theta H \in \mathcal{S}_\theta$. This ensures that the rightmost term in \eqref{eq:general_M} is well defined. Nonetheless, the same property holds for $\theta < 0$. Indeed, for each integer $k \ge 0$ and $u \in [0,1]$,
\begin{equation*}
    \partial_u^k (\mathfrak{A}_\theta H)(u)
    = \partial_u^k (\Phi'(\bar\rho_\theta)\,\Delta H)(u)
    = \sum_{j=0}^{k} \binom{k}{j}
    \partial_u^{k-j}\Phi'(\bar\rho_\theta(u))
    \partial_u^{j+2}H(u).
\end{equation*}
Hence, if $\partial_u^k H(0) = \partial_u^k H(1) = 0$ for all $k$, then also $\partial_u^k (\mathfrak{A}_\theta H)(0) = \partial_u^k (\mathfrak{A}_\theta H)(1) = 0$ for all $k$.
\end{remark}

We will show, {under the additional assumption \eqref{assumption:g_bdd},} that the stationary fluctuations of our model are described by \eqref{eq:general_OU} with operators
\begin{equation}\label{eq:OU_operators}
    \begin{split}&\mathfrak{A}_\theta H(u)=\Phi'(\bar\rho_\theta(u))\Delta H(u) \quad \textrm{and}\quad \mathfrak{B}_\theta H(u)=\nabla H(u).
    \end{split}
\end{equation}
Then, for $\theta\ge0$ and up to imposing a condition on the field at the initial time, solutions are unique in distribution, as guaranteed by the following result.

\begin{proposition}\label{prop:uniqueness}
For each $\theta\ge0$, there exists a unique (in law) stochastic process $\mathscr{Y}$ in the space of continuous distributions $\mathcal{C}([0, T], \mathcal{S}'_\theta)$ such that:    
\begin{enumerate}[i)]
    
    \item for every $H\in \mathcal{S}_\theta$, the stochastic processes $\{\mathscr{M}_t(H), t\in[0, T]\}$ and $\{\mathscr{N}_t(H), t\in[0, T]\}$ defined by
    \begin{equation}\label{eq:martingale_problem}
        \begin{split}
        &\mathscr{M}_t(H):=\mathscr{Y}_t(H)-\mathscr{Y}_0(H)-\int_0^t\mathscr{Y}_s\left(\Phi'(\bar\rho_\theta)\Delta H\right)\de s,
        \\&\mathscr{N}_t(H):=\mathscr{M}_t(H)^2-t\left\|\nabla H\right\|_{L^{2, \theta}}^2
        \end{split}
    \end{equation}
    are $\mathcal{F}_t$-martingales, where $\mathcal{F}_t:=\sigma\{\mathscr{Y}_s(H): s\le t, H\in\mathcal{S}_\theta\}$,

    \item $\mathscr{Y}_0$ is a mean-zero Gaussian field with covariance given by, for $H, G\in\mathcal{S}_\theta$,
    \begin{equation}\label{eq:Y_covariance}
        \expected\left[\mathscr{Y}_0(H)\mathscr{Y}_0(G)\right]=\int_0^1 H(u)G(u)\chi_\theta(u)\de u,
    \end{equation}
    with $\chi_\theta$ given in \eqref{eq:chi}.

\end{enumerate}
\end{proposition}

For $\theta<0$, the space of test functions $\mathcal{S}_\theta$ is not rich enough to ensure uniqueness of solutions to the martingale problem given in Proposition~\ref{prop:uniqueness}, so in this case we need some additional conditions. Define the norm $\|\cdot\|_{L^2}$ on the space of maps $H:[0,1]\to\mathbb{R}$ via
\begin{equation}\label{def:L2_norm}
    \|H\|_{L^2}^2:=\int_0^1 H(u)^2\de u,
\end{equation}
and let $L^2([0,1]):=\{H:[0,1]\to\mathbb{R}: \|H\|_{L^2}^2<\infty\}$. 

\begin{proposition}\label{prop:uniqueness_neg}
For each $\theta<0$, there exists a unique (in law) stochastic process $\mathscr{Y}$ in the space of continuous distributions $\mathcal{C}([0, T], \mathcal{S}'_\theta)$ which satisfies both items of Proposition~\ref{prop:uniqueness}, and which also satisfies the two following additional conditions:
\begin{enumerate}[i)]    

    \item $\expected[\mathscr{Y}_t(H)^2]\lesssim \| H\|_{L^2}^2$\, for every $H\in \mathcal{S}_\theta$,
    
    \item setting $\iota_\eps^0:=\frac{1}{\eps}\boldsymbol{1}_{(0, \eps]}$ and $\iota_\eps^1:=\frac{1}{\eps}\boldsymbol{1}_{[1-\eps, 1)}$ and defining $\mathscr{Y}(\iota_\eps^0)$ and $\mathscr{Y}(\iota_\eps^1)$ via Lemma~\ref{lemma:def_Y_eps}, for $j\in\{0, 1\}$, we have that
    \begin{equation}\label{eq:Y_limit_eps}
        \mylim_{\eps\to0}\expected\left[\mysup_{t\in[0, T]}\left(\int_0^t\mathscr{Y}_s\left({\Phi'(\bar\rho_\theta)}\iota_\eps^j\right)\de s\right)^2\right] = 0.
    \end{equation}
    
\end{enumerate}
Moreover, the law of $\mathscr{Y}$ coincides with that of the unique stationary solution of the martingale problem in Proposition~\ref{prop:uniqueness}, but where the space $\mathcal{S}_\theta$ is defined as that for $0\le\theta<1$ in Definition \ref{def:S_theta}.
\end{proposition}

The result above is reminiscent of~\cite[Theorem~2.13]{bgjs22}. Note, however, that in item ii) the field is not only tested against $\iota_\eps^0$ and $\iota_\eps^1$, but rather against these multiplied by $\Phi'(\bar\rho_\theta)$. As we will see in Section~\ref{subsec:characterisation} and Appendix~\ref{sec:app_uniqueness}, this choice is natural for our field and will allow us to extend both the derivation of this condition and the uniqueness result of~\cite{bgjs22} to our setting. The proof of Propositions~\ref{prop:uniqueness} and \ref{prop:uniqueness_neg} is postponed to Appendix~\ref{sec:app_uniqueness}.

We note that, in the Dirichlet regime, we obtain two distinct martingale problems associated with the same Ornstein-Uhlenbeck equation. When $\theta \in [0,1)$, the fluctuation field can be tested against functions satisfying boundary conditions that require all even-order derivatives to vanish at the boundary points of $[0,1]$. However, when $\theta<0$, we are no longer able to control the boundary terms arising in Dynkin's formula. To overcome this difficulty, we restrict the class of test functions to a smaller space and instead derive the Dirichlet boundary conditions directly from the particle system, by proving \eqref{eq:Y_limit_eps}. 

We further observe that, in the regime $\theta \ge 1$, it is also possible to formulate two notions of the martingale problem. In this setting, however, both notions can be defined on the \textit{larger} space of test functions $\widetilde{S} := C^\infty([0,1])$. In the next proposition, we introduce the alternative martingale problem, and then state the corresponding convergence result in Proposition~\ref{prop:fluctuations_new}.

\begin{proposition}\label{prop:uniqueness_new} For $\theta\ge1$, there exists a unique (in law) stochastic process $\mathscr{Y}$ in the space of continuous distributions $\mathcal{C}([0, T], \widetilde{S}')$ such that \eqref{eq:Y_covariance} holds for functions in $\widetilde{S}$ and, for every $H\in \widetilde{S}$, the stochastic processes $\{\mathscr{M}^*_t(H), t\in[0, T]\}$ and $\{\mathscr{N}^*_t(H), t\in[0, T]\}$ defined by
\begin{equation}\label{eq:martingale_problem_new}
    \begin{split}
    &\mathscr{M}^*_t(H):=\mathscr{Y}_t(H)-\mathscr{Y}_0(H)-\int_0^t\mathscr{Y}_s\left(\Phi'(\bar\rho_\theta)\Delta H\right)\de s-\mathscr{Z}^{\theta, 0}_t(H)-\mathscr{Z}_t^{\theta, 1}(H),
    \\&\mathscr{N}^*_t(H):=\mathscr{M}_t^*(H)^2-t\left\| H\right\|_{L^{2, \theta,*}}^2
    \end{split}
\end{equation}
are $\mathcal{F}_t$-martingales, with the ${L^{2, \theta,*}}$-norm  defined through the inner product
\begin{align*}
    \scal{H, G}_{L^{2, \theta, *}} := \;2\int_0^1 \Phi(\bar\rho_\theta(u))H'(u)G'(u)\de u&+\left\{\alpha+\lambda\Phi(\bar\rho_\theta(0))\right\}H(0)G(0)\boldsymbol{1}_{\{1\}}(\theta)
    \\&+\left\{\beta+\delta\Phi(\bar\rho_\theta(1))\right\}H(1)G(1)\boldsymbol{1}_{\{1\}}(\theta).
\end{align*}
Here, for $i\in\{0, 1\}$, the processes $\mathscr{Z}_t^{\theta, i}$ are defined by 
\begin{equation}
    \mathscr{Z}_t^{\theta, i}(H) := \mylim_{\eps\to 0}\int_0^t c_N^{\theta, i} (H,s)
    \mathscr Y_s(\Phi'(\bar\rho_\theta)\iota_\eps^0) \de s,
\end{equation} 
with $c_N^{\theta, 0}(H, s):=H_s'(0)-\boldsymbol{1}_{\{1\}}(\theta)\lambda H_s(0)$ and $c_N^{\theta, 1}(H,s):=H_s'(1)+\boldsymbol{1}_{\{1\}}(\theta)\delta H_s(1)$.
\end{proposition}

\begin{remark} We observe that, if $H \in \mathcal{S}_\theta$, then the coefficients $c_N^{\theta, i}(H, s)$ vanish identically, as they correspond to the boundary conditions appearing in the definition of $\mathcal{S}_\theta$ for $\theta\ge1$ and $k=0$. As a consequence, the martingale $\mathscr{M}^*_t(H)$ coincides with $\mathscr{M}_t(H)$. Moreover, for $H\in\mathcal{S}_\theta$, the norm $\|H\|_{L^{2,\theta,*}}$ defined above coincides with $\|\nabla H\|_{L^{2,\theta}}$. It follows that the martingale $\mathscr{N}^*_t(H)$ coincides with $\mathscr{N}_t(H)$. Finally, since $\widetilde{S} \subset \mathcal S_\theta$ for $\theta\ge1$, and since uniqueness of the martingale problem holds in $\mathcal S_\theta$, we immediately obtain uniqueness for the martingale problem defined in the previous proposition.
\end{remark}

%%%%%%%%%%%%%%%%%%%%%%%%%%%%%%%%%%%%%%%%%%%%%%%%%%
%%%Convergence of Stationary Fluctuations%%%%%%%%
%%%%%%%%%%%%%%%%%%%%%%%%%%%%%%%%%%%%%%%%%%%%%%%%%%
\subsection{Convergence of Stationary Fluctuations}

We now introduce the \textit{density fluctuation field} $\{\mathscr{Y}_t^N, t\ge0\}$ as the $\mathcal{S}'_\theta$-valued field such that, for each $H\in\mathcal{S}_\theta$,
\begin{equation}\label{eq:pre_fluct_field}
       \mathscr{Y}_t^N(H)=\frac{1}{\sqrt{N}} \sum_{x=1}^{N-1}H\left(\frac{x}{N}\right)\bar\eta_t(x),
\end{equation}
where we recall that  $\rho_N(x)=E_{\bar\nu_N}[\eta(x)]$,  and  $\bar\eta(x)$ denotes the centred random variable $\eta(x)-\rho_N(x)$.
Let $\mathcal{D}([0, T], \mathcal{S}_\theta')$ denote the space of càdlàg functions on $[0, T]$ taking values in $\mathcal{S}_\theta'$, equipped with the Skorokhod topology. Let $\mathcal{Q}^N $ denote the probability measure on $\mathcal{D}([0, T], \mathcal{S}_\theta')$ given by the law of the field $\{\mathscr{Y}_t^N, t\in [0, T]\}$. Our main result is the following.

\begin{theorem}\label{thm:fluctuations} Assume \eqref{assumption:g_bdd}, that the process $\{\eta_t, t\ge0\}$ starts from the invariant measure $\bar\nu_N$, and that the parameters $\alpha, \lambda, \beta, \delta$ are all strictly positive. Then, for each $\theta\in\mathbb{R}$, the sequence of probability measures $\{\mathcal{Q}^N\}_N$ converges weakly in $\mathcal{D}([0, T], \mathcal{S}_\theta')$, and its limit $\mathcal{Q}$ is given by the law of the unique stationary solution to:

\begin{enumerate}[i)]
    
    \item the martingale problem of Proposition~\ref{prop:uniqueness} for $\theta\ge0$; 
   
    \item the martingale problem of Proposition~\ref{prop:uniqueness_neg} for $\theta<0$.

\end{enumerate}
\end{theorem}

Consider now the case $\theta\ge1$ and take the space of test functions to be $\widetilde S$. In this setting, we can conclude that the fluctuation field converges to the unique solution of the martingale problem introduced in Proposition~\ref{prop:uniqueness_new}. This result is stated in the following proposition.
\begin{proposition}\label{prop:fluctuations_new} Under the assumption of Theorem~\ref{thm:fluctuations}, for each $\theta\ge1$, the sequence of probability measures $\{\mathcal{Q}^N\}_N$ converges weakly in $\mathcal{D}([0, T], \widetilde{S}')$, and its limit $\mathcal{Q}$ is given by the law of the unique stationary solution to the martingale problem of Proposition~\ref{prop:uniqueness_new}.
\end{proposition}

%%%%%%%%%%%%%%%%%%%%%%%%%%%%%%%%%%%%%%%%%%%%%%%%%%
%%%PROOF OF THE THE MAIN THEOREM%%%%%%%%%%%%%%%%%%
%%%%%%%%%%%%%%%%%%%%%%%%%%%%%%%%%%%%%%%%%%%%%%%%%%
\section{Proof of Theorem~\ref{thm:fluctuations}}\label{sec:proof}

Throughout, we will denote by $\mathbb{P}_{\bar\nu_N}$ the probability measure on the space of càdlàg trajectories $\mathcal{D}([0, T], \Omega_N)$ corresponding to the process $\{\eta_t, t\in[0, T]\}$ started at its stationary measure $\bar\nu_N$, and by $\expected_{\bar\nu_N}[\,\cdot\,]$ expectations with respect to $\mathbb{P}_{\bar\nu_N}$.

The proof is divided into two main steps. First, we establish tightness of the sequence of density fluctuation fields, and second, we identify the unique limiting process. Once these two points are established, convergence follows immediately. For clarity of exposition, we begin with the second step and discuss some computations involving the discrete analogue of the martingale appearing in the first line of \eqref{eq:martingale_problem}.

%%%%%%%%%%%%%%%%%%%%%%%%%%%%%%%%%%%%%%%%%%%%%%%%%%
%%%Dynkin's Martingales%%%%%%%%%%%%%%%%%%%%%%%%%%%
%%%%%%%%%%%%%%%%%%%%%%%%%%%%%%%%%%%%%%%%%%%%%%%%%%
\subsection{Dynkin's Martingales}\label{subsec:dynkin}

By Dynkin's formula~\cite[Appendix~1]{kl99}, for each $H\in \mathcal{S}_\theta$ we have that the processes
\begin{align}
    &\mathscr{M}_t^N(H):=\mathscr{Y}_t^N(H)-\mathscr{Y}_0^N(H)-\int_0^t N^2 \mathcal{L}^N\mathscr{Y}_s^N(H)\de s,\label{eq:dynkin_formula}
    \\&\mathscr{N}_t^N(H):=\mathscr{M}_t^N(H)^2-\int_0^t \left\{N^2\mathcal{L}^N\mathscr{Y}_s^N(H)^2-2N^2\mathscr{Y}_s^N(H)\mathcal{L}^N\mathscr{Y}_s^N(H)\right\}\de s\label{eq:dynkin_qv}
\end{align}
are $\mathcal{F}_t^N$-martingales, with $\mathcal{F}_t^N:=\sigma\{\eta_s^N: s\le t\}$. Let
\begin{alignat*}{2}
    &\nabla_N^+H\left(\frac{x}{N}\right):=N\left[H\left(\frac{x+1}{N}\right)-H\left(\frac{x}{N}\right)\right], &&x\in I_N\cup\{0\},
    \\&\nabla_N^-H\left(\frac{x}{N}\right):=N\left[H\left(\frac{x}{N}\right)-H\left(\frac{x-1}{N}\right)\right], && x\in I_N\cup \{N\},
    \\&\Delta_NH\left(\frac{x}{N}\right):=N^2\left[H\left(\frac{x+1}{N}\right)-2H\left(\frac{x}{N}\right)+H\left(\frac{x-1}{N}\right)\right], \quad && x\in I_N.
\end{alignat*}
 A simple computation shows that we can write
\begin{align*}
    &N^2\mathcal{L}_l^N\mathscr{Y}^N(H)=
    \frac{N^2}{N^\theta\sqrt{N}}\big(\alpha-\lambda g(\eta(1))\big)H\left(\frac{1}{N}\right),
    \\&N^2\mathcal{L}_0^N\mathscr{Y}^N(H)=\frac{1}{\sqrt{N}}\sum_{x=2}^{N-2}g(\eta(x))\Delta_N H\left(\frac{x}{N}\right)
    \\&\phantom{N^2\mathcal{L}_0^N\mathscr{Y}^N(H)=}+\frac{N}{\sqrt{N}}g(\eta(1))\nabla_N^+ H\left(\frac{1}{N}\right)-\frac{N}{\sqrt{N}}g(\eta(N-1))\nabla_N ^-H\left(\frac{N-1}{N}\right),
    \\&N^2\mathcal{L}_r^N\mathscr{Y}^N(H)=\frac{N^2}{N^\theta\sqrt{N}}\big(\beta-\delta g(\eta(N-1))\big)H\left(\frac{N-1}{N}\right).
\end{align*}
Summing and subtracting $\frac{1}{\sqrt{N}}\sum_{x=2}^{N-2}\bar\varphi_N(x)\Delta_N H\left(\frac{x}{N}\right)$, the bulk action can be rewritten as
\begin{equation}\label{eq:bulk_action}
    \begin{split}
    N^2\mathcal{L}_0^N\mathscr{Y}^N(H)=&\frac{1}{\sqrt{N}}\sum_{x=2}^{N-2}\left[g(\eta(x))-\bar\varphi_N(x)\right]\Delta_N H\left(\frac{x}{N}\right)  
    \\&+\frac{1}{\sqrt{N}}\sum_{x=2}^{N-2} \bar\varphi_N(x)\Delta_N H\left(\frac{x}{N}\right)
    \\&+\sqrt{N}g(\eta(1))\nabla_N^+H\left(\frac{1}{N}\right)
    \\&-\sqrt{N}g(\eta(N-1))\nabla_N^-H\left(\frac{N-1}{N}\right).
    \end{split}
\end{equation}
Now, note that
\begin{equation}\label{eq:expectation_corr}
    \begin{split}
    \frac{1}{\sqrt{N}}\sum_{x=2}^{N-2} \bar\varphi_N(x)\Delta_N H\left(\frac{x}{N}\right)&=\sqrt{N}\sum_{x=2}^{N-2} \bar\varphi_N(x)\left\{ \nabla_N^+H\left(\frac{x}{N}\right)-\nabla_N^+H\left(\frac{x-1}{N}\right)\right\}
    \\&=\sqrt{N}\bar\varphi_N(N-2) \nabla_N^- H\left(\frac{N-1}{N}\right) 
    \\&\phantom{=}-\sqrt{N}\bar\varphi_N(2) \nabla_N^+ H\left(\frac{1}{N}\right) 
    \\&\phantom{=}+\sqrt{N}\sum_{x=2}^{N-3} \nabla_N^+H\left(\frac{x}{N}\right)\left\{\bar\varphi_N(x)-\bar\varphi_N(x+1)\right\}.
    \end{split}
\end{equation}
We now make use of the following property of the profile $\bar\varphi_N$: for each $x=1, \ldots, N-2$, one can check that
\begin{equation}
    \bar\varphi_N(x+1)-\bar\varphi_N(x)=\frac{\lambda \bar\varphi_N(1)-\alpha}{N^\theta}=\frac{\beta-\delta\bar\varphi_N(N-1)}{N^\theta}.
\end{equation}
Hence, \eqref{eq:expectation_corr} can be rewritten as
\begin{equation}
    \begin{split}
    &\sqrt{N}\bar\varphi_N(N-1) \nabla_N^- H\left(\frac{N-1}{N}\right) +\frac{\sqrt{N}}{N^\theta}\big(\beta -\delta\bar\varphi_N(N-1)\big)\nabla_N^- H\left(\frac{N-1}{N}\right) 
    \\&-\sqrt{N}\bar\varphi_N(1) \nabla_N^+ H\left(\frac{1}{N}\right) -\frac{\sqrt{N }}{N^\theta}\big(\lambda \bar\varphi_N(1)-\alpha\big)\nabla_N^+ H\left(\frac{1}{N}\right)
    \\&+\sqrt{N}\sum_{x=2}^{N-3} \nabla_N^+H\left(\frac{x}{N}\right)\left\{\bar\varphi_N(x)-\bar\varphi_N(x+1)\right\}
    \\&=\sqrt{N}\bar\varphi_N(N-1) \nabla_N^- H\left(\frac{N-1}{N}\right) -\sqrt{N}\bar\varphi_N(1) \nabla_N^+ H\left(\frac{1}{N}\right)
    \\&\phantom{=}-\frac{N^{3/2}}{N^\theta}\big(\beta -\delta\bar\varphi_N(N-1)\big)H\left(\frac{N-1}{N}\right)+\frac{N^{3/2}}{N^\theta}\big(\lambda \bar\varphi_N(1)-\alpha\big)H\left(\frac{1}{N}\right).
    \end{split}
\end{equation}
Finally, combining everything together and recalling \eqref{eq:Phi=varphi}, we get
\begin{equation}\label{eq:dynkin_terms}
    \begin{split}N^2\mathcal{L}^N \mathscr{Y}_s^N(H)=\,&\frac{1}{\sqrt{N}}\sum_{x=2}^{N-2}\left[g(\eta_s(x))-\Phi(\rho_N(x))\right]\Delta_N H\left(\frac{x}{N}\right)   
    \\&+\sqrt{N}[g(\eta_s(1))-\bar\varphi_N(1)]\nabla_N^+H\left(\frac{1}{N}\right)
    \\&+ 
    \frac{\lambda N^{3/2}}{N^\theta}[\bar\varphi_N(1)-g(\eta_s(1))]H\left(\frac{1}{N}\right)
    \\&-\sqrt{N}[g(\eta_s(N-1))-\bar\varphi_N(N-1)]\nabla_N^-H\left(\frac{N-1}{N}\right)
    \\&+\frac{\delta N^{3/2}}{N^\theta}[\bar\varphi_N(N-1)-g(\eta_s(N-1))]H\left(\frac{N-1}{N}\right).
    \end{split}
\end{equation}
Assuming \eqref{assumption:g_bdd}, the term appearing in the first line of \eqref{eq:dynkin_terms} can be treated with the use of the Boltzmann-Gibbs principle (Theorem~\ref{thm:boltzmann_gibbs}), which allows us to rewrite it as 
\begin{equation}\label{eq:bg_term}
  \frac{1}{\sqrt{N}}\sum_{x=2}^{N-2}\Phi'(\rho_N(x))\Delta_N H\left(\frac{x}{N}\right)\left[\eta_s(x)-\rho_N(x)\right],
\end{equation}
which in turn is nothing but $\int_0^t \mathcal Y_s^N(\Phi'(\rho_N(N\cdot))\Delta_N H)$. It is easy to see that, for each $x\in I_N$, $\left|\Phi'(\rho_N(x))-\Phi'\left(\bar\rho_\theta\left(\frac{x}{N}\right)\right)\right|$ vanishes as $N\to\infty$, where the limit density $\bar\rho_\theta$ is defined in Remark \ref{remark:limit_density}. Then, we conclude that, as $N\to\infty$, \eqref{eq:bg_term} gives rise to the rightmost term in the first line of \eqref{eq:martingale_problem}. 

In order to treat the remaining terms appearing in \eqref{eq:dynkin_terms}, we need to take specific test functions as defined above. Recalling Definition \ref{def:S_theta}, we get the following:

\begin{enumerate}[i)]

    \item $\theta<0$: by performing a Taylor expansion of $H$ around $0$, for each positive integer $m$  we can write
    \begin{equation*}
        \frac{\lambda N^{3/2}}{N^\theta}[\bar\varphi_N(1)-g(\eta_s(1))]H\left(\frac{1}{N}\right)=\frac{\lambda N^{3/2}\partial_u^m H(x_0)}{N^{\theta+m}\, m!}
    \end{equation*}
    for some $x_0\in[0, \frac{1}{N}]$. Writing the bottom line of \eqref{eq:dynkin_terms} in a similar way and choosing $m>\frac{3}{2}-\theta$, it is easy to see that the third and fifth line of \eqref{eq:dynkin_terms} vanish in $L^2(\prob_{\bar\nu_N})$ as $N\to\infty$, so that that the remaining boundary terms are simply
    \begin{equation}\label{eq:theta<0}
    \begin{split}&\sqrt{N}[g(\eta_s(1))-\bar\varphi_N(1)]\nabla_N^+H\left(\frac{1}{N}\right)
    \\&-\sqrt{N}[g(\eta_s(N-1))-\bar\varphi_N(N-1)]\nabla_N^-H\left(\frac{N-1}{N}\right).
    \end{split}
    \end{equation}
    By the boundary replacement lemma (Lemma~\ref{lemma:repl_boundaries}), we see that the last display also vanishes in $L^2(\prob_{\bar\nu_N})$ as $N\to\infty$;
    
    \item $0\le \theta<1$: in this case, since $H(0)=H(1)=0$, we can rewrite the last four lines of \eqref{eq:dynkin_terms} as 
    \begin{equation}\label{eq:theta<1}
        \begin{split}&\sqrt{N}[g(\eta_s(1))-\bar\varphi_N(1)]\nabla_N^+H\left(\frac{1}{N}\right)\\&+\frac{\lambda \sqrt{N}}{N^\theta}[\bar\varphi_N(1)- g(\eta_s(1))]\nabla_N^+H(0)
        \\&-\sqrt{N}[g(\eta_s(N-1))-\bar\varphi_N(N-1)]\nabla_N^-H\left(\frac{N-1}{N}\right)
        \\&+\frac{\delta \sqrt{N}}{N^\theta}[\bar\varphi_N(N-1)-g(\eta_s(N-1))]\nabla_N^-H(1).
        \end{split}
    \end{equation}
    Now, by performing a Taylor expansion to replace discrete operators with continuous ones and by Lemma~\ref{lemma:repl_boundaries}, the last display vanishes in $L^2(\prob_{\bar\nu_N})$ as $N\to\infty$;
    \item $\theta=1$:
    note that the last four lines of \eqref{eq:dynkin_terms} can be written as 
    \begin{equation}\label{eq:theta=1}
        \begin{split}&\sqrt{N}[g(\eta_s(1))-\bar\varphi_N(1)]\left\{\nabla_N^+H\left(\frac{1}{N}\right)
        -\lambda H\left(\frac{1}{N}\right)\right\}
        \\&-\sqrt{N}[g(\eta_s(N-1))-\bar\varphi_N(N-1)]\left\{\nabla_N^-H\left(\frac{N-1}{N}\right)+\delta H\left(\frac{N-1}{N}\right)\right\}.
        \end{split}
    \end{equation}
    As above, we perform a Taylor expansion to replace discrete operators with continuous ones; then, recalling that $H'(0)=\lambda H(0)$ and $H'(1)=-\delta H(1)$, we conclude that the last display vanishes in $L^2(\prob_{\bar\nu_N})$ as $N\to\infty$;

    \item $\theta>1$: finally, to treat this case, we repeat the same arguments as above, namely we replace discrete operators by continuous ones, and then use the fact that $H'(0)=H'(1)=0$ to rewrite the last four lines of \eqref{eq:dynkin_terms} as
    \begin{equation}\label{eq:theta>1}
        \begin{split}
        &\frac{\lambda N^{3/2}}{N^\theta}[\bar\varphi_N(1)-g(\eta_s(1))]H\left(\frac{1}{N}\right)
        \\&+\frac{\delta N^{3/2}}{N^\theta}[\bar\varphi_N(N-1)-g(\eta_s(N-1))]H\left(\frac{N-1}{N}\right).
        \end{split}
    \end{equation}
    By Lemma~\ref{lemma:repl_boundaries}, we conclude that the last display vanishes in $L^2(\prob_{\bar\nu_N})$ as $N\to\infty$.

\end{enumerate}

As discussed above, in order to express the martingale $\mathscr{M}_t^N(H)$ as a function of the density field, we need two fundamental lemmas. We first present the replacement lemma in the bulk, also known as the Boltzmann-Gibbs principle, and then establish a similar result at the boundary points.

%%%%%%%%%%%%%%%%%%%%%%%%%%%%%%%%%%%%%%%%%%%%%%%%%%
%%%The Boltzmann-Gibbs Principle%%%%%%%%%%%%%%%%%%
%%%%%%%%%%%%%%%%%%%%%%%%%%%%%%%%%%%%%%%%%%%%%%%%%%
\subsection{The Boltzmann-Gibbs Principle}\label{subsec:bg}

For a positive integer $k$, let $\tau_k$ denote the shift by $k$ modulo $N$, namely $\tau_k\eta(x):=\eta([\![x+k]\!])$, where $[\![\,\cdot\,]\!]$ denotes the congruence class in $\{0, \ldots, N-1\}$ modulo $N$. We define the map $V_N:\Omega_N\to\mathbb{R}$ via
\begin{equation*}
    V_N(\eta):=g(\eta(1))-\Phi(\rho_N(1))-\Phi'(\rho_N(1))\left[\eta(1)-\rho_N(1)\right],
\end{equation*}
as well as its $k$-shifts
\begin{equation}
    \tau_kV_N(\eta):=g(\eta(k+1))-\Phi(\rho_N(k+1))-\Phi'(\rho_N(k+1))\left[\eta(k+1)-\rho_N(k+1)\right]
\end{equation}
for positive integers $k$.

\begin{theorem}[Boltzmann-Gibbs Principle]\label{thm:boltzmann_gibbs} If \eqref{assumption:g_bdd} holds, then for every continuous function $G:[0, 1]\to\mathbb{R}$ and every $t>0$,
\begin{equation}
    \mylim_{N\to\infty} \expected_{\bar\nu_N }\left[\left(\frac{1}{\sqrt{N}}\int_0^t\sum_{x=2}^{N-2}G\left(\frac{x}{N}\right)\tau_x V_N(\eta_s)\de s\right)^2\right]=0.
\end{equation}
\end{theorem}

This theorem is classical in interacting particle systems; however, our setting presents some peculiarities, as we lack translation invariance and reversibility of the invariant measure $\bar\nu_N$.

The next result is a local version of the one above and will be needed to show that, for $\theta<0$, the field satisfies item ii) of Proposition~\ref{prop:uniqueness_neg}. Note that, unlike in Theorem~\ref{thm:boltzmann_gibbs}, the test function is no longer continuous; however, the proof of the theorem above can be readily adapted to cover this case.

\begin{theorem}[Local Boltzmann-Gibbs Principle]\label{thm:local_boltzmann_gibbs} For $\eps>0$, define the map $\iota_\eps^0:[0, 1]\to[0, \infty)$ as in Proposition~\ref{prop:uniqueness_neg}, namely $\iota_\eps^0(u):=\frac{1}{\eps}\boldsymbol{1}_{(0,\eps]}(u)$. If \eqref{assumption:g_bdd} holds, then for every $t>0$,
\begin{equation*}
    \mylim_{N\to\infty} \expected_{\bar\nu_N }\left[\left(\frac{1}{\sqrt{N}}\int_0^t\sum_{x=1}^N \iota_\eps^0\left(\frac{x}{N}\right)\tau_x V_N(\eta_s)\de s\right)^2\right]=0.
\end{equation*}
The same result holds with $\iota_\eps^0$ replaced by $\iota_\eps^1$, where $\iota_\eps^1(u):=\frac{1}{\eps}\boldsymbol{1}_{[1-\eps,1)}(u)$.
\end{theorem}

Before proving these theorems, we will need some preliminary results, stated and proved below.

\subsubsection[$L^2$ Bounds and Dirichlet Forms]{$\boldsymbol{L^2}$ Bounds and Dirichlet Forms}

\begin{lemma}\label{lemma:f_eta_xy} Let \eqref{assumption:g_bdd} hold. If $f\in L^2(\bar\nu_N)$, then for each $x, y\in I_N$ we have that the map $\eta\mapsto f(\eta^{x, y})$ is also in $L^2(\bar\nu_N)$.  
\end{lemma}
\begin{proof}
We start by noting that, for all $x, y\in I_N$,
\begin{equation}\label{eq:der_radnyk}
    \frac{\bar\nu_N(\eta^{y, x})}{\bar\nu_N(\eta)}=\frac{\bar\varphi_N(x)}{\bar\varphi_N(y)}\frac{g(\eta(y))}{g(\eta(x)+1)}.
\end{equation}
Hence, we can write
\begin{align*}
    E_{\bar\nu_N}\left[f(\eta^{x, y})^2\right]&=\int_{\Omega_N} f(\eta)^2\frac{\de\bar\nu_N(\eta^{y, x})}{\de\bar\nu_N(\eta)} 
    \de\bar\nu_N(\eta)
    = \int_{\Omega_N} f(\eta)^2 \frac{\bar\varphi_N(x)}{\bar\varphi_N(y)}\frac{g(\eta(y))}{g(\eta(x)+1)}\de\bar\nu_N(\eta)
    \\& \le \frac{M_g\left\{\mysup_N\mysup_{z\in I_N}\bar\varphi_N(z)\right\}}{m_g\left\{\myinf_N \myinf_{z\in I_N}\bar\varphi_N(z)\right\}} E_{\bar\nu_N}\left[f(\eta)^2 \right],
\end{align*}
which completes the proof.
\end{proof}

Now, given a function $f\in L^2(\bar\nu_N)$, we define its \textit{Dirichlet form} with respect to the invariant measure $\bar\nu_N$ by

\begin{equation}\label{def:dirichlet_form}
\mathscr{D}_N (f):=-\scal{f, \mathcal{L}^N f}_{\bar\nu_N}.
\end{equation}
Moreover, define
\begin{align*}
    \mathscr{D}_N^0(f):=\frac{1}{2}\int_{\Omega_N}\sum_{\substack{x, y\in I_N \\ |x-y|=1}}g(\eta(x))\left[f(\eta^{x, y})-f(\eta)\right]^2 \de \bar\nu_N 
\end{align*}
and, for $\theta\in\mathbb{R}$,
\begin{align*}
    &\mathscr{D}_N^l(f):=\frac{1}{2}\int_{\Omega_N}\left\{\frac{\alpha}{N^\theta}\left[f(\eta^{1+})-f(\eta)\right]^2+\frac{\lambda g(\eta(1))}{N^\theta}\left[f(\eta^{1-})-f(\eta)\right]^2\right\}\de \bar\nu_N ,
    \\&\mathscr{D}_N^r(f):=\frac{1}{2}\int_{\Omega_N}\bigg\{\frac{\beta}{N^\theta}\left[f(\eta^{(N-1)+})-f(\eta)\right]^2
    \\&\phantom{\mathscr{D}_N^r(f):=\frac{1}{2}\int_{\Omega_N}\bigg\{}+\frac{\delta g(\eta(N-1))}{N^\theta}\left[f(\eta^{(N-1)-})-f(\eta)\right]^2\bigg\}\de \bar\nu_N .
\end{align*}

\begin{lemma}\label{lemma:dirichlet} For each $\theta\in\mathbb{R}$ and any function $f\in L^2(\bar\nu_N)$, we have that
\begin{equation*}
    \mathscr{D}_N (f)=\mathscr{D}_N ^l(f)+\mathscr{D}_N^0(f)+\mathscr{D}_N ^r(f).
\end{equation*}
\end{lemma}

\begin{proof}
We start by writing
\begin{equation}\label{eq:sum_generators}
    \mathscr{D}_N(f)=-\int_{\Omega_N}f(\eta)\left\{\mathcal{L}^N_l f(\eta)+\mathcal{L}^N_0f(\eta)+\mathcal{L}^N_r f(\eta)\right\}\de \bar\nu_N .
\end{equation}
Now, note that
\begin{align*}
    &-\int_{\Omega_N}f(\eta)\mathcal{L}^N _lf(\eta)\de \bar\nu_N 
    \\&=-\int_{\Omega_N}f(\eta)\left\{\frac{\alpha}{N^\theta}\left[f(\eta^{1+})-f(\eta)\right]+\frac{\lambda g(\eta(1))}{N^\theta}\left[f(\eta^{1-})-f(\eta)\right]\right\}\de \bar\nu_N 
    \\&=-\int_{\Omega_N}\frac{\alpha}{N^\theta}\frac{\left[f(\eta)-f(\eta^{1+})\right]}{2}\left[f(\eta^{1+})-f(\eta)\right]\de \bar\nu_N 
    \\&\phantom{=}-\int_{\Omega_N}\frac{\alpha}{N^\theta}\frac{\left[f(\eta)+f(\eta^{1+})\right]}{2}\left[f(\eta^{1+})-f(\eta)\right]\de \bar\nu_N 
    \\&\phantom{=}-\int_{\Omega_N}\frac{\lambda g(\eta(1))}{N^\theta}\frac{\left[f(\eta)-f(\eta^{1-})\right]}{2}\left[f(\eta^{1-})-f(\eta)\right]\de \bar\nu_N 
    \\&\phantom{=}-\int_{\Omega_N}\frac{\lambda g(\eta^1)}{N^\theta}\frac{\left[f(\eta)+f(\eta^{1-})\right]}{2}\left[f(\eta^{1-})-f(\eta)\right]\de \bar\nu_N 
    \\&=\mathscr{D}_N ^l(f)+\frac{1}{2}\int_{\Omega_N}\mathcal{L}^N _l f^2(\eta)\de \bar\nu_N .
\end{align*}
Hence, by performing a similar computation for the terms with $\mathcal{L}^N_0$ and $\mathcal{L}^N_r$ in \eqref{eq:sum_generators},
we get that
\begin{align*}
    \mathscr{D}_N (f)&=\mathscr{D}_N ^l(f)+\mathscr{D}_N^0(f)+\mathscr{D}_N ^r(f)+\frac{1}{2}\int_{\Omega_N}\left\{\mathcal{L}^N _l+\mathcal{L}^N_0+\mathcal{L}^N_r\right\}f^2(\eta)\de {\bar\nu_N }
    \\&=\mathscr{D}_N^l(f)+\mathscr{D}_N^0(f)+\mathscr{D}_N ^r(f)+\frac{1}{2}\int_{\Omega_N}\mathcal{L}^N f^2(\eta)\de {\bar\nu_N }.
\end{align*}
But now, since $\bar\nu_N $ is invariant for the dynamics, we have that $\int_{\Omega_N}\mathcal{L}^N h(\eta)\de {\bar\nu_N }=0$ for every $h:\Omega_N\to\mathbb{R}$, so we conclude.
\end{proof}

\subsubsection{Proof of Theorems~\ref{thm:boltzmann_gibbs} and \ref{thm:local_boltzmann_gibbs}}

\begin{proof}[Proof of Theorem~\ref{thm:boltzmann_gibbs}]
The proof is inspired by the argument of~\cite{br84} (see also~\cite[Theorem~11.1.1]{kl99}), but requires substantial modifications due to the non-reversible and non-translation-invariant setting. 

\vspace{1em} \textbf{Reduction to Finite Boxes and Local Generators.} Given a positive integer $K$, we split the interval $I_N^\circ:=I_N\setminus \{1, N-1\}$ into $M$ consecutive boxes $\{B_i\}_{i=1}^M$ of size $2K+1$ (with $M$ necessarily of order $\frac{N}{K}$). For each $i$, we denote by $B_i^\circ$ the set of points in $B_i$ which are at distance at least $2$ from the complement of $B_i$ in $I_N^\circ$, and we set $B^c:=\cup_{i=1}^M (B_i\setminus B_i^\circ)$. Calling $y_i$ the midpoint of the box $B_i$, we write
\begin{align*}
    \frac{1}{\sqrt{N}}\sum_{x=2}^{N-2} G\left(\frac{x}{N}\right)\tau_x V_N(\eta_s)=&\,\frac{1}{\sqrt{N}}\sum_{x\in B^c}G\left(\frac{x}{N}\right)\tau_x V_N(\eta_s)
    \\&+\frac{1}{\sqrt{N}}\sum_{i=1}^M\sum_{x\in B_i^\circ}\tau_x V_N(\eta_s)\left[G\left(\frac{x}{N}\right)-G\left(\frac{y_i}{N}\right)\right]
    \\&+\frac{1}{\sqrt{N}}\sum_{i=1}^M\sum_{x\in B_i^\circ}\tau_x V_N(\eta_s)G\left(\frac{y_i}{N}\right).
\end{align*}
The first two terms vanish by the same arguments given in the proof of~\cite[Theorem~11.1.1]{kl99} as these arguments do not rely on reversibility nor translation invariance, so we only consider the last term. 

For each $1\le i\le M$, let $\mathcal{L}_{B_i}$ denote the bulk generator $\mathcal{L}^N_0$ restricted to the box $B_i$, namely the operator acting on maps $f:\Omega_N\to\mathbb{R}$ via
\begin{equation*}
    \mathcal{L} _{B_i}f(\eta):=\sum_{\substack{x, y\in B_i\\ |x-y|=1}}g(\eta(x))[f(\eta^{x, y})-f(\eta)].
\end{equation*}
We also introduce the operators $\mathcal{\widetilde L}_{B_i}^N$ and $\mathcal{\widehat L}_{B_i}^{N}$ acting on maps $f:\Omega_N\to\mathbb{R}$ via
\begin{equation}\label{eq:tilde_L}
    \mathcal{\widetilde L}_{B_i}^{N} f(\eta):=\sum_{\substack{x, y\in B_i\\ |x-y|=1}}g(\eta(x))[f(\eta^{x, y})-f(\eta)]\left[1+\frac{\frac{\bar\varphi_N(y)}{\bar\varphi_N(x)}-1}{2}\right]
\end{equation}
and
\begin{equation}\label{eq:hat_L}
    \mathcal{\widehat L}_{B_i}^{N} f(\eta):=\sum_{\substack{x, y\in B_i\\ |x-y|=1}}g(\eta(x))[f(\eta^{x, y})-f(\eta)]\left[\frac{1-\frac{\bar\varphi_N(y)}{\bar\varphi_N(x)}}{2}\right],
\end{equation}
respectively, so that in particular $\mathcal{L} _{B_i}=\mathcal{\widetilde L}_{B_i}^N+\mathcal{\widehat L}_{B_i}^N$. Finally, we define the local quadratic forms
\begin{equation*}
    \Gamma_N^{B_i} (f):=\int_{\Omega_N}\sum_{\substack{x, y\in B_i\\ |x-y|=1}}g(\eta(x))[f(\eta^{x, y})-f(\eta)]^2\de\bar\nu_N, \qquad i=1, \ldots, M.
\end{equation*}

\textbf{Measures with Constant Fugacity and Weighted Translations.} Recalling the definition of $\varphi^*$ as the radius of convergence of the partition function $Z$ in \eqref{eq:Z}, for $\varphi\in(0, \varphi^*)$, we denote by $\bar\nu_{\varphi}$ the product measure given by \eqref{eq:nu} with constant parameter $\varphi$, namely
\begin{equation*}\label{eq:nu_constant}
    \bar\nu_{\varphi}\{\eta(x)=k\}:=\frac{1}{Z(\varphi)}\frac{\varphi^k}{g(k)!}.
\end{equation*}
Recall the definition of $y_i$ as the midpoint of the box $B_i$, and let $\bar\varphi_i:=\mylim_{N\to\infty}\bar\varphi_\theta(\frac{y_i}{N})$, where $\bar\varphi_\theta$ is the limit fugacity profile described in Remark \ref{remark:limit_fugacity}. We highlight that, for $\theta>1$, the profile $\bar\varphi_\theta$ is constant, so the $\bar\varphi_i$'s are in fact all the same. Moreover, let
\begin{equation}\label{eq:eps_K}
    \eps_K\in\left(0, \mymin\left\{\frac{1}{K}, \varphi^*-\mysup_{u\in[0, 1]}\bar\varphi_\theta(u)\right\}\right),
\end{equation}
so that in particular $\mysup_{i=1, \ldots, M}\bar\varphi_i+\eps_K<\varphi^*$. For each $\varphi\in(0, \varphi^*)$, let $\pi_{\varphi, \eps_K}$ denote the operator acting on functions $f:\Omega_N\to\mathbb{R}$ which are measurable with respect to (the sigma algebra generated by) $\xi:=(\eta(1), \ldots, \eta(2K+1))$ via
\begin{equation}\label{eq:pi}
    \pi_{\varphi, \eps_K} f(\xi):=f(\xi)\sqrt{\frac{\de \bar\nu_{\bar\varphi_1+\eps_K}(\xi)}{\de \bar\nu_{\varphi+\eps_K}(\xi)}}, \qquad \xi\in\mathbb{N}^{2K+1}.
\end{equation}
It is immediate to see that $f\in L^2(\bar\nu_{\bar\varphi_1+\eps_K})$ if and only if $\pi_{\varphi, \eps_K}f\in L^2(\bar\nu_{\varphi+\eps_K})$, as
\begin{equation}\label{eq:L2_iff}
    E_{\bar\nu_{\varphi+\eps_K}}\left[\big(\pi_{\varphi, \eps_K}f(\xi)\big)^2\right]=\int_{\Omega_N} f(\xi)^2\frac{\de \bar\nu_{\bar\varphi_1+\eps_K}(\xi)}{\de \bar\nu_{\varphi+\eps_K}(\xi)}\de\bar\nu_{\varphi+\eps_K}(\xi)=E_{\bar\nu_{\bar\varphi_1+\eps_K}}\left[f(\xi)^2\right].
\end{equation}

For each $i=1, \ldots, M$, we will denote by $\xi_i$ the restriction of the configuration $\eta$ to the box $B_i$, namely $\xi_i:=(\eta(x))_{x\in B_i}$. Given a function $f_1:\Omega_N\to\mathbb{R}$ in $L^2(\bar\nu_{\bar\varphi_1+\eps_K})$ and measurable with respect to $\xi_1$, we then denote by $f_i$ the translation of $\pi_{\bar\varphi_i, \eps_K}f_1$ which makes it measurable with respect to $\xi_i$. 

\vspace{1em}
\textbf{Subtracting the Local Generators.} Recalling that $\mathcal{L} _{B_i}=\mathcal{\widetilde L}_{B_i}^N+\mathcal{\widehat L}_{B_i}^N$, by a standard convexity inequality, we see that 
\begin{align}
    &\expected_{\bar\nu_N}\left[\left(\frac{1}{\sqrt{N}}\int_0^t \sum_{i=1}^M G\left(\frac{y_i}{N}\right)\mathcal{L}_{B_i}^Nf_i(\xi_i(s))\de s\right)^2\right]\nonumber
    \\&\lesssim \expected_{\bar\nu_N}\left[\left(\frac{1}{\sqrt{N}}\int_0^t \sum_{i=1}^M G\left(\frac{y_i}{N}\right)\mathcal{\widetilde L}_{B_i}^Nf_i(\xi_i(s))\de s\right)^2\right]\label{eq:bg_kv}
    \\&\phantom{\lesssim} + \expected_{\bar\nu_N}\left[\left(\frac{1}{\sqrt{N}}\int_0^t \sum_{i=1}^M G\left(\frac{y_i}{N}\right)\mathcal{\widehat L}_{B_i}^Nf_i(\xi_i(s))\de s\right)^2\right].\label{eq:bg_cs}
\end{align}
We start by bounding $\eqref{eq:bg_kv}$: by the Kipnis-Varadhan inequality, 
\begin{equation}
    \eqref{eq:bg_kv}\lesssim t \mysup_{h\in L^2(\bar\nu_N)}\left\{\frac{2}{\sqrt{N}}\sum_{i=1}^M G\left(\frac{y_i}{N}\right)\int_{\Omega_N} \mathcal{\widetilde L}_{B_i}^Nf_i(\xi_i)h(\eta)\de \bar\nu_N-N^2\mathscr{D}_N (h)\right\}, \label{eq:supremum_boxes}
\end{equation}
where $\mathscr{D}_N$ was defined in \eqref{def:dirichlet_form}. For each $i$, we write
\begin{align}
    &2\int_{\Omega_N} \mathcal{\widetilde L}_{B_i}^Nf_i(\xi_i)h(\eta)\de \bar\nu_N
    = \int_{\Omega_N}\sum_{\substack{x, y\in B_i\\ |x-y|=1}}g(\eta(x))[f_i(\eta^{x, y})-f_i(\eta)]\left[h(\eta)-h(\eta^{x, y})\right]\de\bar\nu_N\nonumber
    \\&\phantom{=}+\int_{\Omega_N}\sum_{\substack{x, y\in B_i\\ |x-y|=1}}g(\eta(x))[f_i(\eta^{x, y})-f_i(\eta)]\left[h(\eta)+h(\eta^{x, y})\right]\de\bar\nu_N\label{eq:bg_correction_term_a}
    \\& \phantom{=}+\int_{\Omega_N} \sum_{\substack{x, y\in B_i\\ |x-y|=1}}g(\eta(x))[f_i(\eta^{x, y})-f_i(\eta)]h(\eta)\left[\frac{\bar\varphi_N(y)}{\bar\varphi_N(x)}-1\right]\de\bar\nu_N.\label{eq:bg_correction_term_b}
\end{align}
Now, performing the change of variable $\eta\mapsto\eta^{y, x}$ in the summand involving $h(\eta^{x, y})$ and recalling \eqref{eq:der_radnyk}, we see that  \eqref{eq:bg_correction_term_a} can be rewritten as
\begin{align*}
    \eqref{eq:bg_correction_term_a}&=\int_{\Omega_N}\sum_{\substack{x, y\in B_i\\ |x-y|=1}}g(\eta(x))[f_i(\eta^{x, y})-f_i(\eta)]h(\eta)\de\bar\nu_N
    \\&\phantom{=}+\int_{\Omega_N}\sum_{\substack{x, y\in B_i\\ |x-y|=1}}g(\eta(x)+1)[f_i(\eta)-f_i(\eta^{y, x})]h(\eta)\frac{\bar\varphi_N(x)}{\bar\varphi_N(y)}\frac{g(\eta(y))}{g(\eta(x)+1)}\de\bar\nu_N
    \\&=\int_{\Omega_N} \sum_{\substack{x, y\in B_i\\ |x-y|=1}}g(\eta(x))[f_i(\eta^{x, y})-f_i(\eta)]h(\eta)\left[1-\frac{\bar\varphi_N(y)}{\bar\varphi_N(x)}\right]\de \bar\nu_N,
\end{align*}
which yields $\eqref{eq:bg_correction_term_a}+\eqref{eq:bg_correction_term_b}=0$. We can then apply Young's inequality to see that, for every $1\le i\le M$ and $A_i>0$,
\begin{equation*}
    \int_{\Omega_N} \mathcal{\widetilde L}_{B_i}^N f_i(\xi_i)h(\eta)\de \bar\nu_N \le \frac{A_i}{2}\Gamma_N^{B_i}(f_i)+\frac{1}{2A_i}\Gamma_N^{B_i}(h).
\end{equation*}
Choosing $A_i=\frac{1}{N^2\sqrt{N}}\left|G\left(\frac{y_i}{N}\right)\right|$ and using the fact that $\mathscr{D}_N (h)\ge\sum_{i=1}^M\Gamma_N^{B_i}(h)$ (see Lemma~\ref{lemma:dirichlet} above), we see that
\begin{equation*}
    \eqref{eq:supremum_boxes} \le \frac{t}{N^3}\sum_{i=1}^M G\left(\frac{y_i}{N}\right)^2\Gamma_N^{B_i}(f_i)\le \frac{t\|G\|_\infty^2}{N^2 K}\mysup_{i=1, \ldots, M} \Gamma_N ^{B_i}(f_i),
\end{equation*}
But now, since $\bar\varphi_N(x)\le\bar\varphi_i+\eps_K$ for each $x\in B_i$ and for $N$ sufficiently large, we note that
\begin{align*}
    E_{\bar\nu_N}\left[f_i(\xi_i)^2\right]&=\int_{\Omega_N} f_i(\xi_i)^2 \frac{\de \bar\nu_N(\xi_i)}{\de \bar\nu_{\bar\varphi_i+\eps_K}(\xi_i)}\de \bar\nu_{\bar\varphi_i+\eps_K}(\xi_i)
    \\&=\sum_{\eta\in\Omega_N} f_i(\xi)^2 \left\{\prod_{x\in B_i}\frac{Z(\bar\varphi_i+\eps_K)}{Z(\bar\varphi_N(x))}\left(\frac{\bar\varphi_N(x)}{\bar\varphi_i+\eps_K}\right)^{\eta(x)}\right\}\bar\nu_{\bar\varphi_i+\eps_K}(\xi_i)
    \\&\lesssim E_{\bar\nu_{\bar\varphi_i+\eps_K}}\left[f_i(\xi_i)^2\right]\lesssim 1.
\end{align*}
Thus, by \eqref{assumption:g_bdd} and Lemma~\ref{lemma:f_eta_xy}, \eqref{eq:supremum_boxes} vanishes as $N\to\infty$. 

As for \eqref{eq:bg_cs}, by the stationarity of $\bar\nu_N$ and the Cauchy-Schwarz inequality,
\begin{equation}\label{eq:hat_L_bound}
    \begin{split} \eqref{eq:bg_cs}&\lesssim \frac{t^2}{N}\sum_{i=1}^M G\left(\frac{y_i}{N}\right)^2E_{\bar\nu_N}\left[\left(\mathcal{\widehat L}_{B_i}^N f_i(\xi_i)\right)^2\right]
    \\&\phantom{\lesssim}+\frac{t^2}{N}\sum_{\substack{i, j=1\\ i\ne j}}^M G\left(\frac{y_i}{N}\right)G\left(\frac{y_j}{N}\right)E_{\bar\nu_N}\left[\mathcal{\widehat L}_{B_i}^N f_i(\xi_i)\right]E_{\bar\nu_N}\left[\mathcal{\widehat L}_{B_j}^N f_j(\xi_j)\right].
    \end{split}
\end{equation}
But now, for each $i=1, \ldots, M$, by \eqref{assumption:g_bdd} and a standard convexity inequality, we have that
\begin{align*}
    E_{\bar\nu_N}\left[\left(\mathcal{\widehat L}_{B_i}^N f_i(\xi_i)\right)^2\right]&=E_{\bar\nu_N}\left[\left(\frac{1}{2}\sum_{\substack{x, y\in B_i\\ |x-y|=1}}g(\eta(x))\left[f_i(\eta^{x, y})-f_i(\eta)\right]\left[1-\frac{\bar\varphi_N(y)}{\bar\varphi_N(x)}\right]\right)^2\right]
    \\&\lesssim \frac{M_g^2}{N^2}\left\{K^2 E_{\bar\nu_N}\left[f_i(\xi_i)^2\right]+K \sum_{\substack{x, y\in B_i\\ |x-y|=1}} E_{\bar\nu_N}\left[f_i(\xi_i^{x, y})^2\right]\right\}\lesssim \frac{K^2}{N^2},
\end{align*}
where we used $\left|\frac{\bar\varphi_N(x\pm 1)}{\bar\varphi_N(x)}-1\right|\lesssim \frac{1}{N}$ and Lemma~\ref{lemma:f_eta_xy}. Similarly, for each $i=1, \ldots, M$, 
\begin{equation*}
    \left|E_{\bar\nu_N}\left[\mathcal{\widetilde L}_{B_i}^{N}f_i(\xi_i)\right]\right|\lesssim\frac{M_g}{N}  \left\{ K E_{\bar\nu_N}\left[\left|f_i(\xi_i)\right|\right]+\sum_{\substack{x, y\in B_i\\ |x-y|=1}}E_{\bar\nu_N}\left[\left|f_i(\xi_i^{x, y})\right|\right]\right\}\lesssim \frac{K}{N}.
\end{equation*}
This yields
\begin{equation*}
    \eqref{eq:hat_L_bound}\lesssim \frac{t^2\|G\|_\infty^2}{N}\cdot \frac{N}{K}\cdot\frac{K^2}{N^2}+\frac{t^2\|G\|_\infty^2}{N}\cdot \left(\frac{N}{K}\right)^2\cdot\left(\frac{K}{N}\right)^2, 
\end{equation*}
which vanishes as $N\to\infty$. 

Combining everything together, we have thus proved that 
\begin{equation}
    \myinf_{f_1}\mylim_{N\to\infty}\expected_{\bar\nu_N }\left[\left(\frac{1}{\sqrt{N}}\int_0^t \sum_{i=1}^M G\left(\frac{y_i}{N}\right)\mathcal{ L}_{B_i}f_i(\xi_i(s))\de s\right)^2\right]=0,
\end{equation}
where the infimum runs along all functions $f_1\in L^2(\bar\nu_{\bar\varphi_1+\eps_K})$ which are measurable with respect to $\xi_1$, and $f_i$ is defined as below \eqref{eq:pi}. 

\vspace{1em}
\textbf{Removing Crossed Terms.}
So far, we have reduced the proof to showing that
\begin{equation}\label{eq:remaining}
    \begin{split}\mylim_{K\to\infty}\myinf_{f_1}\mylim_{N\to\infty}\expected_{\bar\nu_N }\Bigg[&\Bigg(\frac{1}{\sqrt{N}}\int_0^t \sum_{i=1}^M G\left(\frac{y_i}{N}\right)\times
    \\&\times\Bigg(\sum_{x\in B_i^\circ}\tau_x V_N(\eta_s)-\mathcal{L}_{B_i}f_i(\xi_i(s))\Bigg)\de s\Bigg)^2\Bigg]=0,
    \end{split}
\end{equation}
where the infimum runs along all functions $f_1\in L^2(\bar\nu_{\bar\varphi_1+\eps_K})$ which are measurable with respect to $\xi_1$, and $f_i$ is defined as below \eqref{eq:pi}. Note that the supports of $\tau_x V_N-\mathcal{L}_{B_i}f_i$ and $\tau_y V_N-\mathcal{L}_{B_j}f_j$ are disjoint for $x\in B_i, y\in B_j$ when $i\ne j$. Hence, by the stationarity of $\bar\nu_N$, the Cauchy-Schwarz inequality and the fact that $E_{\bar\nu_N}[\tau_x V_N(\eta)]=0$ for all $x$, the expectation in \eqref{eq:remaining} is bounded from above by
\begin{align}
    &\frac{t^2}{N}E_{\bar\nu_N }\left[\left(\sum_{i=1}^M G\left(\frac{y_i}{N}\right)\left(\sum_{x\in B_i^\circ}\tau_x V_N(\eta)-\mathcal{L}_{B_i}f_i(\xi_i)\right)\right)^2\right]\nonumber
    \\&= \frac{t^2}{N}\sum_{i=1}^M G\left(\frac{y_i}{N}\right)^2E_{\bar\nu_N }\left[\left(\sum_{x\in B_i^\circ}\tau_x V_N(\eta)-\mathcal{L}_{B_i}f_i(\xi_i)\right)^2\right]\label{eq:i}
    \\&\phantom{=}+\frac{t^2}{N}\sum_{\substack{i, j=1\\ i\ne j}}^MG\left(\frac{y_i}{N}\right)G\left(\frac{y_j}{N}\right)E_{\bar\nu_N }\left[\mathcal{L}_{B_i}f_i(\xi_i)\right]E_{\bar\nu_N }\left[\mathcal{L}_{B_j}f_j(\xi_j)\right].\label{eq:ij}
\end{align}
Note that the crossed terms \eqref{eq:ij} do not automatically vanish as the measure $\bar\nu_N $ is not, in general, reversible for the dynamics. But now, using \eqref{eq:der_radnyk}, we can write
\begin{align*}
    E_{\bar\nu_N }\left[\mathcal{L}_{B_i}f_i(\xi_i)\right]&=\int_{\Omega_N}\sum_{\substack{x, y\in B_i \\ |x-y|=1}}g(\xi_i(x))\left[f_i\big(\xi_i^{x, y}\big)-f_i(\xi_i)\right]\de \bar\nu_N 
    \\&=\int_{\Omega_N}\sum_{\substack{x, y\in B_i \\ |x-y|=1}}f_i(\xi_i)\left[g(\xi_i(y))\frac{\bar\varphi_N(x)}{\bar\varphi_N(y)}-g(\xi_i(x))\right]\de \bar\nu_N 
    \\&=\sum_{\substack{x, y\in B_i \\ |x-y|=1}}E_{\bar\nu_N }[f_i(\xi_i)g(\xi_i(x))]\left[\frac{\bar\varphi_N(y)}{\bar\varphi_N(x)}-1\right].
\end{align*}
Thus, since $\left|\frac{\bar\varphi_N(x\pm 1)}{\bar\varphi_N(x)}-1\right|\lesssim \frac{1}{N}$, we get
\begin{align}
    \eqref{eq:ij}&\le \frac{t^2\|G\|_\infty^2}{N}\sum_{\substack{i=1\\ i\ne j}}^ME_{\bar\nu_N }\left[\mathcal{L}_{B_i}f_i(\xi_i)\right]E_{\bar\nu_N }\left[\mathcal{L}_{B_j}f_j(\xi_j)\right]\nonumber
    \\&\lesssim \frac{t^2\|G\|_\infty^2}{N^3}\sum_{\substack{i=1\\ i\ne j}}^M\sum_{x\in B_i}E_{\bar\nu_N }[f_i(\xi_i)g(\xi_i(x))]\sum_{y\in B_j}E_{\bar\nu_N }[f_j(\xi_j)g(\xi_j(y))].\label{eq:crossed}
\end{align}
But then, since $E_{\bar\nu_N }[f_i(\xi_i)g(\xi_i(x))]\le \frac{1}{2}E_{\bar\nu_N }[f_i(\xi_i)^2]+\frac{1}{2}E_{\bar\nu_N }[g(\xi_i(x))^2]$, 
\begin{align*}
    \eqref{eq:crossed}\lesssim \frac{t^2\|G\|_\infty^2}{N^3}\left(\frac{N}{K}\right)^2K^2 \left\{\mysup_{i=1,\ldots, M}E_{\bar\nu_N }\left[f_i(\xi_i)^2\right]+\mysup_{x\in I_N}E_{\bar\nu_N }\left[g(\eta(x))^2\right]\right\}^2
\end{align*}
which vanishes as $N\to\infty$ by Lemma~\ref{lemma:g_moments} and because $E_{\bar\nu_N}[f_i(\xi_i)^2]\lesssim 1$ as proved above, so we are only left to control \eqref{eq:i}.

\vspace{1em}
\textbf{Comparison to the Measures with Constant Fugacity.} Unlike the case in~\cite[Theorem~11.1.1]{kl99}, our random variables $\{\eta(x)\}_{x\in I_N}$ are independent but \textit{not} identically distributed. To get around this problem, following a similar approach as that in the proof of~\cite[Proposition~18]{lab18}, note first that \eqref{eq:i} satisfies
\begin{equation*}
    \eqref{eq:i}\le  \frac{t^2\|G\|_\infty^2}{N}\sum_{i=1}^M E_{\bar\nu_N }\left[\left(\sum_{x\in B_i^\circ}\tau_x V_N(\eta)-\mathcal{L}_{B_i}f_i(\xi_i)\right)^2\right].
\end{equation*}
Again since $\bar\varphi_N(x)\le\bar\varphi_i+\eps_K$ for all $x\in B_i$ and $N$ sufficiently large, we see that 
\begin{equation}\label{eq:lim_F^N_pre}
    \begin{split}&E_{\bar\nu_N }\left[\left(\sum_{x\in B_i^\circ}\tau_x V_N(\eta)-\mathcal{L}_{B_i}f_i(\xi_i)\right)^2\right]
    \\&\lesssim E_{\bar\nu_{ \bar\varphi_i+\eps_K}}\left[\left(\sum_{x\in B_i^\circ}\tau_x V_N(\eta)-\mathcal{L}_{B_i}f_i(\xi_i)\right)^2\right]=:F_K^{N, \eps_K}(i).
    \end{split}
\end{equation}
But now, recalling the definition of the function $R$ in \eqref{eq:R}, for each $\varphi\in(0, \varphi^*)$ let
\begin{equation*}
    V^{\varphi}(\eta):=g(\eta(1))-\varphi-\Phi'(R(\varphi))\left[\eta(1)-R(\varphi)\right].
\end{equation*}
Note first that
\begin{align}
    &E_{\bar\nu_{\bar\varphi_i+\eps_K}}\left[\left(\sum_{x\in B_i^\circ}\big\{\tau_x V_N(\eta)-\tau_x V^{\bar\varphi_i}(\eta)\big\}\right)^2\right]\nonumber
    \\&\lesssim \left(\sum_{x\in B_i^\circ}\big\{\bar\varphi_N(x)-\bar\varphi_i\big\}\right)^2    \label{eq:term_1}
    \\&\phantom{\le}+E_{\bar\nu_{\bar\varphi_i+\eps_K}}\left[\left(\sum_{x\in B_i^\circ}\bigg\{\Phi'(R(\bar\varphi_N(x)))-\Phi'(R(\bar\varphi_i))\bigg\}\eta(1)\right)^2\right]   \label{eq:term_2}
    \\&\phantom{\le}+\left(\sum_{x\in B_i^\circ}\bigg\{\Phi'(R(\bar\varphi_N(x)))R(\bar\varphi_N(x))-\Phi'(R(\bar\varphi_i))R(\bar\varphi_i)\bigg\}\right)^2.\label{eq:term_3}
\end{align}
Now, $E_{\bar\nu_\varphi}[\eta(x)]\lesssim 1$ for any $\varphi\in(0, \varphi^*)$ any $x\in I_N$. Moreover, $|\bar\varphi_N(x)-\bar\varphi_i|\lesssim N^{-|1-\theta|-\boldsymbol{1}_{\{1\}}(\theta)}$, and both $\Phi'$ and $R$ are smooth functions, so that
\begin{equation*}
    \eqref{eq:term_1}, \eqref{eq:term_2}, \eqref{eq:term_3}\lesssim \frac{K^2}{N^{|1-\theta|+\boldsymbol{1}_{\{1\}}(\theta)}}.
\end{equation*}
But then, setting 
\begin{equation*}
    F_K^{\eps_K}(i):=E_{\bar\nu_{ \bar\varphi_i+\eps_K}}\left[\left(\sum_{x\in B_i^\circ}\tau_x V^{\bar\varphi_i}(\eta)-\mathcal{L}_{B_i}f_i(\xi_i)\right)^2\right],
\end{equation*}
we have that
\begin{align*}
    &\left|F_K^{N, \eps_K}(i)-F_K^{\eps_K}(i)\right|
    \\&=\Bigg|E_{\bar\nu_{ \bar\varphi_i+\eps_K}}\Bigg[\left(\sum_{x\in B_i^\circ}\bigg\{\tau_x V_N(\eta)+\tau_x V^{\bar\varphi_i}(\eta)\bigg\}-2\mathcal{L}_{B_i}f_i(\xi_i)\right)\times
    \\&\phantom{\Bigg|E_{\bar\nu_{ \bar\varphi_i+\eps_K}}\Bigg[}\times\left(\sum_{x\in B_i^\circ}\bigg\{\tau_x V_N(\eta)-\tau_x V^{\bar\varphi_i}(\eta)\bigg\}\right)\Bigg]\Bigg|
    \lesssim \frac{K^4}{N^{|1-\theta|+\boldsymbol{1}_{\{1\}}(\theta)}}.
\end{align*}
Hence,
\begin{equation}\label{eq:lim_F^N}
    \mylim_{N\to\infty} F_K^{N, \eps_K}(i) = F_K^{\eps_K}(i).
\end{equation}
Together with \eqref{eq:lim_F^N_pre},
recalling the definition of $\pi_{\varphi, \eps_K}$ in \eqref{eq:pi}, this implies that, for each fixed $f_1$,
\begin{equation*}
    \mylim_{N\to\infty}\eqref{eq:i}\lesssim  \frac{1}{K}\mysup_{\varphi\in[\bar\varphi_a,\bar\varphi_b]}  E_{\bar\nu_{\varphi+\eps_K}}\left[\left(\sum_{x\in B_1^\circ}\tau_x V^{\varphi}(\eta)-\mathcal{L}_{B_1}\pi_{\varphi, \eps_K}f_1(\eta)\right)^2\right],
\end{equation*}
where $\bar\varphi_a:=\myinf_{u\in[0, 1]}\bar\varphi_\theta(u)$ and $\bar\varphi_b:=\mysup_{u\in[0, 1]}\bar\varphi_\theta(u)$. Now, it is not hard to check that the map
\begin{equation*}
    [\bar\varphi_a,\bar\varphi_b]\ni \varphi\mapsto E_{\bar\nu_{\varphi+\eps_K}}\left[\left(\sum_{x\in B_1^\circ}\tau_x V^{\varphi}(\eta)-\mathcal{L}_{B_1}\pi_{\varphi, \eps_K}f_1(\eta)\right)^2\right]\in[0, \infty)
\end{equation*}
is continuous, so by compactness of $[\bar\varphi^a,\bar\varphi^b]$, the supremum is achieved for some $\varphi$: we have thus reduced the proof to showing that
\begin{equation*}
    \mylim_{K\to\infty}\myinf_{f_1}\frac{1}{K}E_{\bar\nu_{\varphi+\eps_K}}\left[\left(\sum_{x\in B_1^\circ}\tau_x V^{\varphi}(\eta)-\mathcal{L}_{B_1}\pi_{\varphi, \eps_K} f_1(\eta)\right)^2\right]=0,
\end{equation*}
where the infimum runs along functions $f_1\in L^2(\bar\nu_{\bar\varphi_1+\eps_K})$ which are measurable with respect to $\xi_1$. But now, as seen in \eqref{eq:L2_iff}, $f_1\in L^2(\bar\nu_{\bar\varphi_1+\eps_K})$ if and only if $\pi_{\varphi, \eps_K} f_1\in L^2(\bar\nu_{\varphi+\eps_K})$, so the above is equivalent to showing that
\begin{equation}\label{eq:infimum_goal}
    \mylim_{K\to\infty}\myinf_{f}\frac{1}{K}E_{\bar\nu_{\varphi+\eps_K}}\left[\left(\sum_{x\in B_1^\circ}\tau_x V^{\varphi}(\eta)-\mathcal{L}_{B_1} f(\eta)\right)^2\right]=0,
\end{equation}
where the infimum now runs along functions $f\in L^2(\bar\nu_{\varphi+\eps_K})$ which are measurable with respect to $\xi_1$.

\vspace{1em}
\textbf{$\boldsymbol{L^2}$ Projection and Conditional Expectation.} Following the exact same argument given in the proof of~\cite[Theorem~11.1.1]{kl99}, the infimum in \eqref{eq:infimum_goal} is equal to
\begin{equation}\label{eq:cond_expect}
    \frac{t^2}{K}\|G\|_2^2\; E_{\bar\nu_{\varphi+\eps_K}}\left[\left\{E_{\bar\nu_{\varphi+\eps_K}}\left[\sum_{x\in B_1^\circ} \tau_x V^\varphi(\eta) \;\bigg| \;\bar\eta^{B_1}\right]\right\}^2\right]
\end{equation}
where $\bar\eta^{B_1}:=\frac{1}{2K+1}\sum_{x\in B_1}\eta(x)$ denotes the average number of particles in the box $B_1$. Given $x\in B_1^\circ$, let $\tilde \Phi(\bar\eta^{B_1})$ denote the conditional expectation of $g(\eta(x))$ given $\bar\eta^{B_1}$, namely
\begin{equation*}
    \tilde \Phi(\bar\eta^{B_1}):=E_{\bar\nu_{\varphi+\eps_K}}\left[g(\eta(x))\; \big |\; \bar\eta^{B_1}\right],
\end{equation*}
which does not depend on $x$ since the measure $\bar\nu_{\varphi+\eps_K}$ is homogeneous. With this notation, we can rewrite \eqref{eq:cond_expect} as
\begin{equation}\label{eq:cond_expression}
    \begin{split}&\frac{t^2 (2K+1)^2}{K}\|G\|_2^2  E_{\bar\nu_{\varphi+\eps_K}}\left[\left\{\tilde \Phi(\bar\eta^{B_1})-\varphi -\Phi'(R(\varphi))\left[\bar\eta^{B_1}-R(\varphi)\right]\right\}^2 \right].
    \end{split}
\end{equation}
But now, 
\begin{align}
    &\eqref{eq:cond_expression}\lesssim K E_{\bar\nu_{\varphi+\eps_K}}\left[\left\{\tilde\Phi(\bar\eta^{B_1})-\Phi(\bar\eta^{B_1})\right\}^2\right]\label{eq:equivalence_term}
    \\&\phantom{\lesssim}+K  E_{\bar\nu_{\varphi+\eps_K}}\left[\left\{\Phi(\bar\eta^{B_1})-(\varphi+\eps_K)-\Phi'(R(\varphi+\eps_K))[\bar\eta^{B_1}-R(\varphi+\eps_K)]\right\}^2\right]\label{eq:taylor_term}
    \\&\phantom{\lesssim}+K \eps_K^2\label{eq:eps_K_term1}
    \\&\phantom{\lesssim}+K\left\{\Phi'(R(\varphi))-\Phi'(R(\varphi+\eps_K))\right\}^2E_{\bar\nu_{\varphi+\eps_K}}[(\bar\eta^{B_1})^2]\label{eq:eps_K_term2}
    \\&\phantom{\lesssim}+K \left\{(\Phi'(R(\varphi)) R(\varphi)-\Phi'(R(\varphi+\eps_K))R(\varphi+\eps_K))\right\}^2.\label{eq:eps_K_term3}
\end{align}

Given $A>0$, we write the expectation in \eqref{eq:equivalence_term} as the sum of the same expectation but with the term inside multiplied by $\boldsymbol{1}_{(A, \infty)}(\bar\eta^{B_1})$ and $\boldsymbol{1}_{[0, A]}(\bar\eta^{B_1})$. By \eqref{eq:function_g}, it is easy to see that $\Phi(\rho)\le g^*\rho$ and $\tilde\Phi(\rho)\le g^*\rho$, so that 
\begin{align}
    &KE_{\bar\nu_{\varphi+\eps_K}}\left[\left\{\tilde\Phi(\bar\eta^{B_1})-\Phi(\bar\eta^{B_1})\right\}^2\boldsymbol{1}_{(A, \infty)}(\bar\eta^{B_1})\right]\label{eq:large_dev}
    \\&\lesssim  KE_{\bar\nu_{\varphi+\eps_K}}\left[(\bar\eta^{B_1})^2\boldsymbol{1}_{(A, \infty)}(\bar\eta^{B_1})\right]\nonumber
    \\&\le KE_{\bar\nu_{\varphi+\eps_K}}\left[(\bar\eta^{B_1})^4\right]^\frac{1}{2} E_{\bar\nu_{\varphi+\eps_K}}\left[\boldsymbol{1}_{(A, \infty)}(\bar\eta^{B_1})\right]^\frac{1}{2}.\nonumber
\end{align}
Choosing $A$ sufficiently large, by the large deviations principle for independent random variables, $E_{\bar\nu_{\varphi+\eps_K}}\left[\boldsymbol{1}_{(A, \infty)}(\bar\eta^{B_1})\right]$ decays exponentially in $K$, so \eqref{eq:large_dev} vanishes as $K\to\infty$.  On the other hand, by the equivalence of ensembles~\cite[Corollary~A2.1.7]{kl99}, we have that
\begin{equation*}
    \left|\tilde\Phi(\bar\eta^{B_1})-\Phi(\bar\eta^{B_1})\right|\boldsymbol{1}_{[0, A]}(\bar\eta^{B_1})\lesssim  \frac{1}{K},
\end{equation*}
which yields $\mylim_{K\to\infty}\eqref{eq:equivalence_term}=0$.

By performing a Taylor expansion up to the second order, we see that
\begin{equation*}
    \eqref{eq:taylor_term}\lesssim K E_{\bar\nu_{\varphi+\eps_K}}\left[\left\{\bar\eta^{B_1}-R(\varphi+\eps_K)\right\}^4\right]\lesssim \frac{1}{K^2},
\end{equation*}
so \eqref{eq:taylor_term} also vanishes as $K\to\infty$.

Finally, since $\Phi'$ and $R$ are smooth and recalling that $\eps_K\lesssim \frac{1}{K}$ (see \eqref{eq:eps_K}), we immediately get that \eqref{eq:eps_K_term1}, \eqref{eq:eps_K_term2} and \eqref{eq:eps_K_term3} all vanish as $K\to\infty$. This completes the proof.
\end{proof}

\begin{proof}[Proof of Theorem~\ref{thm:local_boltzmann_gibbs}]
We will only show the result for $\iota_\eps^0$, as the proof for the case $\iota_\eps^1$ is completely analogous. Throughout, even if $\eps N$ is not an integer, we will keep the notation $\eps N$ to denote its integer part $\lfloor \eps N\rfloor$. We start by noting that the integral term appearing in the statement can be rewritten as
\begin{equation*}
    \frac{1}{\eps \sqrt{N}}\int_0^t\sum_{x=1}^{\eps N} \tau_x V_N(\eta_s)\de s.
\end{equation*}
We now proceed in a similar way to the proof of Theorem~\ref{thm:boltzmann_gibbs}: given a positive integer $K$, we split the box $ I_N^\eps:=\{1,\ldots, \eps N\}$ into $M$ consecutive boxes $\{B_i\}_{i=1}^M$ of size $K$ (with $M$ necessarily of order $\frac{\eps N}{K}$). For each $i$, we denote by $B_i^\circ$ the set of points in $B_i$ which are at distance at least $r+1$ from the complement of $B_i$, and we set $B^c:=\cup_{i=1}^M (B_i\setminus B_i^\circ)$. We write
\begin{equation*}
    \frac{1}{\eps \sqrt{N}}\sum_{x=1}^{\eps N} \tau_x V_N(\eta_s)=\frac{1}{\eps \sqrt{N}}\sum_{x\in B^c} \tau_x V_N(\eta_s)+\frac{1}{\eps \sqrt{N}}\sum_{i=1}^M\sum_{x\in B_i^\circ}\tau_x V_N(\eta_s).
\end{equation*}
From this point, the proof is entirely analogous to that of Theorem~\ref{thm:boltzmann_gibbs}. The reader is invited to retrace the steps therein, and will observe that, since $\eps>0$ is kept fixed, no additional argument is required.
\end{proof}

%%%%%%%%%%%%%%%%%%%%%%%%%%%%%%%%%%%%%%%%%%%%%%%%%%
%%%Replacement Lemma at the Boundary%%%%%%%%%%%%%%
%%%%%%%%%%%%%%%%%%%%%%%%%%%%%%%%%%%%%%%%%%%%%%%%%%
\subsection{Replacement Lemma at the Boundary}\label{subsec:replacement}

\begin{lemma}[Boundary Replacement Lemma]\label{lemma:repl_boundaries} For each $\theta\in\mathbb{R}$, we have that
\begin{align}               
    &\expected_{\bar\nu_N}\left[\mysup_{t\in[0, T]}\left(\int_0^t \left\{g(\eta_s(1))-\bar\varphi_N(1)\right\}\de s\right)^2\right]\lesssim  N^{\theta-2},\label{eq:repl_left}
    \\&\expected_{\bar\nu_N}\left[\mysup_{t\in[0, T]}\left(\int_0^t \left\{g(\eta_s(N-1))-\bar\varphi_N(N-1)\right\}\de s\right)^2\right]\lesssim  N^{\theta-2}.\nonumber
\end{align}
\end{lemma}

\begin{proof}
We show the first bound, as the second one is entirely analogous. By the Kipnis-Varadhan inequality~\cite[Lemma~2.4]{klo12}, the first line of \eqref{eq:repl_left} is bounded from above by 
\begin{equation*}
    CT\mysup_{f\in L^2(\bar\nu_N)}\left\{2\int_{\Omega_N}\left\{ g(\eta(1))-\bar\varphi_N(1)\right\}f(\eta)\de \bar\nu_N -N^2\mathscr{D}_N(f)\right\}
\end{equation*}
for some positive constant $C$ independent of $N$, with $\mathscr{D}_N(f)$ defined in \eqref{def:dirichlet_form}. We write
\begin{align}
    2\int_{\Omega_N} \left\{g(\eta(1))-\bar\varphi_N(1) \right\}f(\eta)\de \bar\nu_N=&\int_{\Omega_N} g(\eta(1))\left[f(\eta)-f(\eta^{1-})\right]\de \bar\nu_N\label{eq:-A}
    \\&+\int_{\Omega_N} g(\eta(1))\left[f(\eta)+f(\eta^{1-})\right]\de \bar\nu_N\label{eq:+A}
    \\&-\int_{\Omega_N}\bar\varphi_N(1)\left[f(\eta)-f(\eta^{1+})\right]\de \bar\nu_N\label{eq:-B}
    \\&-\int_{\Omega_N}\bar\varphi_N(1)\left[f(\eta)+f(\eta^{1+})\right]\de \bar\nu_N.\label{eq:+B}
\end{align}
By applying Young's inequality $xy\le \frac{Ax^2}{2}+\frac{y^2}{2A}$ to both  and \eqref{eq:-A} and \eqref{eq:-B} with $A=2\lambda N^{2-\theta}$ and  $A=\frac{2\alpha N^{2-\theta}}{\bar\varphi_N(1)}$, respectively, we see that
\begin{align*}
    \eqref{eq:-B}+\eqref{eq:-A} &\le N^2\int_{\Omega_N}\left\{\frac{\lambda g(\eta(1))}{N^\theta}\left[f(\eta^{1-})-f(\eta)\right]^2+\frac{\alpha}{N^\theta}\left[f(\eta^{1+})-f(\eta)\right]^2\right\}\de \bar\nu_N
    \\&\phantom{\le}+\frac{N^{\theta-2}}{4}\int_{\Omega_N}\left\{\frac{ g(\eta(1))}{\lambda}+\frac{\bar\varphi_N(1)^2}{\alpha}\right\}\de \bar\nu_N
    \\& \le N^2 \mathscr{D}_N(f)+C N^{\theta-2}
\end{align*}
for some constant $C$ independent of $N$, where we used $\mathscr{D}_N(f)\ge \mathscr{D}_N^l(f)$ and $E_{\bar\nu_N}[g(\eta(1))]=\bar\varphi_N(1)\lesssim 1$. 

We now make the change of variable  $\eta^{1-}\mapsto\eta$ in \eqref{eq:+A}, so as to get 
\begin{align*}
    \eqref{eq:+A}&=\int_{\Omega_N}g(\eta(1))f(\eta)\left\{1+ \frac{g(\eta(1)+1)\bar\nu_N(\eta^{1+})}{g(\eta(1))\bar\nu_N(\eta)}\right\}\de \bar\nu_N
    \\&=\int_{\Omega_N}f(\eta)\left[ g(\eta(1))+\bar\varphi_N(1)\right]\de \bar\nu_N,
\end{align*}
where we used $\frac{\bar\nu_N(\eta^{1+})}{\bar\nu_N(\eta)}=\frac{\bar\varphi_N(1)}{g(\eta(1)+1)}$. Similarly, making the change of variable 
$\eta^{1+}\mapsto \eta$ and noting that $\frac{\bar\nu_N(\eta^{1-})}{\bar\nu_N(\eta)}=\frac{g(\eta(1))}{\bar\varphi_N(1)}$, we get
\begin{equation*}
    \eqref{eq:+B}=-\int_{\Omega_N} f(\eta)\left\{\bar\varphi_N(1)+g(\eta(1))\right\}\de \bar\nu_N.
\end{equation*}
But then, 
\begin{equation*}
    \eqref{eq:+B}+\eqref{eq:+A}=\int_{\Omega_N}f(\eta)\left\{\big(g(\eta(1))+\bar\varphi_N(1)\big)-\big(\bar\varphi_N(1)+g(\eta(1))\big)\right\}\de \bar\nu_N=0,
\end{equation*}
which completes the proof.\end{proof}

%%%%%%%%%%%%%%%%%%%%%%%%%%%%%%%%%%%%%%%%%%%%%%%%%%
%%%Tightness%%%%%%%%%%%%%%%%%%%%%%%%%%%%%%%%%%%%%%
%%%%%%%%%%%%%%%%%%%%%%%%%%%%%%%%%%%%%%%%%%%%%%%%%%
\subsection{Tightness}\label{subsec:tightness}

In the section, we show the following result.
\begin{proposition}\label{prop:tightness} For each $\theta\in\mathbb{R}$, the sequence of measures $\{\mathcal{Q}^N\}_N$ is tight in $\mathcal{D}([0, T], \mathcal{S}_\theta')$ and all limit points are concentrated on continuous paths.
\end{proposition}

We start by recalling two well-known tightness criteria. Let $D([0, T], \mathbb{R})$ denote the space of real-valued càdlàg functions on $[0, T]$.

\begin{proposition}[Mitoma's Criterion~\cite{mit83}]\label{prop:mitoma} Let $\mathcal{S}$ be a nuclear Fréchet space. A sequence of $\mathcal{S}'$-valued stochastic processes $\{\mathcal{X}^N_t, t\in[0, T]\}_N$ taking values in $\mathcal{D}([0, T], \mathcal{S}')$ is tight with respect to the Skorokhod topology if and only if, for any $H\in\mathcal{S}$, the sequence of real-valued processes $\{\mathcal{X}^N_t(H), t\in[0, T]\}_N$ is tight with respect to the Skorokhod topology of $D([0, T], \mathbb{R})$. 
\end{proposition}

\begin{proposition}[Aldous's Criterion~\cite{ald78}]\label{prop:aldous} A sequence of real-valued processes $\{X_t^N,$ $t\in[0, T]\}_N$ is tight with respect to the Skorokhod topology of $D([0, T], \mathbb{R})$ if:
\begin{enumerate}[i)]
    
    \item the sequence of real-valued random variables $\{X_t^N\}_N$ is tight for every $t\in[0, T]$, and
    
    \item for any $\eps>0$,
\begin{equation*}\mylim_{\delta\to0}\mylimsup_{N\to\infty}\mysup_{\gamma\le\delta}\mysup_{\tau\in\mathcal{I}_T}\mathbb{P}\left\{\left|X^N_{(\tau+\gamma)\wedge T}-X^N_\tau\right|>\eps\right\}=0,
    \end{equation*}
    where $\mathcal{I}_T$ is the set of stopping times almost surely bounded from above by $T$.

\end{enumerate}
\end{proposition}

Note that, in order to apply Mitoma's criterion to the sequence $\{\mathscr{Y}_t^N, t\in[0, T]\}_N$, we need the space of test functions $\mathcal{S}_\theta$ to be a nuclear Fréchet space, but this is guaranteed by Proposition~\ref{prop:frechet}.

The proof of Proposition~\ref{prop:tightness} will follow from the sequence of lemmas stated and proved below.

\subsubsection{The Initial Condition}

\begin{lemma}\label{lemma:initial_time} For each $\theta\in\mathbb{R}$, the sequence of fields $\{\mathscr{Y}_0^N\}_N$ converges in distribution to a mean-zero Gaussian field $\mathscr{Y}_0$ with covariance given by, for each $H, G\in\mathcal{S}_\theta$, $\expected[\mathscr{Y}_0(H)\mathscr{Y}_0(G)]$ $=\int_0^1 H(u)G(u)\chi_\theta(u)\de u$, where $\chi_\theta$ is defined in \eqref{eq:chi}.
\end{lemma}

\begin{proof}
For each positive integer $N$ and any $H\in\mathcal{S}_\theta$, define the random variables $\{X_x^N\}_{x\in I_N}$ via $X_x^N:=H\left(\frac{x}{N}\right)\bar\eta(x)$, so that $\mathscr{Y}_0^N(H)=\frac{1}{\sqrt{N}}\sum_{x=1}^N X_x^N$, and let
\begin{equation*}
    S_N^2:=\sum_{x=1}^N \Var_{\bar\nu_N }(X_x^N)=\sum_{x=1}^N H\left(\frac{x}{N}\right)^2\Var_{\bar\nu_N}(\bar\eta(x)).
\end{equation*}
It is easy to see that $\mylim_{N\to\infty}\frac{1}{N}S_N^2=\int_0^1 H(u)^2\chi_\theta(u)\de u$. Hence, by Lemma~\ref{lemma:bounded_moments}, the Lyapunov condition (see~\cite[Section~27]{bil95}) is verified with $\delta=1$, namely
\begin{align*}
    \frac{1}{S_N^3}\sum_{x=1}^N E_{\bar\nu_N }\left[|X_x^N|^3\right]= \frac{1}{\left(\frac{S_N^2}{N}\right)^{3/2} N^{3/2}}\sum_{x=1}^N \left|H\left(\frac{x}{N}\right)\right|^3E_{\bar\nu_N }\left[|\bar\eta(x)|^3\right]\lesssim \frac{1}{\sqrt{N}},
\end{align*}
which vanishes as $N\to\infty$. Thus, by the Lyapunov Central Limit Theorem~\cite[Theorem~27.3]{bil95}, we have that
\begin{align*}
    \frac{1}{S_N}\sum_{x=1}^NX_x^N\to\mathcal{N}(0, 1)
\end{align*}
in distribution as $N\to\infty$. By the convergence of  $\frac{S_N}{\sqrt{N}}$, this implies that
\begin{equation*}
    \mathscr{Y}_0^N(H)\to\mathcal{N}\left(0, \int_0^1 H(u)^2\chi_\theta(u)\de u\right)
\end{equation*}
in distribution as $N\to\infty$. The expression of the covariance follows easily by noting that, given $H, G\in\mathcal{S}_\theta$
\begin{align*}
    \mylim_{N\to\infty}\expected_{\bar\nu_N }\left[\mathscr{Y}_0^N(H)\mathscr{Y}_0^N(G)\right]&=\mylim_{N\to\infty}\frac{1}{N}\sum_{x=1}^NH\left(\frac{x}{N}\right)G\left(\frac{x}{N}\right)\Var_{\bar\nu_N }(\eta(x))
    \\&=\int_0^1 H(u)G(u)\chi_\theta(u)\de u,
\end{align*}
which completes the proof.
\end{proof}

\subsubsection{The Integral Term}

\begin{lemma}\label{lemma:integral_tightness} For each $\theta\in\mathbb{R}$, the sequence of $\mathcal{S}'_\theta$-valued processes $\{\int_0^tN^2\mathcal{L}^N \mathscr{Y}_s^{N}\de s, t\in[0, T]\}_N$ is tight with respect to the Skorokhod topology of $\mathcal{D}([0, T], \mathcal{S}'_\theta)$. 
\end{lemma}

\begin{proof}
By Mitoma's criterion, it suffices to prove that, given $H\in\mathcal{S}_\theta$, the sequence of real-valued processes $\{\int_0^tN^2\mathcal{L}^N \mathscr{Y}_s^{N}(H)\de s, t\in[0, T]\}_N$ is tight with respect to the uniform topology of $D([0, T], \mathbb{R})$. We will apply Aldous's criterion (Proposition~\ref{prop:aldous}): 
\begin{enumerate}[i)]

    \item $0\le\theta<1$: recalling \eqref{eq:bg_term} and \eqref{eq:theta<1}, first note that, for any $0\le r\le t\le T$,
    \begin{align}
        &\expected_{\bar\nu_N}\left[\left(\int_r^t N^2\mathcal{L}^N \mathscr{Y}_s^{N}(H)\de s\right)^2\right]\nonumber
        \\&\lesssim \expected_{\bar\nu_N}\left[\left(\int_r^t \frac{1}{\sqrt{N}}\sum_{x=2}^{N-2}[g(\eta_s(x))-\Phi(\rho_N(x))]\Delta_N H\left(\frac{x}{N}\right)\de s\right)^2\right]\label{eq:bulk}
        \\&\phantom{\le}+\expected_{\bar\nu_N}\left[\left(\int_r^t \sqrt{N}[g(\eta_s(1))-\bar\varphi_N(1)]\nabla_N^+H\left(\frac{1}{N}\right)\de s\right)^2\right] \label{eq:left1}
        \\&\phantom{\le}+\expected_{\bar\nu_N}\left[\left(\int_r^t \frac{\lambda\sqrt{N}}{N^\theta}[\bar\varphi_N(1)- g(\eta_s(1))]\nabla_N^+H(0)\de s\right)^2\right]\label{eq:left2}
        \\&\phantom{\le}+\expected_{\bar\nu_N}\left[\left(\int_r^t \sqrt{N}[g(\eta_s(N-1))-\bar\varphi_N(N-1)]\nabla_N^-H\left(\frac{N-1}{N}\right)\de s\right)^2\right] \label{eq:right1}
        \\&\phantom{\le}+\expected_{\bar\nu_N }\left[\left(\int_r^t \frac{\delta\sqrt{N}}{N^\theta}[\bar\varphi_N(N-1)- g(\eta_s(N-1))]\nabla_N^-H(1)\de s\right)^2\right]\label{eq:right2}.
    \end{align}
    Now, by the Cauchy-Schwarz inequality, the stationarity of $\bar\nu_N $ and Lemma~\ref{lemma:g_moments},
    \begin{equation*}
        \eqref{eq:bulk}\le \frac{(t-r)^2}{N}\sum_{x=2}^{N-2} \Var_{\bar\nu_N }(g(\eta(x)))\left(\Delta_N H\left(\frac{x}{N}\right)\right)^2\lesssim (t-r)^2.
    \end{equation*}
    Moreover, by performing a Taylor expansion to replace discrete gradients with derivatives and by the boundary replacement lemma (Lemma~\ref{lemma:repl_boundaries}), we have that
    \begin{equation*}
        \eqref{eq:left1}+\eqref{eq:right1}\lesssim N^{\theta-1}
    \end{equation*}
    and
    \begin{equation*}
        \eqref{eq:left2}+\eqref{eq:right2}\lesssim N^{-\theta-1}.
    \end{equation*}
    Choosing $r=0$, this shows that the first condition of Aldous's criterion is verified. Fix now a stopping time $\tau$ almost surely bounded from above by $T$: combining the bounds above, we obtain
    \begin{align*}
        &\prob_{\bar\nu_N}\left\{\left|\int_\tau^{(\tau+\gamma)\wedge T} N^2\mathcal{L}^N\mathscr{Y}_s^N(H)\de s\right|>\eps\right\}
        \\&\le \frac{1}{\eps^2} \expected_{\bar\nu_N}\left[\left(\int_\tau^{(\tau+\gamma)\wedge T}N^2\mathcal{L}^N\mathscr{Y}_s^N(H)\de s\right)^2\right]
        \\&\lesssim \frac{1}{\eps^2}\left\{\gamma^2+ N^{\theta-1}+ N^{-\theta-1}\right\},
    \end{align*}
    which vanishes as $N\to\infty$ and $\gamma\to0$;

    \item $\theta<0$ and $\theta=1$: since the boundary terms \eqref{eq:theta<0} and \eqref{eq:theta=1} vanish in $L^2(\prob_{\bar\nu_N})$ as $N\to\infty$, the claim follows from the same bound given above for \eqref{eq:bulk};

    \item $\theta>1$: recall \eqref{eq:bg_term} and \eqref{eq:theta>1}, and note that, for any $0\le r\le t\le T$, by performing a Taylor expansion as above and by the boundary replacement lemma, we have that
    \begin{equation*}
        \expected_{\bar\nu_N }\left[\left(\int_r^t \frac{\lambda N^{3/2}}{N^\theta}[\bar\varphi_N(1)-g(\eta_s(1))]H\left(\frac{1}{N}\right)\de s\right)^2\right]\lesssim N^{1-\theta},
    \end{equation*}
    and
    \begin{equation*}
        \expected_{\bar\nu_N }\left[\left(\int_r^t \frac{\delta N^{3/2}}{N^\theta}[\bar\varphi_N(N-1)-g(\eta_s(N-1))]H\left(\frac{N-1}{N}\right)\de s\right)^2\right]\lesssim N^{1-\theta}.
    \end{equation*}
    Hence, following the same argument given for $0\le \theta<1$ above, we immediately obtain the first condition of Aldous's criterion, and moreover
    \begin{equation*}
        \prob_{\bar\nu_N}\left\{\left|\int_\tau^{(\tau+\gamma)\wedge T} N^2\mathcal{L}^N\mathscr{Y}_s^N(H)\de s\right|>\eps\right\}\lesssim \frac{1}{\eps^2}\left\{\gamma^2 + N^{1-\theta}\right\},
    \end{equation*}
    which again vanishes as $N\to\infty$ and $\gamma\to0$.
    
\end{enumerate}
This  completes the proof. \end{proof}

\subsubsection{The Martingale Part} 

First note that, for each $\theta\in\mathbb{R}$, a simple computation shows that \eqref{eq:dynkin_qv} can be written as
\begin{equation}\label{eq:quad_var}
    \begin{split}\scal{\mathscr{M}^N(H)}_t=&\;\int_0^t\frac{1}{N}\sum_{\substack{x, y\in I_N\\ |x-y|=1}}g(\eta_s(x))\left\{H\left(\frac{y}{N}\right)-H\left(\frac{x}{N}\right)\right\}^2\de s
    \\&+N^{1-\theta}\int_0^t\big(\alpha+\lambda g(\eta_s(1))\big)H\left(\frac{1}{N}\right)^2\de s
    \\&+ N^{1-\theta}\int_0^t\big(\beta+\delta g(\eta_s(N-1))\big)H\left(\frac{N-1}{N}\right)^2\de s.
    \end{split}
\end{equation}

\begin{lemma}\label{lemma:martingale_conv}
For any $\theta\in\mathbb{R}$ and any $H\in\mathcal{S}_\theta$, the sequence of  martingales $\{\mathscr{M}_t^N(H), t\in[0, T]\}_N$ converges in law to $\|\nabla H\|_{L^{2, \theta}}\mathscr{W}_t$, where $\|\cdot\|_{L^{2, \theta}}$ is given in \eqref{def:theta_norm} and $\{\mathscr{W}_t, t\in[0, T]\}$ is a standard one-dimensional Brownian motion.
\end{lemma}

In order to show the proposition above, we will use the following theorem, which is a straightforward corollary of~\cite[Theorem~VIII.3.11]{js03}.

\begin{theorem}\label{thm:jacod} Let $\{M^N_t, t\in[0, T]\}_N$ be a sequence of real-valued càdlàg martingales and let $\scal{M^N}_t$ denote the predictable quadratic variation of $M^N_t$. Let $f:[0, T]\to[0, \infty)$ be a deterministic, continuous function. Assume that:
\begin{enumerate}[i)]
    
    \item there exists a constant $K$ such that, for each $N$ and each $s\in[0, T]$, $\left|M^N_s-M^N_{s^-}\right|\le K$ almost surely,
    
    \item $\mylim_{N\to\infty}\expected\left[\mysup_{s\in[0, T]}\left|M^N_s-M^N_{s^-}\right|\right]=0$,
    
    \item for any $t\in[0, T]$, the sequence of random variables $\{\scal{M^N}_t\}_N$ converges in probability to $f(t)$.

\end{enumerate}
Then, the sequence $\{M^N_t, t\in[0, T]\}_N$ converges in law in $D([0, T], \mathbb{R})$ to a mean-zero Gaussian martingale on $[0, T]$ with continuous trajectories and with quadratic variation given by $f$.
\end{theorem}

We start by checking that, for each $H\in\mathcal{S}_\theta$, the sequence $\{\mathscr{M}_t^N(H), t\in[0, T]\}_N$ satisfies items i) and ii) of the theorem above.

\begin{lemma}\label{lemma:martingale_1}
For any $H\in\mathcal{S}_\theta$,
\begin{equation}\label{eq:jump_bound}
    \left|\mathscr{M}^N_s(H)-\mathscr{M}^N_{s-}(H)\right|\le \|\nabla H\|_\infty
\end{equation}
for each $s\in[0, T]$ almost surely, and 
\begin{equation}\label{eq:jump_limit}
    \mylim_{N\to\infty} \expected_{\bar\nu_N } \left[\mysup_{s\in[0, T]}\left|\mathscr{M}^N_s(H)-\mathscr{M}^N_{s-}(H)\right|\right]=0.
\end{equation}
\end{lemma}
\begin{proof}
From \eqref{eq:dynkin_formula}, for any $H\in\mathcal{S}_\theta$ we have that
\begin{equation*}
    \mysup_{s\in[0, T]}\left|\mathscr{M}^N_s(H)-\mathscr{M}^N_{s-}(H)\right|=\mysup_{s\in[0, T]}\left|\mathscr{Y}^N_s(H)-\mathscr{Y}^N_{s-}(H)\right|.
\end{equation*}
But now, it is not hard to see that
\begin{equation*}
    \left|\mathscr{Y}^N_s(H)-\mathscr{Y}^N_{s-}(H)\right|\lesssim \frac{\|H\|_\infty}{\sqrt{N}},
\end{equation*}
for any $s\in[0, T]$, which immediately yields both \eqref{eq:jump_bound} and \eqref{eq:jump_limit}.
\end{proof}

We now proceed to verify item iii), which follows from the next two lemmas.
\begin{lemma}\label{lemma:martingale_2}
For any $H\in\mathcal{S}_\theta$ and any $t\in[0, T]$, the quadratic variation \eqref{eq:dynkin_qv} of $\mathscr{M}^N_t(H)$ satisfies
\begin{equation*}
    \mylim_{N\to\infty} \expected_{\bar\nu_N }\left[\scal{\mathscr{M}^N(H)}_t\right]=t\|\nabla H\|_{L^{2, \theta}}^2,
\end{equation*}
with $\|\cdot\|_{L^{2, \theta}}$ defined in \eqref{def:theta_norm}.
\end{lemma}

\begin{proof}
By \eqref{eq:quad_var} and by the stationarity of $\bar\nu_N $, 
\begin{align}
    \expected_{\bar\nu_N }\left[\scal{\mathscr{M}^N(H)}_t\right] =&\;\int_0^t\frac{1}{N}\sum_{\substack{x, y\in I_N\\ |x-y|=1}}E_{\bar\nu_N }\left[g(\eta(x))\right]\left\{H\left(\frac{y}{N}\right)-H\left(\frac{x}{N}\right)\right\}^2\de s\nonumber
    \\\begin{split}&+N^{1-\theta}\int_0^t\big(\alpha+\lambda E_{\bar\nu_N }\left[g(\eta(1))\right]\big)H\left(\frac{1}{N}\right)^2\de s
    \\&+N^{1-\theta}\int_0^t\big(\beta+\delta E_{\bar\nu_N }\left[g(\eta(N-1))\right]\big)H\left(\frac{N-1}{N}\right)^2\de s.
    \end{split}\label{eq:qv_boundaries}
\end{align}
Recalling that $E_{\bar\nu_N }\left[g(\eta(x))\right]=\Phi(\rho_N(x))$, it is easy to see that
\begin{align*}
    &\mylim_{N\to\infty}\int_0^t\frac{1}{N}\sum_{\substack{x, y\in I_N\\ |x-y|=1}}E_{\bar\nu_N }\left[g(\eta(x))\right]\left\{H\left(\frac{y}{N}\right)-H\left(\frac{x}{N}\right)\right\}^2\de s
    \\&=t\int_0^12\Phi(\bar\rho_\theta(u))\big(\nabla H(u)\big)^2\de u,
\end{align*}
where $\bar\rho_\theta:[0, 1]\to[0, \infty)$ is given in Remark \ref{remark:limit_density}.
Moreover,
\begin{enumerate}[i)]

    \item $\theta<0$: since $H$ has derivatives of all orders vanishing at $0$ and $1$, by performing a Taylor expansion, for each positive integer $m$ we see that
    \begin{equation*}
        \eqref{eq:qv_boundaries} \lesssim \frac{t N^{1-\theta}}{N^{2m}}\hspace{-.2em}\left\{(\partial_u^m H(x_0))^2\big(\alpha+\lambda \bar\varphi_N(1)\big)+(\partial_u^m H(x_1))^2\big(\beta+\delta\bar\varphi_N(N-1)\big)\right\}
    \end{equation*}
    for some $x_0\in[0, \frac{1}{N}]$ and $x_1\in[\frac{N-1}{N}, 1]$. Choosing $m>\frac{1-\theta}{2}$, we see that the last display vanishes as $N\to\infty$;

    \item $0\le \theta <1$: since $H\in \mathcal{S}_\theta$, we have that 
    \begin{equation*}
        \eqref{eq:qv_boundaries} \lesssim \frac{tN^{1-\theta}}{N^2}\left\{(\partial_u H(0))^2\big(\alpha+\lambda \bar\varphi_N(1)\big)+(\partial_u H(1))^2\big(\beta+\delta\bar\varphi_N(N-1)\big)\right\},
    \end{equation*}
    which vanishes as $N\to\infty$;

    \item $\theta=1$: recalling that $H(0)=\frac{1}{\lambda}\partial_uH(0)$ and $H(1)=-\frac{1}{\delta}\partial_u H(1)$, we see that
    \begin{align*}
        \mylim_{N\to\infty} \eqref{eq:qv_boundaries}&= t\left\{\big(\alpha+\lambda \Phi(\bar\rho_\theta(0))\big)H(0)^2 + \big(\beta+\delta \Phi(\bar\rho_\theta(1))\big)H(1)^2\right\}
        \\&=t\left\{\frac{\alpha}{\lambda^2}+\frac{\Phi(\bar\rho_\theta(0))}{\lambda}\right\}(\partial_u H(0))^2+t\left\{\frac{\beta}{\delta^2}+\frac{\Phi(\bar\rho_\theta(1))}{\delta}\right\}(\partial_u H(1))^2;
    \end{align*}

    \item $\theta\ge1$: in this case,
    \begin{equation*}
        \eqref{eq:qv_boundaries}\lesssim t N^{1-\theta} \left\{H(0)^2\big(\alpha+\lambda\bar\varphi_N(1)\big)+H(1)^2\big(\beta+\delta\bar\varphi_N(N-1)\big)\right\},
    \end{equation*}
    so it automatically vanishes as $N\to\infty$.

\end{enumerate}
This completes the proof.
\end{proof}

\begin{lemma}\label{lemma:martingale_3}
For any $H\in\mathcal{S}_\theta$ and any $t\in[0, T]$, 
\begin{equation*}
    \mylim_{N\to\infty} \expected_{\bar\nu_N } \left[\left(\scal{\mathscr{M}^N(H)}_t-\expected_{\bar\nu_N }\left[\scal{\mathscr{M}^N(H)}_t\right]\right)^2\right]=0.
\end{equation*}
\end{lemma}

\begin{proof}
By \eqref{eq:quad_var}, the stationarity of $\bar\nu_N $ and a standard convexity inequality, we have that
\begin{align}
    &\expected_{\bar\nu_N } \left[\left(\scal{\mathscr{M}^N(H)}_t-\expected_{\bar\nu_N }\left[\scal{\mathscr{M}^N(H)}_t\right]\right)^2\right]\nonumber
    \\&\lesssim\expected_{\bar\nu_N } \left[\left(\frac{2}{N}\int_0^t\sum_{x=1}^{N-2}\left\{g(\eta_s(x))-E_{\bar\nu_N }\left[g(\eta(x))\right]\right\}\nabla_N^+ H\left(\frac{x}{N}\right)^2\de s\right)^2 \right]\label{eq:martingale_bulk}
    \\&\phantom{=}+\expected_{\bar\nu_N } \left[\left(\lambda N^{1-\theta}\int_0^t\left\{g(\eta_s(1))-E_{\bar\nu_N }\left[g(\eta(1))\right]\right\} H\left(\frac{1}{N}\right)^2\de s\right)^2 \right]\label{eq:martingale_left}
    \\&\phantom{=}+\expected_{\bar\nu_N } \left[\left(\delta N^{1-\theta}\int_0^t\left\{g(\eta_s(N-1))-E_{\bar\nu_N }\left[g(\eta(N-1))\right]\right\} H\left(\frac{N-1}{N}\right)^2\de s\right)^2 \right].\label{eq:martingale_right}
\end{align}
Now, by the Cauchy-Schwarz inequality and the fact that $\bar\nu_N $ is a product measure,
\begin{equation*}
    \eqref{eq:martingale_bulk} \lesssim \frac{t^2}{N^2}\sum_{x=1}^{N-2}\Var_{\bar\nu_N }(g(\eta(x))),
\end{equation*}
which vanishes as $N\to\infty$ by Lemma~\ref{lemma:g_moments}. As for \eqref{eq:martingale_left}, again by the Cauchy-Schwarz inequality:
\begin{enumerate}[i)]

    \item $\theta<0$: by performing a Taylor expansion of $H$ around $0$, for each positive integer $m$ we see that
    \begin{equation*}
        \eqref{eq:martingale_left} \lesssim t^2N^{2-2\theta-4m}(\partial_u^m H(x_0))^2,
    \end{equation*}
    for some $x_0\in[0, \frac{1}{N}]$. Choosing $m>\frac{1-\theta}{2}$, we see that the last display vanishes as $N\to\infty$;
    
    \item $0\le\theta<1$: since $H\in\mathcal{S}_\theta$ and by Lemma~\ref{lemma:repl_boundaries},
    \begin{equation*}
        \eqref{eq:martingale_left} \lesssim t^2N^{-2\theta}(\partial_u H(0))^2 N^{\theta-2},
    \end{equation*}
    which vanishes as $N\to\infty$;

    \item $\theta\ge 1$: by the Cauchy-Schwarz inequality and Lemma~\ref{lemma:repl_boundaries},
    \begin{equation*}
        \eqref{eq:martingale_left} \lesssim t^2N^{2-2\theta}H(0)^2 N^{\theta-2},
        \end{equation*}
    which again vanishes as $N\to\infty$.

\end{enumerate}
The argument for \eqref{eq:martingale_right} is identical, hence the claim follows.
\end{proof}

\begin{proof}[Proof of Lemma~\ref{lemma:martingale_conv}] By Lemmas~\ref{lemma:martingale_1}, \ref{lemma:martingale_2} and \ref{lemma:martingale_3}, for each $H\in\mathcal{S}_\theta$, the sequence of martingales $\{\mathscr{M}_t^N(H), t\in[0, T]\}_N$ satisfies the assumptions of Theorem~\ref{thm:jacod} with $f(t)=t\|\nabla H\|_{L^{2, \theta}}^2$. Calling $\{\mathscr{M}_t(H), t\in[0, T]\}$ its limit point, by Lévy's characterisation of Brownian motion, the process 
\begin{equation*}\big(\| \nabla H\|_{L^{2, \theta}}\big)^{-1}\mathscr{M}_t(H)
\end{equation*}
is a standard one-dimensional Brownian motion, which completes the proof.
\end{proof}

%%%%%%%%%%%%%%%%%%%%%%%%%%%%%%%%%%%%%%%%%%%%%%%%%%
%%%Characterisation of the Limit Point%%%%%%%%%%%%
%%%%%%%%%%%%%%%%%%%%%%%%%%%%%%%%%%%%%%%%%%%%%%%%%%
\subsection{Characterisation of the Limit Point}\label{subsec:characterisation}

We finally conclude the proof of Theorem~\ref{thm:fluctuations} by completing the characterisation of the limit point of the density fluctuation field. 

What is mainly left to show is that, in the case $\theta<0$, any limit path $\mathscr{Y}$ of the density fluctuation field satisfies items i) and ii) of Proposition~\ref{prop:uniqueness_neg}. We start by showing item i), which follows immediately.

\begin{lemma} For each $\theta<0$, any limit point $\mathcal{Q}$ of $\{\mathcal{Q}^N\}_N$ is concentrated on paths $\mathscr{Y}\in\mathcal{D}([0, T], \mathcal{S}'_\theta)$ which satisfy item i) of Proposition~\ref{prop:uniqueness_neg}.
\end{lemma}

\begin{proof} Recall \eqref{def:L2_norm}. It suffices to show that $\expected_{\bar\nu_N}[\mathscr{Y}_t^N(H)^2]\lesssim \| H\|_{L^2}^2$ for each $H\in\mathcal{S}_\theta$. But now,
\begin{align*}
    \expected_{\bar\nu_N}\left[\mathscr{Y}_t^N(H)^2\right]&=\frac{1}{N}\sum_{x=1}^{N-1} H\left(\frac{x}{N}\right)^2 E_{\bar\nu_N}[\bar\eta(x)^2]\lesssim \frac{1}{N}\sum_{x=1}^{N-1} H\left(\frac{x}{N}\right)^2,
\end{align*}
where we used Lemma~\ref{lemma:bounded_moments}.
\end{proof}

The proof of item ii) is more involved. Since the functions $\Phi'(\bar\rho_\theta)\iota_\eps^0$ and $\Phi'(\bar\rho_\theta)\iota_\eps^1$ do not belong to the space of test functions $\mathcal{S}_\theta$, we first have to define the quantities involved in condition ii). To this end, following~\cite[Section~5]{bgjs22}, let $\mathscr{H}$ denote the Hilbert space of real-valued, progressively measurable processes $\{x_t, t \in [0,T]\}$ in $L^2([0, 1])$, equipped with the norm $\|\cdot \|_{\mathscr{H}}$ defined via

\begin{equation}\label{def:H_norm}
    \|x\|_{\mathscr{H}}^2:=\expected\left[\int_0^T |x_t|^2\de t\right].
\end{equation}

\begin{lemma}\label{lemma:def_Y_eps} 
Assume that the stochastic process $\mathscr{Y}\in\mathcal{D}([0, T], \mathcal{S}'_\theta)$ satisfies condition i) of Proposition~\ref{prop:uniqueness_neg}. For any $H\in L^2([0, 1])$, we can define in a unique way a stochastic process on $\mathscr{H}$, which we will still denote by $\{\mathscr{Y}_t(H), t\in [0,T]\}$, such that it coincides with $\{\mathscr{Y}_t(H), t\in [0,T]\}$ for $H\in\mathcal{S}_\theta$. Moreover, condition i) also holds for any $H\in L^2([0, 1])$.
\end{lemma}

\begin{proof} Let $H\in L^2([0, 1])$ and let $\{H_\eps\}_{\eps>0}$ be a sequence of functions in $\mathcal{S}_\theta$ converging to $H$ in $L^2([0, 1])$ as $\eps\to0$. By the Cauchy-Schwarz inequality and condition i), the sequence of real-valued processes $\{\mathscr{Y}_t(H_\eps), t \in [0,T]\}$ is a Cauchy sequence in $\mathscr{H}$, and thus it converges to a process (in $\mathscr{H}$) which we denote by $\{\mathscr{Y}_t(H), t \in [0,T]\}$. It is easy to show that the limiting process depends only on $H$ and not on the approximating sequence $\{H_\eps\}_{\eps>0}$, justifying the notation. Condition i) trivially holds for this process.
\end{proof}

Throughout, even if $\eps N$ is not an integer, we will write $\eps N$ for $\lfloor \eps N\rfloor$. The result below follows the approach of~\cite[Lemma~5.2]{bgjs22}. The proof of this result relies on the local Boltzmann-Gibbs principle (Theorem~\ref{thm:local_boltzmann_gibbs}), and hence we require the boundedness condition \eqref{assumption:g_bdd}.

\begin{lemma} Assume \eqref{assumption:g_bdd}. For each $\theta<0$, any limit point $\mathcal{Q}$ of $\{\mathcal{Q}^N\}_N$ is concentrated on paths $\mathscr{Y}\in\mathcal{D}([0, T], \mathcal{S}'_\theta)$ which satisfy item ii) of Proposition~\ref{prop:uniqueness_neg}.
\end{lemma}

\begin{proof} We will only show the result for $\iota_\eps^0$, as the proof for $\iota_\eps^1$ is completely analogous. Fix $\eps>0$ and let $\{H_\eps^m\}_{m\ge1}$ be a sequence of functions in $\mathcal{S}_\theta$ such that $\mylim_{m\to\infty}\|H_\eps^m-\iota_\eps^0\|_{L^2}^2=0$. By condition i) of Proposition~\ref{prop:uniqueness_neg} and the Cauchy-Schwarz inequality, we see that
\begin{align*}
    \expected\left[\mysup_{t\in[0, T]}\left(\int_0^t \mathscr{Y}_s(\Phi'(\bar\rho_\theta)\iota_\eps^0)\de s\right)^2\right]&\lesssim T^2 \left\|\Phi'(\bar\rho_\theta)(H_\eps^m-\iota_\eps^0)\right\|_{L^2}^2
    \\&\phantom{\le}+ \expected\left[\mysup_{t\in[0, T]}\left(\int_0^t \mathscr{Y}_s(\Phi'(\bar\rho_\theta)H_\eps^m)\de s\right)^2\right].
\end{align*}
Note now that, if we equip $\mathcal{D}([0, T], \mathcal{S}_\theta')$ with the uniform topology, for any $H\in\mathcal{S}_\theta$, the map
\begin{equation*}
    \mathcal{D}([0, T], \mathcal{S}_\theta') \ni \mathscr{Z}\mapsto \mysup_{t\in[0, T]}\left(\int_0^t\mathscr{Z}_s(H)\de s\right)^2
\in [0, \infty)
\end{equation*}
is continuous, and in particular it is lower semi-continuous and bounded from below. Hence, by the convergence of $\{\mathscr{Y}^N_t, t\in[0, T]\}$ to $\{\mathscr{Y}_t, t\in[0, T]\}$ in distribution, we have that
\begin{equation*}
    \expected\left[\mysup_{t\in[0, T]}\hspace{-.1em}\left(\int_0^t \mathscr{Y}_s(\Phi'(\bar\rho_\theta)H_\eps^m)\de s\right)^2\right]\hspace{-.2em}\le\myliminf_{N\to\infty} \expected_{\bar\nu_N}\hspace{-.2em}\left[\mysup_{t\in[0, T]}\hspace{-.1em}\left(\int_0^t \mathscr{Y}_s^N(\Phi'(\bar\rho_\theta)H_\eps^m)\de s\right)^2\right].
\end{equation*}
By the Cauchy-Schwarz inequality and the stationarity of $\{\mathscr{Y}^N_t, t\in[0, T]\}$ with respect to $\bar\nu_N$, we have that
\begin{align*}
    \expected_{\bar\nu_N}\left[\mysup_{t\in[0, T]}\left(\int_0^t \mathscr{Y}_s^N(\Phi'(\bar\rho_\theta)H_\eps^m)\de s\right)^2\right]&\lesssim T^2
    \left\|\Phi'(\bar\rho_\theta)(H_\eps^m-\iota_\eps^0)\right\|_{L^2}^2
    \\&\phantom{\lesssim}+\expected_{\bar\nu_N}\left[\mysup_{t\in[0, T]}\left(\int_0^t \mathscr{Y}_s^N(\Phi'(\bar\rho_\theta)\iota_\eps^0)\de s\right)^2\right].
\end{align*}
Then, by the linearity of $\{\mathscr{Y}_t(\cdot), t\in[0, T]\}$, a standard convexity inequality and Lemma~\ref{lemma:def_Y_eps}, we see that
\begin{align*}
    \expected\left[\mysup_{t\in[0, T]}\left(\int_0^t \mathscr{Y}_s(\Phi'(\bar\rho_\theta)\iota_\eps^0)\de s\right)^2\right]&\lesssim T^2\|\Phi'(\bar\rho_\theta)\|_{\infty}^2
    \|H_\eps^m-\iota_\eps^0\|_{L^2}^2
    \\&\phantom{\le}+\,\myliminf_{N\to\infty}\expected_{\bar\nu_N}\left[\mysup_{t\in[0, T]}\left(\int_0^t \mathscr{Y}_s^N(\Phi'(\bar\rho_\theta)\iota_\eps^0)\de s\right)^2\right].
\end{align*}
Sending $m\to\infty$, we get that it suffices to show that
\begin{equation*}
    \mylim_{\eps\to0}\mylim_{N\to\infty}\expected_{\bar\nu_N}\left[\mysup_{t\in[0, T]}\left(\int_0^t\mathscr{Y}_s^N(\Phi'(\bar\rho_\theta)\iota_\eps^0)\de s\right)^2\right]=0.
\end{equation*}
To this end, we write
\begin{align}
    &\expected_{\bar\nu_N}\left[\mysup_{t\in[0, T]}\left(\int_0^t\mathscr{Y}_s^N(\Phi'(\bar\rho_\theta)\iota_\eps^0)\de s\right)^2\right]\nonumber 
    \\&=\expected_{\bar\nu_N}\left[\mysup_{t\in[0, T]}\left(\int_0^t\frac{1}{\eps\sqrt{N}}\sum_{x=1}^{\eps N} \Phi'\left(\bar\rho_\theta\left(\frac{x}{N}\right)\right)\bar\eta_s^N(x)\de s\right)^2\right]\nonumber 
    \\&\lesssim \expected_{\bar\nu_N}\left[\mysup_{t\in[0, T]}\left(\int_0^t\frac{1}{\eps\sqrt{N}}\sum_{x=1}^{\eps N} \left\{\Phi'\left(\bar\rho_\theta\left(\frac{x}{N}\right)\right)-\Phi'(\rho_N(x))\right\}\bar\eta_s^N(x)\de s\right)^2\right]\label{eq:new_1}
    \\&\phantom{\lesssim}+\expected_{\bar\nu_N}\left[\mysup_{t\in[0, T]}\left(\int_0^t\frac{1}{\eps\sqrt{N}}\sum_{x=1}^{\eps N} \left\{\Phi'(\rho_N(x))\bar\eta_s^N(x)-\big(g(\eta_s(x))-\varphi_N(x)\big)\right\}\de s\right)^2\right]\label{eq:new_2} \\&\phantom{\lesssim}+\expected_{\bar\nu_N}\left[\mysup_{t\in[0, T]}\left(\frac{1}{\eps\sqrt{N}}\int_0^t\sum_{x=1}^{\eps N} \big(g(\eta_s(x))-\bar\varphi_N(x)\big)\de s\right)^2\right]. \nonumber 
\end{align}
Now, for each $x\in I_N$, $\left|\Phi'\left(\bar\rho_\theta\left(\frac{x}{N}\right)\right)-\Phi'(\rho_N(x))\right|$ goes to zero as $N\to\infty$ (by definition of $\bar\rho_\theta$), and hence \eqref{eq:new_1} vanishes as $N\to\infty$. By the local Boltzmann-Gibbs principle (Theorem~\ref{thm:local_boltzmann_gibbs}), \eqref{eq:new_2} also vanishes as $N\to\infty$. Thus, it only remains to show that
\begin{equation*}
    \mylim_{\eps\to0}\mylim_{N\to\infty}\expected_{\bar\nu_N}\left[\mysup_{t\in[0, T]}\left(\frac{1}{\eps\sqrt{N}}\int_0^t\sum_{x=1}^{\eps N} \big(g(\eta_s(x))-\bar\varphi_N(x)\big)\de s\right)^2\right]=0.
\end{equation*}
This is achieved by Lemma~\ref{lemma:local_limit} below.
\end{proof}

\begin{lemma}\label{lemma:local_limit} For each $\theta<0$, we have that
\begin{equation*}
    \mylim_{\eps\to0}\mylim_{N\to\infty}\expected_{\bar\nu_N}\left[\mysup_{t\in[0, T]}\left(\frac{1}{\eps\sqrt{N}}\int_0^t\sum_{x=1}^{\eps N} \big(g(\eta_s(x))-\bar\varphi_N(x)\big)\de s\right)^2\right]=0.
\end{equation*}
\end{lemma}
\begin{proof} By a standard convexity inequality, we have that
\begin{align}
    &\expected_{\bar\nu_N}\left[\mysup_{t\in[0, T]}\left(\frac{1}{\eps\sqrt{N}}\int_0^t\sum_{x=1}^{\eps N} \big(g(\eta_s(x))-\bar\varphi_N(x)\big)\de s\right)^2\right]\nonumber
    \\&\le 2\,\expected_{\bar\nu_N}\left[\mysup_{t\in[0, T]}\left(\frac{1}{\eps\sqrt{N}}\int_0^t\sum_{x=1}^{\eps N} \big(g(\eta_s(x))-g(\eta_s(1))+\bar\varphi_N(1)-\bar\varphi_N(x)\big)\de s\right)^2\right]\label{eq:long_term}
    \\&\phantom{\le}+2\,\expected_{\bar\nu_N}\left[\mysup_{t\in[0, T]}\left(\frac{1}{\eps\sqrt{N}}\int_0^t\sum_{x=1}^{\eps N} \big(g(\eta_s(1))-\bar\varphi_N(1)\big)\de s\right)^2\right].\label{eq:rl_term}
\end{align}
By Lemma~\ref{lemma:repl_boundaries}, we immediately see that
\begin{equation*}
    \eqref{eq:rl_term}=2\,\expected_{\bar\nu_N}\left[\mysup_{t\in[0, T]}\left(\sqrt{N}\int_0^t\big(g(\eta_s(1))-\bar\varphi_N(1)\big)\de s\right)^2\right]\lesssim \frac{1}{\eps^2} N^{\theta-1},
\end{equation*}
which vanishes as $N\to\infty$ for any $\theta<0$. 

Let now $\bar g(\eta(x)):=g(\eta(x))-\bar\varphi_N(x)$ for each $x\in I_N$ (so that $E_{\bar\nu_N}[\bar g(\eta(x))]=0$). It is not hard to check that we can write
\begin{align*}
    &\sum_{x=1}^{\eps N} \big(g(\eta(x))-g(\eta(1))+\bar\varphi_N(1)-\bar\varphi_N(x)\big)
    \\&=\sum_{x=1}^{\eps N} \sum_{y=1}^{x-1}\left\{g(\eta(y))-\frac{\bar\varphi_N(y)}{\bar\varphi_N(y+1)}g(\eta(y+1))\right\}
    \\&\phantom{=}+\sum_{x=1}^{\eps N} \sum_{y=1}^{x-1}\bar g(\eta(y+1))\left\{\frac{\bar\varphi_N(y)}{\bar\varphi_N(y+1)}-1\right\}.\label{eq:cs_term}
\end{align*}
Hence, by the same convexity inequality as above, it only remains to bound
\begin{equation}\label{eq:kv_term}
    \expected_{\bar\nu_N}\left[\mysup_{t\in[0, T]}\left(\frac{1}{\eps\sqrt{N}}\int_0^t\sum_{x=1}^{\eps N} \sum_{y=1}^{x-1}\left\{g(\eta_s(y))-\frac{\bar\varphi_N(y)}{\bar\varphi_N(y+1)}g(\eta_s(y+1))\right\}\de s\right)^2\right]
\end{equation}
and
\begin{equation}\label{eq:cs_term}
    \expected_{\bar\nu_N}\left[\mysup_{t\in[0, T]}\left(\frac{1}{\eps\sqrt{N}}\int_0^t\sum_{x=1}^{\eps N} \sum_{y=1}^{x-1}\bar g(\eta_s(y+1))\left\{\frac{\bar\varphi_N(y)}{\bar\varphi_N(y+1)}-1\right\}\de s\right)^2\right].
\end{equation}
Now, by the Kipnis-Varadhan inequality~\cite[Lemma~2.4]{klo12}, we have that
\begin{equation}\label{eq:long_sup}
    \begin{split}
    \eqref{eq:kv_term}\lesssim T\mysup_{f\in L^2(\bar\nu_N)}\Bigg\{&\int_{\Omega_N}\frac{2}{\eps\sqrt{N}}\sum_{x=1}^{\eps N}\sum_{y=1}^{x-1}\left\{g(\eta(y))-\frac{\bar\varphi_N(y)}{\bar\varphi_N(y+1)}g(\eta(y+1))\right\}f(\eta)\de \bar\nu_N 
    \\&-N^2\mathscr{D}_N(f)\Bigg\},
    \end{split}
\end{equation}
with $\mathscr{D}_N(f)$ defined in \eqref{def:dirichlet_form}. We now write
\begin{align}
    &\int_{\Omega_N}\frac{2}{\eps\sqrt{N}}\sum_{x=1}^{\eps N}\sum_{y=1}^{x-1}g(\eta(y))f(\eta)\de \bar\nu_N \nonumber
    \\&=\int_{\Omega_N}\frac{2}{\eps\sqrt{N}}\sum_{x=1}^{\eps N}\sum_{y=1}^{x-1}g(\eta(y))\left[f(\eta)-f(\eta^{y, y+1})\right]\de \bar\nu_N \label{eq:young_term}
    \\&\phantom{=}+\int_{\Omega_N}\frac{2}{\eps\sqrt{N}}\sum_{x=1}^{\eps N}\sum_{y=1}^{x-1}g(\eta(y))f(\eta^{y, y+1})\de \bar\nu_N.\label{eq:cancellation_term}
\end{align}
By Young's inequality, for each $A>0$ we see that
\begin{align}
    \eqref{eq:young_term}&\le \frac{A}{\eps\sqrt{N}}\sum_{x=1}^{\eps N}\sum_{y=1}^{x-1}\int_{\Omega_N}g(\eta(y))\left[f(\eta)-f(\eta^{y, y+1})\right]^2\de \bar\nu_N\label{eq:dir_term}
    \\&\phantom{\le}+\frac{1}{A\eps\sqrt{N}}\sum_{x=1}^{\eps N}\sum_{y=1}^{x-1}\int_{\Omega_N}g(\eta(y))\de \bar\nu_N.\label{eq:vanishing_term}
\end{align}
Choosing $A=N^{3/2}$, we get $\eqref{eq:dir_term}\le N^2\mathscr{D}_N(f)$. Moreover, since $E_{\bar\nu_N}[g(\eta(y))]$ is uniformly bounded in $N$ and $y\in I_N$, this choice yields $\eqref{eq:vanishing_term}\lesssim \frac{(\eps N)^2}{N^{3/2}\eps\sqrt{N}}=\eps$, which vanishes as $N\to\infty$ and $\eps\to0$.

As for \eqref{eq:cancellation_term}, by performing the change of variable $\eta\mapsto\eta^{y+1, y}$ and using the identity
\begin{equation*}
    \frac{\bar\nu_N(\eta^{y+1, y})}{\bar\nu_N(\eta)}=\frac{\bar\varphi_N(y)}{\bar\varphi_N(y+1)}\frac{g(\eta(y+1))}{g(\eta(y)+1)}
\end{equation*}
(which follows immediately from \eqref{eq:nu}), we see that
\begin{equation*}
    \eqref{eq:cancellation_term}=\int_{\Omega_N}\frac{2}{\eps\sqrt{N}}\sum_{x=1}^{\eps N}\sum_{y=1}^{x-1}g(\eta(y+1))f(\eta)\frac{\bar\varphi_N(y)}{\bar\varphi_N(y+1)}\de \bar\nu_N.
\end{equation*}
But then, this term cancels with the remaining contribution in the supremum in \eqref{eq:long_sup}, so that this all yields $\mylim_{\eps\to0}\mylim_{N\to\infty}\eqref{eq:kv_term}=0$.

Finally, it remains to treat \eqref{eq:cs_term}: by the Cauchy-Schwarz inequality, since $\bar\nu_N$ is a product measure and by Lemma~\ref{lemma:g_moments}, we have that
\begin{equation*}
    \eqref{eq:cs_term}\lesssim \frac{T^2}{\eps^2 N} \sum_{x=1}^{\eps N}\sum_{y=1}^{x-1}\sum_{z=1}^{\eps N}\left\{\frac{\bar\varphi_N(y)}{\bar\varphi_N(y+1)}-1\right\}^2E_{\bar\nu_N}[\bar g(\eta(y))^2]\lesssim \frac{T^2(\eps N)^3}{\eps^2 N\cdot N^2}=T^2\eps ,
\end{equation*}
where we used $\left|\frac{\bar\varphi_N(y)}{\bar\varphi_N(y+1)}-1\right|\lesssim\frac{1}{N}$. Sending $\eps\to0$ completes the proof.
\end{proof}

We are now ready to conclude the proof of Theorem~\ref{thm:fluctuations}.

\begin{proof}[Proof of Theorem~\ref{thm:fluctuations}]
By the arguments given in Section~\ref{subsec:dynkin}, together with Lemmas~\ref{lemma:initial_time} and \ref{lemma:martingale_conv}, we immediately get that any limit point $\mathcal{Q}$ of $\{\mathcal{Q}^N\}_N$ is concentrated on stationary paths $\mathscr{Y}\in\mathcal{D}([0, T], \mathcal{S}'_\theta)$ solving equation \eqref{eq:general_OU} with operators \eqref{eq:OU_operators} started at a Gaussian field $\mathscr{Y}_0$ satisfying \eqref{eq:Y_covariance}. But now, it follows from Lemma~\ref{lemma:martingale_conv} that, for each $H\in\mathcal{S}_\theta$, $\mathscr{Y}(H)$ has continuous trajectories: by a standard argument, this implies that $\mathscr{Y}$ has continuous trajectories with respect to the strong topology of $\mathcal{S}'_\theta$. Hence, we get the following:
\begin{enumerate}[i)]
    
    \item $\theta\ge0$: by Proposition~\ref{prop:uniqueness}, $\mathcal{Q}$ is concentrated on the \textit{unique} stationary path $\mathscr{Y}\in\mathcal{C}([0, T], \mathcal{S}'_\theta)$ solving equation \eqref{eq:general_OU} with operators \eqref{eq:OU_operators} and started at $\mathscr{Y}_0$ satisfying \eqref{eq:Y_covariance};

    \item $\theta<0$: by Proposition~\ref{prop:uniqueness_neg}, the same argument as for $\theta\ge0$ and the sequence of results above, the claim follows.
    
\end{enumerate}

This completes the proof.
\end{proof}

\section{Proof of Proposition~\ref{prop:fluctuations_new}}\label{sec:martingale_new}

The proof of this proposition follows the same steps as the proof of Theorem~\ref{thm:fluctuations} presented in the previous section. We therefore do not repeat it in full, but only highlight the terms that require additional arguments. These are precisely the terms $\mathscr{Z}_t^{\theta, i}(H)$ appearing in \eqref{eq:martingale_problem_new}.

\begin{proof}[Proof of Proposition~\ref{prop:fluctuations_new}] Recall \eqref{eq:dynkin_terms}. We present the proof only for the case $\theta=1$, since the case $\theta>1$ is completely analogous and, in fact, simpler. Observe that the space of test functions is now $\widetilde{S}$, and therefore we cannot use boundary conditions to simplify the expression of the martingale. When $\theta=1$, the terms that require further analysis are those appearing in the second through fourth lines of \eqref{eq:dynkin_terms}. We focus on the left boundary terms, as the right boundary terms can be treated in the same way. We thus need to analyse
\begin{equation*}
    \sqrt{N}[g(\eta_s(1))-\bar\varphi_N(1)]
    \left\{\nabla_N^+H\left(\frac{1}{N}\right)
    -\lambda H\left(\frac{1}{N}\right)\right\}.
\end{equation*}
We first claim that
\begin{equation*}
    \mylim_{N\to\infty} \expected_{\bar\nu_N}\left[\left(\int_0^t\Big\{\sqrt{N}[g(\eta_s(1))-\bar\varphi_N(1)]-\mathscr{Y}_s^N(\Phi'(\bar\rho_\theta)\iota_\eps^0)\Big\}\de s\right)^2\right] \lesssim \eps(T^2+1).
\end{equation*}
To prove this claim, we apply a standard convexity inequality and bound the expectation above by the sum of the terms in \eqref{eq:new_1}, \eqref{eq:new_2}, and \eqref{eq:long_term}. Collecting these bounds yields the desired estimate. By the triangle inequality and for $\delta<\eps$, we further obtain
\begin{equation*}
    \mylim_{N\to\infty}
    \expected_{\bar\nu_N}\left[\left(\int_0^t\Big\{\mathscr{Y}_s^N(\Phi'(\bar\rho_\theta)\iota_\eps^0)-\mathscr{Y}_s^N(\Phi'(\bar\rho_\theta)\iota_\delta^0)\Big\}\de s\right)^2\right] \lesssim  \eps(T^2+1).
\end{equation*}
Since upper moment bounds are stable under convergence in distribution, by passing to the limit along a subsequence we obtain
\begin{equation}
    \expected\left[\left(\int_0^t\Big\{\mathscr{Y}_s(\Phi'(\bar\rho_\theta)\iota_\eps^0)-\mathscr{Y}_s(\Phi'(\bar\rho_\theta)\iota_\delta^0)\Big\}\de s\right)^2\right] \lesssim  \eps(T^2+1).
\end{equation}
Therefore, the family
$\{\mathscr{Y}_s(\Phi'(\bar\rho_\theta)\iota_\eps^0)\}_{\eps>0}$
is a Cauchy sequence, and hence it converges in $L^2$. We denote its limit by $\mathscr{Z}_t^{\theta, 0}(H)$. For simplicity, we have assumed that the test functions are time-independent; however, the argument extends verbatim to the time-dependent case. Finally, to handle the quadratic variation, it suffices to repeat the same steps leading to \eqref{eq:quad_var}, which do not rely on boundary conditions for the test functions. This shows that the limiting martingale problem satisfies \eqref{eq:martingale_problem_new}, which completes the proof.
\end{proof}

\appendix
%%%%%%%%%%%%%%%%%%%%%%%%%%%%%%%%%%%%%%%%%%%%%%%%%%
%%%UNIQUENESS OF THE MARTINGALE PROBLEM%%%%%%%%%%%
%%%%%%%%%%%%%%%%%%%%%%%%%%%%%%%%%%%%%%%%%%%%%%%%%%
\section{Uniqueness of the Martingale Problems}\label{sec:app_uniqueness}

In this appendix we prove Propositions~\ref{prop:uniqueness} and \ref{prop:uniqueness_neg}. Our method to show the former can be seen as a generalisation of the approach used to show an analogue uniqueness result in~\cite{fgn19}, whereas for the latter we follow \cite{bgjs22}.

Throughout, we will denote by $\|\cdot\|_\infty$ the $L^\infty$ norm on $[0, 1]$, namely $\|f(u)\|_\infty:=\mysup_{u\in[0, 1]}|f(u)|$ for each $f:[0, 1]\to\mathbb{R}$. For the sake of completeness, we start by showing the following property of the map $\Phi'(\bar\rho_\theta(\cdot))$. 

\begin{lemma}\label{lemma:Phi'_positive} For each $\theta\in\mathbb{R}$, we have that $\Phi'(\bar\rho_\theta(u))>0$ for every $u\in [0, 1]$. 
\end{lemma}

\begin{proof} Recalling the definition of $\Phi$ as the inverse of the function $R$ defined in \eqref{eq:R}, we can write 
\begin{equation}\label{eq:Phi'=1/R}
    \Phi'(\rho)=\frac{1}{R'(\Phi(\rho))}
\end{equation}
for each $\rho\ge0$. For $\varphi\in [0, \varphi^*)$, let $\mu_\varphi$ denote the probability distribution on $\mathbb{N}$ given by $\mu_\varphi(j):=Z(\varphi)^{-1}\frac{\varphi^j}{g(j)!}$, and let
\begin{equation*}
    Z_1(\varphi):=\sum_{j=0}^\infty \frac{j\varphi^j}{g(j)!} \quad \text{and} \quad Z_2(\varphi):=\sum_{j=0}^\infty \frac{j^2\varphi^j}{g(j)!}.
\end{equation*}
Then, writing $R=\frac{Z_1}{Z}$ and denoting by $E_\varphi[\,\cdot\,]$ and $\Var_\varphi(\cdot)$ the expectation and the variance with respect to $\mu_\varphi$, respectively, we see that
\begin{align}
    R'(\varphi) &= \frac{Z_1'(\varphi)Z(\varphi)-Z_1(\varphi)Z'(\varphi)}{Z(\varphi)^2}\nonumber
     = \frac{1}{\varphi}\frac{Z_2(\varphi)Z(\varphi)-Z_1(\varphi)^2}{Z(\varphi)^2}\nonumber
    \\&=\frac{1}{\varphi}\left(E_\varphi[j^2]-E_\varphi[j]^2\right)
    =\frac{\Var_\varphi(j)}{\varphi},\label{eq:var_varphi}
\end{align}
which, by \eqref{eq:Phi'=1/R}, immediately gives $\Phi'(\rho)\ge 0$ for each $\rho\ge0$. But now, since
\begin{equation*}
    0<\myinf_{u\in[0, 1]}\bar\varphi_\theta(u)\le \mysup_{u\in[0, 1]}\bar\varphi_\theta(u)\le\varphi^*
\end{equation*}
and since the expression in \eqref{eq:var_varphi} is strictly positive on $\varphi\in(0, \varphi^*]$, our claim follows.
\end{proof}

%%%%%%%%%%%%%%%%%%%%%%%%%%%%%%%%%%%%%%%%%%%%%%%%%%
%%%Sturm-Liouville Problems%%%%%%%%%%%%%%%%%%%%%%%
%%%%%%%%%%%%%%%%%%%%%%%%%%%%%%%%%%%%%%%%%%%%%%%%%%
\subsection{Sturm-Liouville Problems}

Let $A$ be a smooth, positive function on $[0, 1]$ and consider the following \textit{Sturm-Liouville (SL) problems} on $[0, 1]$:
\begin{equation}\label{eq:SL_0}
    \left\{ 
    \begin{array}{ll}
    A(u)\psi''(u) +\gamma \psi(u)=0 &\mbox{ for } u\in(0, 1),
    \\\psi(0) =\psi(1) = 0, &
    \end{array}
    \right.
\end{equation}
\begin{equation}\label{eq:SL_1}
    \left\{ 
    \begin{array}{ll}
    A(u)\psi''(u) +\gamma \psi(u)=0 &\mbox{ for } u\in(0, 1),
    \\\psi'(0) = \lambda\psi(0), &
    \\\psi'(1) = -\delta\psi(1), &
    \end{array}
    \right.
\end{equation}
\begin{equation}\label{eq:SL_2}
    \left\{ 
    \begin{array}{ll}
    A(u)\psi''(u) +\gamma \psi(u)=0 & \mbox{ for } u\in(0, 1),
    \\\psi'(0) =\psi'(1) = 0 &
    \end{array}
    \right.
\end{equation}
(with $\lambda, \delta>0$), which we will refer to as $\textsc{SL}_0(A)$, $\textsc{SL}_1(A)$ and $\textsc{SL}_2(A)$, respectively; see~\cite[Chapter~10]{br89} for a detailed exposition of the topic. 

\begin{lemma}\label{lemma:SL_facts} For each $j=0, 1, 2$, the system $\textsc{SL}_j(A)$ satisfies the following:
\begin{enumerate}[i)]
    
    \item the spectrum is real and an ordered sequence $\gamma_1<\gamma_2<\ldots \to\infty$; 
    
    \item the sequence of normalised eigenfunctions $\{\psi_n\}_n$ forms a complete, orthonormal basis of the $A^{-1}$-weighted Hilbert space      \begin{equation}\label{def:weighted_hilbert}
        L^2([0, 1], A(u)^{-1}\de u):=\left\{f:[0, 1]\to\mathbb{R}: \|f\|^2_{L^2([0, 1], A^{-1})}<\infty\right\},
    \end{equation}
    where $\|\cdot\|_{L^2([0, 1], A^{-1})}$ denotes the norm associated to the inner product
    \begin{equation*}
        \langle f, h\rangle_{L^2([0, 1], A^{-1})}:=\int_0^1 f(u)A(u)^{-1}h(u)\de u;
    \end{equation*} 
    
    \item the sequence $\{\psi_n\}_n$ defined above satisfies the bound $\|\psi_n\|_\infty \lesssim \sqrt{\gamma_n}\vee 1$.

\end{enumerate}
\end{lemma}

\begin{proof} All three systems $\textsc{SL}_0(A), \textsc{SL}_1(A)$ and $\textsc{SL}_2(A)$ are \textit{regular} SL problems; see~\cite[Section~10.1]{br89}. Standard results of the theory, then, immediately yield items i) and ii), so we only need to verify item iii).

Recalling \eqref{def:L2_norm}, we start by noting that
\begin{equation}\label{eq:L_infty_bound}
    \|\psi_n\|_\infty \le \|\psi_n\|_{L^2}+\|\psi_n'\|_{L^2}.
\end{equation} 
Since $A$ is smooth and positive, the spaces $L^2([0, 1])$ and $L^2([0, 1], A(u)^{-1}\de u)$ coincide, and the norms $\|\cdot\|_{L^2}$ and $\|\cdot\|_{L^2([0, 1], A^{-1})}$ are equivalent. Then, since $\{\psi_n\}_n$ is an orthonormal basis of $L^2([0, 1], A(u)^{-1}\de u)$, we immediately get a uniform upper bound for $\|\psi_n\|_{L^2}$. In order to estimate $\|\psi_n'\|_{L^2}$, performing an integration by parts and using $A(u)\psi_n''(u)+\gamma_n \psi_n(u)=0$ for $u\in(0, 1)$, we write
\begin{equation}\label{eq:psi'_bound}
    \|\psi_n'\|_{L^2}^2=\psi_n\psi_n'\big|_0^1-\int_0^1 \psi_n''(u)\psi_n(u)\de u = \psi_n\psi_n'\big|_0^1+\gamma_n\int_0^1 \frac{\psi_n(u)^2}{A(u)}\de u.
\end{equation}
But now, for both problems \eqref{eq:SL_0} and \eqref{eq:SL_2}, the boundary conditions imply that $\psi_n\psi_n'\big|_0^1=0$; nonetheless, for the problem \eqref{eq:SL_1} we get
\begin{equation*}
    \psi_n\psi_n'\big|_0^1=-\delta\psi_n(1)^2-\lambda\psi_n(0)^2\le 0.
\end{equation*}
Hence, since $\myinf_{u\in[0, 1]}A(u)>0$, combining \eqref{eq:L_infty_bound} and \eqref{eq:psi'_bound} completes the proof.
\end{proof}

Below, for two continuous functions $A, \widetilde A$ on $[0, 1]$, we write $A\le \widetilde A$ if $A(u)\le \widetilde A(u)$ for (almost) every $u\in [0, 1]$.

\begin{lemma}\label{lemma:SL_ineuqality} For each $j=0, 1, 2$, the following holds. If $A\le \widetilde A$, then the sequences of eigenvalues $\{\gamma_n\}_n, \{\widetilde\gamma_n\}_n$ of $\textsc{SL}_j(A)$ and $\textsc{SL}_j(\widetilde A)$, respectively, satisfy $\gamma_{n-1}\le\widetilde\gamma_n$ for each $n\ge2$.

\begin{proof} Fix $j\in\{0, 1, 2\}$ and consider the differential equations on $(0, 1)$
\begin{equation*}
    \psi_n''(u)+A(u)^{-1}\widetilde \gamma_n \psi_n(u)=0 \quad \text{and} \quad \widetilde\psi_n''(u)+ \widetilde A(u)^{-1} \widetilde \gamma_n\widetilde\psi_n(u)=0.
\end{equation*}
Since $A(u)^{-1} \widetilde \gamma_n\ge \widetilde  A(u)^{-1} \widetilde \gamma_n $ for each $u\in(0, 1)$, by the Sturm Comparison Theorem~\cite[Theorem~10.6.3]{br89}, the number of zeroes of any solution to the first equation is at least the number of zeroes of any solution to the second one. Moreover, the same theorem implies that the number of zeroes of any solution to the differential equation
\begin{equation*}
    \psi_n''(u)+A(u)^{-1} \gamma \psi_n(u)=0
\end{equation*}
is a non-decreasing function of $\gamma$. But now, by~\cite[Theorem~10.8.5]{br89}, any eigenfunction $\psi_n$ of $\textsc{SL}_j(A)$ corresponding to $\gamma_n$ has exactly $n-1$ zeroes in the open interval $(0, 1)$. This completes the proof.
\end{proof}
\end{lemma}

%%%%%%%%%%%%%%%%%%%%%%%%%%%%%%%%%%%%%%%%%%%%%%%%%%
%%%Semigroup Results%%%%%%%%%%%%%%%%%%%%%%%%%%
%%%%%%%%%%%%%%%%%%%%%%%%%%%%%%%%%%%%%%%%%%%%%%%%%%
\subsection{Semigroup Results}

For $\theta\ge0$, let $P_t^\theta$ denote the semigroup associated to the operator $\mathfrak{A}_\theta$ given in Definition \ref{def:S_theta}.

\begin{lemma}\label{lemma:semigroup} Let $h\in L^2([0, 1])$. For each $\theta\ge0$, the semigroup $(P_t^\theta h)(u)$ can be written as
\begin{equation}\label{eq:explicit_semigroup}
    (P_t^\theta h)(u)=\sum_{n=1}^\infty a_n e^{-\gamma_n t} \psi_n(u),
\end{equation}
for $(t, u)\in[0, \infty)\times[0, 1]$, where:
\begin{enumerate}[i)]

    \item $\{\psi_n\}_n$ is a complete, orthonormal basis of the $\Phi'(\bar\rho_\theta)^{-1}$-weighted Hilbert space $L^2([0, 1], \Phi'(\bar\rho_\theta(u))^{-1}\de u)$, defined in \eqref{def:weighted_hilbert};

    \item $\{a_n\}_n$ are the coefficients of $h$ in that basis, namely $a_n=\langle h, \psi_n\rangle_{L^2([0, 1], \Phi'(\bar\rho_\theta)^{-1})}$;

    \item the sequence $\{\gamma_n\}_n$ satisfies
    \begin{equation}\label{eq:eigen_sandwich}
        \kappa_1((n-2)\pi)^2\le \gamma_n\le \kappa_2((n+1)\pi)^2
    \end{equation}
    for each $n\ge 2$, where
    \begin{equation}\label{eq:kappas}
        0<\kappa_1:=\myinf_{u\in[0,1]}\Phi'(\bar\rho_\theta(u))\le \mysup_{u\in[0,1]}\Phi'(\bar\rho_\theta(u))=:\kappa_2<\infty.
    \end{equation}

\end{enumerate}
Consequently, the series \eqref{eq:explicit_semigroup} converges exponentially fast for $t>0$, and thus $(P_t^\theta h)(u)$ is smooth $(0, \infty)\times[0, 1]$. Moreover, if $h\in\mathcal{S}_\theta$, then $(P_t^\theta h)(u)$ is smooth $[0, \infty)\times[0, 1]$.
\end{lemma}

\begin{proof} We start by noting that the semigroup $P_t^\theta$ is such that, given $h\in L^2([0, 1])$, $(P_t^\theta h)(u)$ is the solution to the following PDEs:

\begin{enumerate}[i)]
    
    \item $0\le\theta<1$:
    \begin{equation}\label{eq:pde_theta<1}
        \left\{ 
        \begin{array}{ll}
        \partial_t \rho_t(u) = \Phi'(\bar\rho_\theta(u))\partial_u^2 \rho_t(u)&\mbox{ for } u \in (0,1) \mbox{ and } t>0,
        \\\rho_t(0) =\rho_t(1) = 0 &\mbox{ for } t>0,
        \\\rho_0(u) = h(u) &\mbox{ for } u \in [0,1];
        \end{array}
        \right.
    \end{equation}
    
    \item $\theta=1$:
    \begin{equation}\label{eq:pde_theta=1}
        \left\{ 
        \begin{array}{ll}
        \partial_t \rho_t(u) = \Phi'(\bar\rho_\theta(u))\partial_u^2 \rho_t(u) &\mbox{ for } u \in (0,1) \mbox{ and } t>0,
        \\ \partial_u\rho_t(0) = \lambda\rho_t(0) &\mbox{ for } t>0,
        \\ \partial_u\rho_t(1) = -\delta\rho_t(1) &\mbox{ for } t>0,
        \\\rho_0(u) = h(u) &\mbox{ for } u \in [0,1];
        \end{array}
        \right.
    \end{equation}
    
    \item $\theta\ge1$:
    \begin{equation}\label{eq:pde_theta>1}
        \left\{ 
        \begin{array}{ll}
        \partial_t \rho_t(u) = \Phi'(\bar\rho_\theta(u))\partial_u^2 \rho_t(u)&\mbox{ for } u \in (0,1) \mbox{ and } t>0,
        \\\partial_u \rho_t(0)= \partial_u \rho_t(1) = 0 &\mbox{ for } t>0,
        \\\rho_0(u) = h(u) &\mbox{ for } u \in [0,1].
        \end{array}
        \right.
    \end{equation}

\end{enumerate}
Consider now the following SL problems on $[0, 1]$:
\begin{enumerate}[i)]
    
    \item $0\le\theta<1$:
    \begin{equation}\label{eq:SL_theta<1}
        \left\{ 
        \begin{array}{ll}
        \Phi'(\bar\rho_\theta(u))\psi''(u) +\gamma \psi(u)=0 &\mbox{ for } u\in(0, 1),
        \\\psi(0)  =\psi(1) =0; &
        \end{array}
        \right.
    \end{equation}
    
    \item $\theta=1$:
    \begin{equation}\label{eq:SL_theta=1}
        \left\{ 
        \begin{array}{ll}
        \Phi'(\bar\rho_\theta(u))\psi''(u) +\gamma \psi(u)=0 &\mbox{ for } u\in(0, 1),
        \\\psi'(0) = \lambda\psi(0), &
        \\\psi'(1) = -\delta\psi(1); &
        \end{array}
        \right.
    \end{equation}

    \item $\theta>1$:
    \begin{equation}\label{eq:SL_theta>1}
        \left\{ 
        \begin{array}{ll}
        \Phi'(\bar\rho_\theta(u))\psi''(u) +\gamma \psi(u)=0 & \mbox{ for } u\in(0, 1),
        \\\psi'(0) = \psi'(1) =0, &
        \end{array}
        \right.
    \end{equation}

\end{enumerate} 
which are associated to \eqref{eq:pde_theta<1}, \eqref{eq:pde_theta=1} and \eqref{eq:pde_theta>1}, respectively. Since $\Phi'(\bar\rho_\theta(\cdot))$ is smooth and positive (see Lemma~\ref{lemma:Phi'_positive} above), the systems \eqref{eq:SL_theta<1}, \eqref{eq:SL_theta=1} and \eqref{eq:SL_theta>1} all satisfy items i), ii) and iii) of Lemma~\ref{lemma:SL_facts}. In particular, applying the classical method of separation of variables and by the completeness of the basis $\{\psi_n\}_n$, we immediately get \eqref{eq:explicit_semigroup}. 

Now, by Lemma~\ref{lemma:Phi'_positive}, the eigenvalues of the SL problems \eqref{eq:SL_theta<1}, \eqref{eq:SL_theta=1} and \eqref{eq:SL_theta>1} can be ``sandwiched" between those of the SL problems given by the differential equations
\begin{equation*}
    \kappa_1\psi''_{(1)}(u)+\gamma \psi_{(1)}(u)=0 \quad \text{and} \quad \kappa_2\psi''_{(2)}(u)+\gamma \psi_{(2)}(u)=0
\end{equation*}
with corresponding boundary conditions, with $\kappa_1, \kappa_2$ defined in \eqref{eq:kappas}. It is not hard to show that, for each of our three sets of boundary conditions, calling $\{\gamma_{n, (1)}\}_n$ and $\{\gamma_{n, (2)}\}_n$ the respective sequences of eigenvalues of the problems described above,
\begin{equation*}
    (n-1)\pi\le \sqrt{\frac{\gamma_{n, (1)}}{\kappa_1}}\le n\pi \quad \text{and} \quad  (n-1)\pi\le \sqrt{\frac{\gamma_{n, (2)}}{\kappa_2}}\le n\pi
\end{equation*}
for each $n\ge 1$, which yields \eqref{eq:eigen_sandwich}.

To show the remaining part of the statement, we now have all the ingredients to apply an analogue argument to the one given at the end of the proof of~\cite[Proposition~3.1]{fgn19}, which we present for the sake of completeness. If $h\in L^2([0, 1])$, then by the Cauchy-Schwarz inequality and the fact that $\{\psi_n\}_n$ is an orthonormal basis of the weighted space $L^2([0, 1], \Phi'(\bar\rho_\theta(u))^{-1}\de u)$, we immediately get \begin{equation*}
    |a_n|=\left|\langle h, \psi_n\rangle_{L^2([0, 1], \Phi'(\bar\rho_\theta)^{-1})}\right|=\left|\langle  h \Phi'(\bar\rho_\theta)^{-1}, \psi_n\rangle_{L^2}\right|\le \frac{\kappa_2}{\kappa_1}\|h\|_{L^2},
\end{equation*}
with $\kappa_1, \kappa_2$ given in \eqref{eq:kappas}, so that the Fourier coefficients $\{a_n\}_n$ are bounded in absolute value. By \eqref{eq:eigen_sandwich} and since $\|\psi_n\|_{\infty}$ grows at most polynomially in $n$ (recall \eqref{eq:eigen_sandwich} and item iii) of Lemma~\ref{lemma:SL_facts}), the series \eqref{eq:explicit_semigroup} converges exponentially fast, which gives smoothness of $(P_t^\theta h)(u)$ on $(0, \infty)\times[0, 1]$. If, moreover, we assume that $h\in\mathcal{S}_\theta$, then
\begin{align*}
    a_n&=\langle h, \psi_n\rangle_{L^2([0, 1], \Phi'(\bar\rho_\theta)^{-1})}=\left\langle h \Phi'(\bar\rho_\theta)^{-1}, \psi_n\right\rangle_{L^2}
    \\&=\frac{1}{-\gamma_n}\left\langle h \Phi'(\bar\rho_\theta)^{-1}, -\gamma_n\psi_n\right\rangle_{L^2}=\frac{1}{-\gamma_n}\left\langle  h \Phi'(\bar\rho_\theta)^{-1}, \Phi'( \bar\rho_\theta)\partial_u^2\psi_n\right\rangle_{L^2}
    \\&=\frac{1}{-\gamma_n}\left\langle  \partial_u^2h, \psi_n\right\rangle_{L^2},
\end{align*}
where in the last equality we used the fact that, for any $\theta\ge0$, the operator $\partial_u^2$ is self-adjoint in $\mathcal{S}_\theta$ with the respect to the inner product $\langle\cdot, \cdot\rangle_{L^2}$. Again by the Cauchy-Schwarz inequality, we get $|a_n|\le \frac{c_1}{n^2}$ for some positive constant $c_1$ independent of $n$. Since $\mathfrak{A}_\theta^k h\in\mathcal{S}_\theta$ for each non-negative integer $k$, the trick above can be performed inductively, so as to show that
\begin{equation*}
    a_n=\frac{1}{(-\gamma_n)^k}\left\langle  \mathfrak A_\theta^k h, \partial_u^2\psi_n\right\rangle_{L^2}=\frac{1}{(-\gamma_n)^k}\left\langle  \partial_u^2(\mathfrak A_\theta^k h), \psi_n\right\rangle_{L^2}
\end{equation*}
for each non-negative integer $k$. Hence, for each such $k$, there exists a positive constant $c_k$ such that $|a_n|\le\frac{c_k}{n^{2k}}$. But now, again recalling that $\|\psi_n\|_\infty\lesssim \sqrt{\gamma_n}\vee 1$, this implies that the series
\begin{equation*}
    \sum_{n=1}^\infty |a_n|(-\gamma_n)^k\|\psi_n\|_\infty
\end{equation*}
converges for each non-negative integer $k$, and thus $(P_t^\theta h)(u)$ is of class $C^k$ on $[0, \infty)\times[0, 1]$ for each non-negative integer $k$.\end{proof}

From Lemma~\ref{lemma:semigroup} we can immediately deduce the following corollaries, which will be important in the proof of uniqueness.

\begin{corollary}\label{corol:semigroup_belonging}
    For any $\theta\ge0$, if $H\in \mathcal{S}_\theta$, then $P_t^\theta H\in\mathcal{S}_\theta$. 
\end{corollary}

\begin{corollary}\label{corol:semigroup_taylor} For any $\theta\ge0$ and $H\in \mathcal{S}_\theta$, we have 
\begin{equation*}
    P_{t+\eps}^\theta H=P_t^\theta H+\eps \mathfrak{A}_\theta P_t^\theta H+o(\eps, t),
\end{equation*}
where $o(\eps, t)$ denotes a function in $\mathcal{S}_\theta$ such that $\mylim_{\eps\to0}\frac{o(\eps, t)}{\eps}=0$ in the topology of $\mathcal{S}_\theta$. Moreover, the limit is uniform on compact time intervals.
\end{corollary}

%%%%%%%%%%%%%%%%%%%%%%%%%%%%%%%%%%%%%%%%%%%%%%%%%%
%%%Proof of Uniqueness%%%%%%%%%%%%%%%%%%%%%%%%%%%%
%%%%%%%%%%%%%%%%%%%%%%%%%%%%%%%%%%%%%%%%%%%%%%%%%%
\subsection{Proof of Proposition~\ref{prop:uniqueness}}

We are now ready to show the uniqueness of the martingale problem in Theorem~\ref{thm:fluctuations} for $\theta\ge0$. Our proof closely follows that of~\cite[Proposition~5.3]{fgn19}, which in turn has the structure of the proof of~\cite[Theorem~11.0.2]{kl99}, based on a more general approach due to Holley and Stroock~\cite{hs76}.

\begin{proof}[Proof of Proposition~\ref{prop:uniqueness}] Fix $H\in\mathcal{S}_\theta$ and $s\in(0, T]$. Note first that the process defined by
\begin{equation*}
   \left(\|\nabla H\|_{L^{2, \theta}}\right)^{-1} \mathscr{M}_t(H)
\end{equation*}
is a standard one-dimensional Brownian motion: hence, by Itô's formula, the $\mathbb{C}$-valued process $\{\mathscr{X}_{t, s}(H), t\in [s, T]\}$ defined by
\begin{equation*}
    \mathscr{X}_{t, s}(H):=e^{\frac{1}{2}(t-s)\|\nabla H\|_{L^{2, \theta}}^2+i\big(\mathscr{Y}_t(H)-\mathscr{Y}_s(H)-\int_s^t\mathscr{Y}_r(\mathfrak{A}_\theta H)\de r\big)}
\end{equation*}
is a martingale. We claim that, for each fixed $S \in (0, T]$ and $H\in\mathcal{S}_\theta$, the $\mathbb{C}$-valued process $\{\mathscr{Z}_{t, S}(H), t\in [0, S]\}$ defined by
\begin{equation*}
    \mathscr{Z}_{t, S}(H):=e^{\frac{1}{2}\int_0^t\left\|\nabla P_{S-r}^\theta H\right\|_{L^{2, \theta}}^2\de r + i \mathscr{Y}_t\left(P_{S-t}^\theta H\right)}
\end{equation*}
is also a martingale (where this process is well defined in view of Corollary~\ref{corol:semigroup_belonging}). In order to prove this, fix $0\le t_1<t_2\le S$ and, for each positive integer $n$ and each integer $0\le j\le n$, set $s_j^n:=t_1+\frac{j}{n}(t_2-t_1)$. This implies that the time sequence $t_1\le s_1^n< s_2^n<\ldots <s_n^n\le t_2$ partitions the interval $[t_1, t_2]$ into $n$ subintervals of equal size. Now, we can write
\begin{align}
    &\prod_{j=0}^{n-1} \mathscr{X}_{s_{j+1}^n, s_j^n}\left(P_{S-s_j^n}^\theta H\right) \nonumber
    \\& = e^{\frac{1}{2}\sum_{j=0}^{n-1}(s_{j+1}^n-s_j^n)\big\|\nabla P_{S-s_j^n}^\theta H\big\|_{L^{2, \theta}}^2} \times \label{eq:exp_sum1}
    \\&\phantom{=}\times e^{i\sum_{j=0}^{n-1}\left(\mathscr{Y}_{s_{j+1}^n}\big(P_{S-s_j^n}^\theta H\big)-\mathscr{Y}_{s_j^n}\big(P_{S-s_j^n}^\theta H\big)-\int_{s_j^n}^{s_{j+1}^n} \mathscr{Y}_r\big(\mathfrak{A}_\theta P_{S-s_j^n}^\theta H\big)\de r\right)}. \label{eq:exp_sum2}
\end{align}
By the smoothness of the semigroup $P_t^\theta$, guaranteed by Lemma~\ref{lemma:semigroup}, we see that
\begin{equation}\label{eq:exp_sum_lim}
    \mylim_{n\to\infty} \eqref{eq:exp_sum1} = e^{\frac{1}{2}\int_{t_1}^{t_2} \left\|\nabla P_{S-r}^\theta H\right\|_{L^{2, \theta}}^2\de r}
\end{equation}
almost surely. Moreover, we can write the sum in \eqref{eq:exp_sum2} as
\begin{equation}\label{eq:nonexp_sum}
    \begin{split}&\mathscr{Y}_{t_2}\left(P_{S-t_2+\frac{t_2-t_1}{n}}^\theta H\right) - \mathscr{Y}_{t_1}\left(P_{S-t_1}^\theta H\right) 
    \\&+\sum_{j=1}^{n-1}\left(\mathscr{Y}_{s_j^n}\left(P_{S-s_{j-1}^n}^\theta H- P_{S-s_{j}^n}^\theta H\right)-\int_{s_j^n}^{s_{j+1}^n} \mathscr{Y}_r\left(\mathfrak{A}_\theta P_{S-s_j^n}^\theta H\right)\de r\right),
    \end{split}
\end{equation}
where we used the linearity of $\{\mathscr{Y}_t, t\in[0, T]\}$ in the test function. Now, by Corollary~\ref{corol:semigroup_taylor}, we can further rewrite the sum in \eqref{eq:nonexp_sum} as
\begin{equation*}
    \begin{split}\sum_{j=1}^{n-1}\Bigg( &\int_{s_j^n}^{s_{j+1}^n} \left\{\mathscr{Y}_{s_j^n}\left(\mathfrak{A}_\theta P_{S-s_j^n}^\theta H\right)- \mathscr{Y}_r\left(\mathfrak{A}_\theta P_{S-s_j^n}^\theta H\right) \right\}\de r 
    \\&+ \mathscr{Y}_{s_j^n}\left(o\left(\frac{t_2-t_1}{n}, S-s_j^n\right)\right)\Bigg).
    \end{split}
\end{equation*}
By Lemma~\ref{lemma:semigroup}, for each fixed $H\in\mathcal{S}_\theta$, the map $t\mapsto \mathfrak{A}_\theta P_t^\theta H$ is uniformly continuous on $[0, T]$. Moreover, since $\mathscr{Y}\in\mathcal{C}([0, T], \mathcal{S}_\theta')$, the map $(s, t)\mapsto \mathscr{Y}_s(\mathfrak{A}_\theta P_t^\theta H)$ is uniformly continuous on the compact set $[0, T]\times[s, T]$. This implies that the last display vanishes as $n\to\infty$, which in turn implies that $\mylim_{n\to\infty}\eqref{eq:nonexp_sum}=\mathscr{Y}_{t_2}(P_{S-t_2}^\theta H)-\mathscr{Y}_{t_1}(P_{S-t_1}^\theta H)$ almost surely. Recalling \eqref{eq:exp_sum_lim}, we have thus proved that 
\begin{equation*}
    \mylim_{n\to\infty} \prod_{j=0}^{n-1} \mathscr{X}_{s_{j+1}^n, s_j^n}\left(P_{S-s_j^n}^\theta H\right)  = e^{\frac{1}{2}\int_{t_1}^{t_2} \left\|\nabla P_{S-r}^\theta H\right\|_{L^{2, \theta}}^2\de r + i \big(\mathscr{Y}_{t_2}\left(P_{S-t_2}^\theta H\right)-\mathscr{Y}_{t_1}\left(P_{S-t_1}^\theta H\right)\big)}
\end{equation*}
almost surely, which in turn is equal to $\frac{\mathscr{Z}_{t_2, S}(H)}{\mathscr{Z}_{t_1, S}(H)}$ almost surely.

\vspace{.6em}
Let now $G$ be a bounded, $\mathcal{F}_{t_1}$-measurable function, where $\mathcal{F}_t=\sigma\{\mathscr{Y}_s(H): s\le t, H\in\mathcal{S}_\theta\}$. By the Dominated Convergence Theorem, the convergence above takes place in $L^1$, and thus
\begin{equation}\label{eq:lim_prodX}
    \expected\left[G\frac{\mathscr{Z}_{t_2, S}(H)}{\mathscr{Z}_{t_1, S}(H)}\right]=\mylim_{n\to\infty}\expected\left[G \prod_{j=0}^{n-1} \mathscr{X}_{s_{j+1}^n, s_j^n}\left(P_{S-s_j^n}^\theta H\right) \right].
\end{equation}
Since $\{\mathscr{X}_{t, s}(H), t\in [s, T]\}$ is a martingale, taking conditional expectations with respect to $\mathcal{F}_{s_{n-1}^n}$ yields
\begin{equation*}
    \expected\left[G \prod_{j=0}^{n-1} \mathscr{X}_{s_{j+1}^n, s_j^n}\left(P_{S-s_j^n}^\theta H\right) \right]=\expected\left[G \prod_{j=0}^{n-2} \mathscr{X}_{s_{j+1}^n, s_j^n}\left(P_{S-s_j^n}^\theta H\right) \right].
\end{equation*}
Iterating this identity inductively, by \eqref{eq:lim_prodX} we deduce that
\begin{equation*}
    \expected\left[G\frac{\mathscr{Z}_{t_2, S}(H)}{\mathscr{Z}_{t_1, S}(H)}\right]=\expected[G]
\end{equation*}
for any bounded, $\mathcal{F}_{t_1}$-measurable function $G$. This shows that $\{\mathscr{Z}_{t, S}(H), t\in [0, S]\}$ is indeed a martingale. But then, since $\expected[\mathscr{Z}_{t, S}(H)|\mathcal{F}_s]=\mathscr{Z}_{s, S}(H)$, we obtain
\begin{equation*}
    \expected\left[e^{\frac{1}{2}\int_0^t \big\|\nabla P_{S-r}^\theta H\big\|_{L^{2, \theta}}^2\de r+i\mathscr{Y}_t\left(P_{S-t}^\theta H\right)}\;\Big|\; \mathcal{F}_s\right]=e^{\frac{1}{2}\int_0^s \big\|\nabla P_{S-r}^\theta H\big\|_{L^{2, \theta}}^2\de r+i\mathscr{Y}_s\left(P_{S-s}^\theta H\right)},
\end{equation*}
which in turn yields
\begin{equation*}
    \expected\left[e^{i\mathscr{Y}_t\left(P_{S-t}^\theta H\right)} \;\Big|\;  \mathcal{F}_s\right]=e^{-\frac{1}{2}\int_s^t \big\|\nabla P_{S-r}^\theta H\big\|_{L^{2, \theta}}^2\de r+i\mathscr{Y}_s\left(P_{S-s}^\theta H\right)}.
\end{equation*}
Note that $P_{S-s}^\theta H=P_{t-s}^\theta P_{S-t}^\theta H$, and hence, writing $h:=P_{S-t}^\theta H$, we get
\begin{equation*}
    \expected\left[e^{i\mathscr{Y}_t\left(h\right)} \;\Big|\;  \mathcal{F}_s\right]=e^{-\frac{1}{2}\int_s^t \big\|\nabla P_{t-r}^\theta h\big\|_{L^{2, \theta}}^2\de r+i\mathscr{Y}_s\left(P_{t-s}^\theta h\right)}.
\end{equation*}
This means that, for each $H\in\mathcal{S}_\theta$ and conditionally on $\mathcal{F}_s$, the random variable $\mathscr{Y}_t(H)$ has a Gaussian distribution of mean $\mathscr{Y}_s(P_{t-s}^\theta H)$ and variance $\int_s^ t\|\nabla P_{t-r}^\theta H\|_{L^{2, \theta}}^2\de r$. But then, since the time-zero distribution is determined by \eqref{eq:Y_covariance}, a standard Markov argument yields the uniqueness of finite-dimensional distributions of $\{\mathscr{Y}_t(H), t\in [0, T]\}$, which in turn ensures the uniqueness in distribution of the stochastic process $\{\mathscr{Y}_t, t\in[0, T]\}$.
\end{proof}

%%%%%%%%%%%%%%%%%%%%%%%%%%%%%%%%%%%%%%%%%%%%%%%%%%
%%%Proof of Uniqueness Neg%%%%%%%%%%%%%%%%%%%%%%%%
%%%%%%%%%%%%%%%%%%%%%%%%%%%%%%%%%%%%%%%%%%%%%%%%%%
\subsection{Proof of Proposition~\ref{prop:uniqueness_neg}}

Finally, we conclude by showing uniqueness of the martingale problem in Theorem~\ref{thm:fluctuations} for $\theta<0$. Our proof follows the approach of~\cite[Section~5.1.2]{bgjs22}. Throughout, we will assume $\theta<0$ and we will call $\mathcal{S}:=\mathcal{S}_\theta$ (since this space of test functions has in fact no dependence on $\theta$), which we recall is given by
\begin{equation*}
    \mathcal{S}=\left\{f\in C^\infty([0, 1]): \begin{array}{ll}\partial_u^k f(0)=0, \\ \partial_u^k f(1)=0\end{array}  \text{for all integers} \ k\ge0\right\}.
\end{equation*}
Moreover, we will denote by $\mathcal{S}_{\theta, \Dir}$ the space of test functions given in Definition \ref{def:S_theta} but for $0\le\theta<1$, namely
\begin{equation*}
    \mathcal{S}_{\theta, \Dir}:=\left\{f\in C^\infty([0, 1]): \begin{array}{ll}\mathfrak{A}_\theta^k f(0)=0, \\ \mathfrak{A}_\theta^k f(1)=0\end{array}  \text{for all integers} \ k\ge0\right\}.
\end{equation*}
By the arguments given in the previous sections, we have already proved the existence of a solution to the martingale problem stated in Proposition~\ref{prop:uniqueness_neg}. The idea is then the following: take any $\mathcal{S}'$-valued solution to this martingale problem, show that, for any time $t\in[0, T]$, there exists an $\mathcal{S}_{\theta, \Dir}'$-valued solution $\mathscr{\tilde Y}$ which extends it, and then show that the extension belongs to the space of continuous paths $\mathcal{C}([0, T], \mathcal{S}_{\theta, \Dir}')$. Since we already know (from the argument given to show Proposition~\ref{prop:uniqueness}) that solutions to the martingale problem in Proposition~\ref{prop:uniqueness} are unique in law, this would then imply uniqueness of the $\mathcal{S}'$-valued limit point.

Let $\mathcal{H}^2([0, 1])$ denote the Sobolev space of square integrable and twice differentiable functions $f:[0, 1]\to\mathbb{R}$ such that $\partial_u f$ and $\partial_u^2 f$ are also square integrable, equipped with the norm $\|\cdot\|_{\mathcal{H}^2}$ defined via
\begin{equation*}
    \|f\|_{\mathcal{H}^2}^2:=\int_0^2 f(u)^2\de u+\int_0^2 (\partial_u f(u))^2\de u+\int_0^2(\partial_u^2 f(u))^2\de u.
\end{equation*}
We define the ``intermediate" space of test functions
\begin{equation*}
    \mathcal{\tilde S}:=\left\{f\in \mathcal{H}^2([0, 1]): \begin{array}{ll}f(0)=\partial_u f(0)=0, \\ f(1)=\partial_u f(1)=0\end{array}\right\}. 
\end{equation*} 
This space is a closed vector subspace of $\mathcal{H}^2([0,1])$ and contains $\mathcal{S}$. 

Recall the definition of the space $L^2([0, 1])$ in Section~\ref{subsec:mg_problems}. The following approximation result can be found in~\cite[Lemma~A.4]{bgjs22}.
\begin{lemma}\label{lemma:approximation} Let $H \in \mathcal{\tilde S}$. There exists a sequence of functions $\{H_\eps\}_{\eps>0}$ in $\mathcal{S}$ such that, for $k=0,1,2$, $\mylim_{\eps\to 0} \partial_u^k H_{\eps} = \partial_u^k H$ in $L^2([0, 1])$, namely $\{H_\eps\}_{\eps>0}$ converges to $H$ in $\mathcal{H}^2([0, 1])$.
\end{lemma}

Let now $\mathscr{B}$ denote the Banach space of real-valued processes $\{x_t, t \in [0,T]\}$ with continuous trajectories, equipped with the norm $\|\cdot\|_{\mathscr{B}}$ defined by 
\begin{equation*}
    \|x\|_{\mathscr{B}}^2:=\expected\left[\mysup_{t\in[0, T]} |x_t|^2\right].
\end{equation*}
Recalling the definition of $\mathscr{H}$ in Section~\ref{subsec:characterisation} and the norm $\|\cdot\|_{\mathscr{H}}$ in \eqref{def:H_norm}, we see that $\mathscr{B} \subseteq \mathscr{H}$, and moreover $\|x\|_{\mathscr{H}} \le \sqrt{T}\|x\|_{\mathscr{B}}$. In particular, if a sequence of processes $\{x^\eps\}_{\eps>0}$ converges to $x$ in $\mathscr{B}$ as $\eps\to 0$, then the convergence also holds in $\mathscr{H}$. 

We will keep the notation $\mathfrak{A}_\theta$ to denote the extension of the operator $\mathfrak{A}_\theta$ to $\mathcal{H}^2([0, 1])$, and similarly for the norm $\|\cdot\|_{L^{2, \theta}}$ defined in \eqref{def:theta_norm}, which we recall is simply given by
\begin{equation*}
    \|H\|_{L^{2, \theta}}^2=2\int_0^1 \Phi(\bar\rho_\theta(u))H^2(u)\de u
\end{equation*}
for $\theta<0$. The result below follows \cite[Lemma~5.3]{bgjs22}.

\begin{lemma}\label{lemma:mg_problem_tilde_S} Let $\mathscr{Y}$ be an $\mathcal{S}'$-valued stochastic process satisfying \eqref{eq:martingale_problem} and item i) of Proposition~\ref{prop:uniqueness_neg} for each $H\in\mathcal{S}$. Then, for any $H\in\mathcal{\tilde S}$, the process $\{\mathscr{\tilde M}(H)_t, t\in[0, T]\}$ defined via
\begin{equation}\label{eq:tilde_M}
    \mathscr{\tilde M}_t(H):=\mathscr{Y}_t(H)-\mathscr{Y}_0(H)-\int_0^t\mathscr{Y}_s(\mathfrak{A}_\theta H)\de s
\end{equation}
is a martingale with continuous trajectories and with quadratic variation given by
\begin{equation}\label{eq:qv_tilde_M}
    \scal{\mathscr{\tilde M}(H)}_t=t\|\partial_u H\|_{L^{2, \theta}}^2.
\end{equation}
Moreover, the process $\{\mathscr{Y}_t(H), t\in[0, T]\}$ has continuous trajectories.
\end{lemma}

\begin{remark} Since, for $H\in\mathcal{\tilde S}$, the functions $H$ and $\mathfrak{A}_\theta H$ may not be in $\mathcal{S}$, the terms on the right-hand side of \eqref{eq:tilde_M} are to be understood in the sense of Lemma~\ref{lemma:def_Y_eps}.
\end{remark}

\begin{proof}[Proof of Lemma~\ref{lemma:mg_problem_tilde_S}] Given $H\in\mathcal{\tilde S}$, by Lemma~\ref{lemma:approximation} we can approximate $H$ by a sequence $\{H_\eps\}_{\eps>0}$ in $\mathcal{S}$ such that, for $k=0, 1, 2$, $\|\partial_u^k H_\eps - \partial_u^k H\|_{L^2}^2\to0$ as $\eps\to0$, so that in particular $\|\partial_u H_\eps - \partial_u H\|_{L^{2, \theta}}^2\le 2\|\Phi(\bar\rho_\theta)\|_\infty\|\partial_u H_\eps - \partial_u H\|_{L^2}^2\to0$ as $\eps\to0$. Using \eqref{eq:martingale_problem}, the process $\{\mathscr{M}_t(H_\eps), t\in[0, T]\}$ defined via
\begin{equation}\label{eq:extended_mg_eps}
	\mathscr{M}_t (H_\eps) := \mathscr{Y}_t (H_\eps) - \mathscr{Y}_0 (H_\eps) - \int_0^t \mathscr{Y}_s (\mathfrak{A}_\theta H_\eps) \de s
\end{equation}
is a martingale with continuous trajectories and with quadratic variation $\scal{\mathscr{M}(H_\eps)}_t = t\|\partial_u H_\eps\|_{L^{2, \theta}}^2$.

We claim that the sequence of real-valued martingales $\{\mathscr{M}_t(H_\eps), t\in[0, T]\}_{\eps>0}$ converges in $\mathscr{B}$ to a martingale, denoted by $\{\mathscr{\tilde M}_t(H), t\in[0, T]\}$, with continuous trajectories and with quadratic variation given by \eqref{eq:qv_tilde_M}. To this end, for any $\eps, \eps'>0$ note that, by Doob's inequality, 
\begin{align*}
	\| \mathscr{M}(H_\eps) - \mathscr{M}(H_{\eps'})\|_{\mathscr{B}}^2 & \lesssim \mysup_{t\in [0,T]} \expected \left[\big(\mathscr{M}_t (H_\eps) - \mathscr{M}_t (H_{\eps'})\big)^2 \right]
    \\&=\mysup_{t\in [0,T]} \expected \left[\big(\mathscr{M}_t (H_\eps -H_{\eps'})\big)^2 \right] 
    \\&= T \| \partial_u H_\eps - \partial_u H_{\eps'}\|_{L^{2, \theta}}^2,
\end{align*}
and the last term converges to $0$ as $\eps, \eps' \to 0$. Hence, $\{\mathscr{M}(H_\eps)\}_{\eps}$ is a Cauchy sequence in $\mathscr{B}$, and thus it converges to a process in $\mathscr{B}$, which we denote by $\mathscr{\tilde M}(H)$. One can easily check that the limit only depends on $H$ and not on the approximating sequence $\{H_\eps\}_{\eps>0}$, so that the notation is justified. Moreover, we have that 
\begin{equation*}
    \expected \left[\mathscr{\tilde M}_t (H)^2\right]=\mylim_{\eps\to 0} \expected \left[\mathscr{M}_t (H_\eps)^2\right] = t \mylim_{\eps\to 0} \|\partial_u H_\eps\|_{L^{2, \theta}}^2 =t \|\partial_u  H \|_{L^{2, \theta}}^2,
\end{equation*}
so that the quadratic variation of $\mathscr{\tilde M} (H)$ is given by \eqref{eq:qv_tilde_M}.
	
Now, by Lemma~\ref{lemma:def_Y_eps} and since $\mathfrak{A}_\theta H_\eps\in\mathcal{S}$, the processes $\mathscr{Y}(H_\eps)$ and $\mathscr{Y}(\mathfrak{A}_\theta H_\eps)$ converge in $\mathscr{H}$ to $\mathscr{Y}(H)$ and $\mathscr{Y}(\mathfrak{A}_\theta H)$, respectively. By the Cauchy-Schwarz inequality and item i) of Proposition~\ref{prop:uniqueness_neg}, the sequence of real-valued processes $\{\int_0^{t} \mathscr{Y}_s (\mathfrak{A}_\theta H_\eps)\de s, t \in [0,T]\}_{\eps>0}$ is a Cauchy sequence in $\mathscr{B}$, which follows from the fact that 
\begin{equation*}
    \|\mathfrak{A}_\theta H_\eps - \mathfrak{A}_\theta    
    H\|_{L^2}^2\le \|\Phi'(\bar\rho_\theta)\|_\infty^2\|\partial_u^2H_\eps-\partial_u^2 H\|_{L^2}^2 \to 0
\end{equation*}
as $\eps\to 0$. Thus, as $\eps\to0$, this sequence converges to a process $\{\mathscr{Z}_t(\mathfrak{A}_\theta H), t\in [0,T]\}$ in $\mathscr{B}$ (where, again, the limiting process is independent of the approximating sequence), and the convergence holds also in $\mathscr{H}$.
	
Hence, from \eqref{eq:extended_mg_eps} we get that in $\mathscr{H}$, and consequently almost everywhere in time and $\mathbb P$-almost surely, the equality
\begin{equation*}
	\mathscr{\tilde M}_t (H) = \mathscr{Y}_t (H) - \mathscr{Y}_0 (H)  - \mathscr{Z}_t (\mathfrak{A}_\theta H)
\end{equation*}
holds. Since $\mathscr{\tilde M}(H)$ and $\mathscr{Z} (\mathfrak{A}_\theta H)$ have continuous trajectories,  this implies that the same holds for the process $\mathscr{Y}(H)$. Finally, it is easy to show that $\mathbb P$-almost surely we have the equality 
\begin{equation*}
	\mathscr{Z}_t (\mathfrak{A}_\theta H) =\int_0^t \mathscr{Y}_s (\mathfrak{A}_\theta H) \de s,
\end{equation*}
which completes the proof.
\end{proof}

Following \cite[Lemma~5.5]{bgjs22}, we are now going to construct a function which is neither in $\mathcal{S}_{\theta, \Dir}$ nor in $ \mathcal{\tilde S}$ for which we can properly define the martingale problem of Proposition~\ref{prop:uniqueness_neg}. This intermediate step will allow us to define the martingale problem for any test function $H \in \mathcal{S}_{\theta, \Dir}$. 

Let $a:\mathbb{R}\to\mathbb{R}$ be the map defined via
\begin{equation*}
    a(u):=c e^{-\frac{1}{u(1-u)}}\boldsymbol{1}_{(0,1)}(u), \quad  \text{where } c:=\left(\int_0^1 e^{-\frac{1}{u(1-u)}}\de u\right)^{-1},
\end{equation*}
and define $\phi:\mathbb{R} \to \mathbb{R}$ by
\begin{equation}\label{eq:function_phi}
     \phi(u):=1-\int_0^u a(t)\de t. 
\end{equation}
For $\beta\in(0, 1)$ and $\alpha>\frac{1}{1-\beta}$, let $\psi_{\alpha, \beta}:\mathbb{R}\to\mathbb{R}$ be defined via
\begin{equation}\label{eq:psi_alpha_beta}
    \psi_{\alpha, \beta}(u):=u\,\phi(\alpha(u-\beta))
\end{equation}
(so that, in particular, $\psi_{\alpha, \beta} (u) = u$ for $u \in [0,\beta]$ and $\psi_{\alpha, \beta}(u)=0$ for $u\in [\beta+\tfrac{1}{\alpha},1]$; see Figure~\ref{fig:psi_alpha_beta}). We set $\tilde{\psi}_{\alpha, \beta}:=\psi_{\alpha, \beta} \circ i$, where $i:[0, 1]\mapsto[0, 1]$ is given by $i(u):=1-u$ (so that $\tilde{\psi}_{\alpha, \beta}$ is simply the reflection of $\psi_{\alpha, \beta}$ with respect to the axis $x=\frac{1}{2}$).

\begin{figure}[htbp]
    \centering
    \begin{subfigure}[t]{0.49\textwidth}
        \centering
        \includegraphics[width=\textwidth]{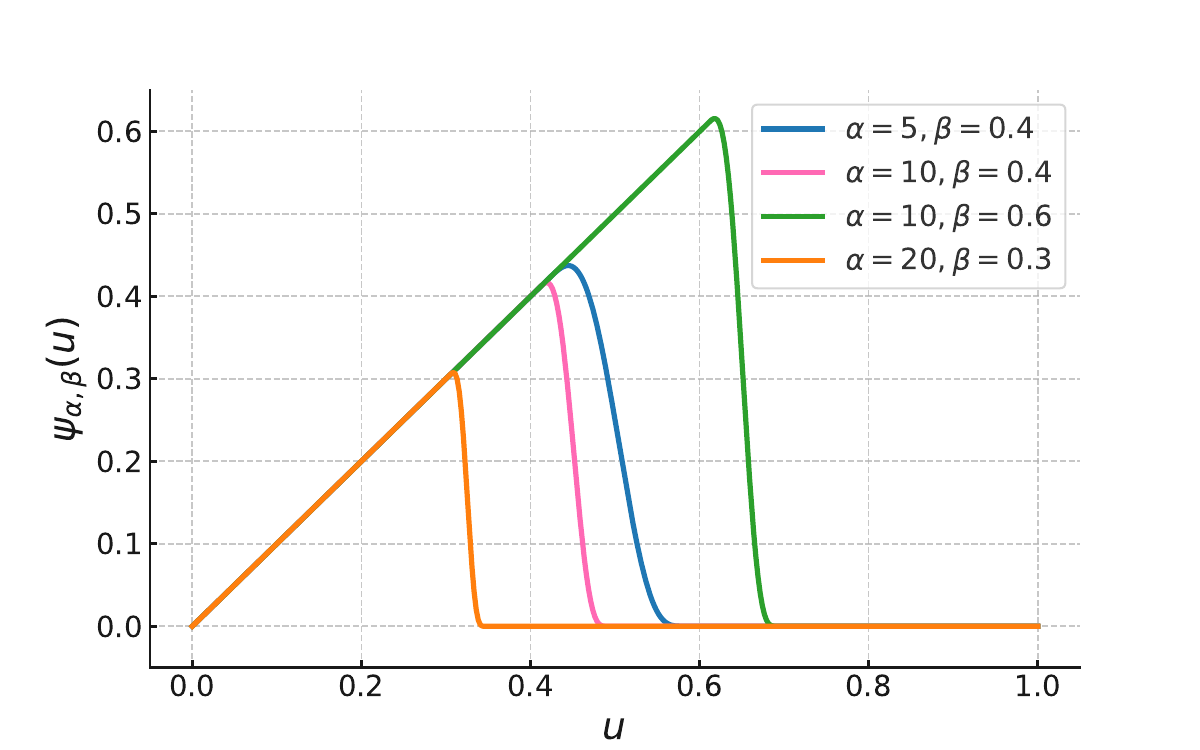}
        \caption{Plot of $\psi_{\alpha,\beta}$.}
        \label{fig:psi_alpha_beta}
    \end{subfigure}
    \hfill
    \begin{subfigure}[t]{0.49\textwidth}
        \centering
        \includegraphics[width=\textwidth]{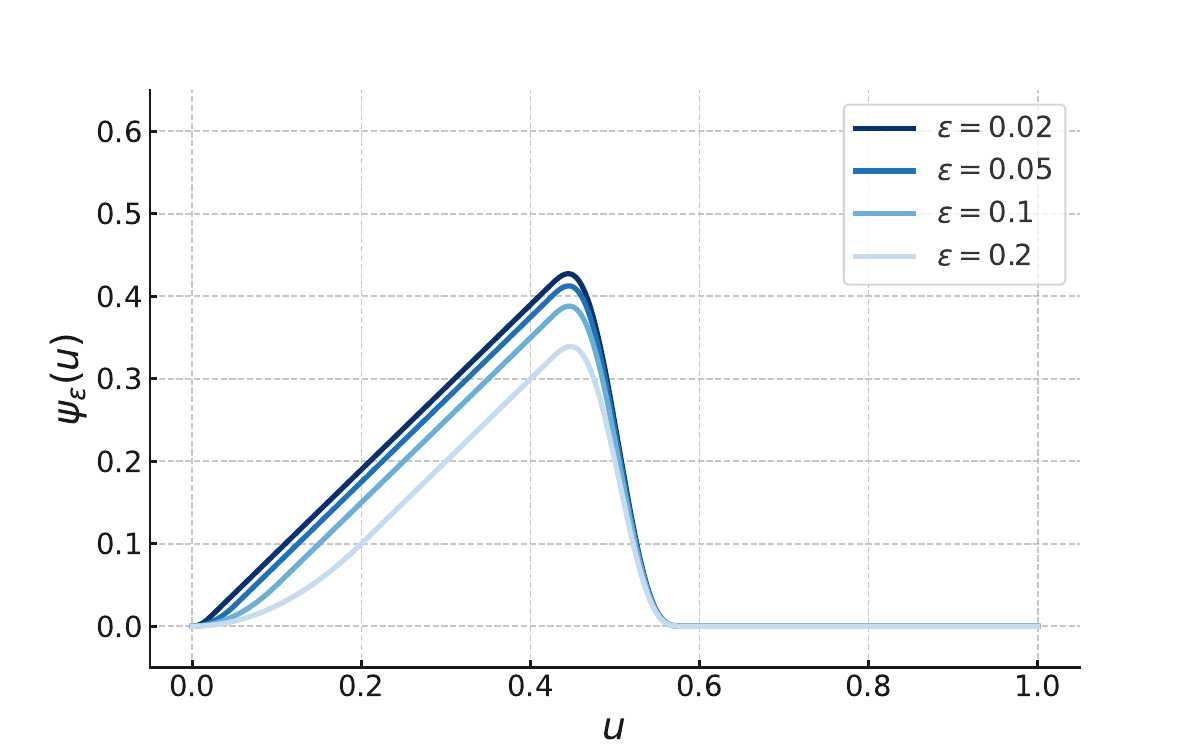}
        \caption{Plot of $\psi_\eps$ for $\alpha=5$ and $\beta=0.4$.}
        \label{fig:psi_eps}
    \end{subfigure}
    \caption{Plot of the functions $\psi_{\alpha,\beta}$ and $\psi_\eps$ for some values of the parameters $\alpha, \beta, \eps$.}
\end{figure}

\begin{lemma}\label{lemma:mg_problem_psi} Let $\mathscr{Y}$ be an $\mathcal{S}'$-valued stochastic process satisfying \eqref{eq:martingale_problem} and items i) and ii) of Proposition~\ref{prop:uniqueness_neg} for each $H\in\mathcal{S}$. Fix $\beta\in(0, 1)$ and $\alpha>\frac{1}{1-\beta}$. For $\psi\in\{\psi_{\alpha, \beta}, \tilde\psi_{\alpha, \beta}\}$, the process $\{\mathscr{\tilde M}_t(\psi), t\in[0, T]\}$ defined via 
\begin{equation}\label{eq:tilde_M_psi}
    \mathscr{\tilde M}_t(\psi):=\mathscr{Y}_t(\psi)-\mathscr{Y}_0(\psi)-\int_0^t\mathscr{Y}_s(\mathfrak{A}_\theta \psi)\de s
\end{equation}
is a martingale with continuous trajectories and with quadratic variation given by
\begin{equation}\label{eq:qv_tilde_M_psi}
    \scal{\mathscr{\tilde M}(\psi)}_t=t\|\partial_u \psi\|_{L^{2, \theta}}^2.
\end{equation}
Moreover, the process $\{\mathscr{Y}_t(\psi), t\in[0, T]\}$ has continuous trajectories.
\end{lemma}

\begin{proof} Since the proof for $\tilde{\psi}_{\alpha, \beta}$ is completely analogous, we will only show the result for $\psi=\psi_{\alpha, \beta}$. For $0<\eps<1$, define the map
\begin{equation*}
	h_\eps(u):=\begin{cases}
	\frac{u^2}{2\eps}, & \text{ for } u \in [0, \eps],\\
	u-\frac{\eps}{2}, & \text{ for } u \in [\eps, 1],
	\end{cases}
\end{equation*}
which belongs to $\mathcal{H}^2([0,1])$, and let ${\phi}_{\alpha,\beta}(u):= \phi(\alpha(u-\beta))$ (with $\phi$ given in \eqref{eq:function_phi}). Recalling \eqref{eq:psi_alpha_beta}, it is not hard to show that $\psi_{\eps}:=h_{\eps}{\phi}_{\alpha,\beta} \in \mathcal{\tilde S}$; see Figure~\ref{fig:psi_eps}.

Moreover, defining the Sobolev space $\mathcal{H}^1([0,1])$ in analogy to $\mathcal{H}^2([0,1])$, the sequence $\{\psi_\eps\}_{\eps>0}$ converges in $\mathcal{H}^1([0,1])$ to $\psi$ as $\eps\to0$: indeed, a simple computation shows that
\begin{equation*}
	\mylim_{\eps\to0}\left(\int_0^1 (\psi_\eps(u)-\psi(u))^2 \de u + \int_0^1 (\partial_u\psi_\eps(u)-\partial_u\psi(u))^2 \de u\right)=0.
\end{equation*}

Now, for $\eps>0$ sufficiently small, by Lemma~\ref{lemma:mg_problem_tilde_S} we have that the process $\{\mathscr{\tilde M}_t(\psi_\eps),$ $t\in[0, T]\}$ defined via
\begin{equation*}
	\mathscr{\tilde M}_t(\psi_\eps) := \mathscr{Y}_t (\psi_\eps) - \mathscr{Y}_0 (\psi_\eps) - \int_0^t \mathscr{Y}_s(\mathfrak{A}_\theta \psi_\eps) \de s
\end{equation*}
is a real-valued martingale with continuous trajectories and quadratic variation given by $\scal{\mathscr{\tilde M}(\psi_\eps)}_t=t\|\partial_u \psi_\eps\|_{L^{2, \theta}}^2$, and moreover the process $\mathscr{Y}(\psi_\eps)$ has continuous trajectories. Also, since $\mylim_{\eps \to 0} \|\partial_u\psi_\eps - \partial_u\psi\|_{L^{2, \theta}}^2 = 0$, the sequence of martingales $\{\mathscr{\tilde M} (\psi_\eps)\}_{\eps>0}$ converges in $\mathscr{B}$ as $\eps\to0$ to a martingale, which we denote by $\mathscr{\tilde M}(\psi)$, which has continuous trajectories and quadratic variation given by \eqref{eq:qv_tilde_M_psi}. Moreover, we have that $\mathscr{Y}(\psi_\eps)$ converges to $\mathscr{Y}(\psi)$ in $\mathscr{H}$ as $\eps\to0$. This can be seen as in the proof of Lemma~\ref{lemma:mg_problem_tilde_S}.
	
Setting now
\begin{equation*}
    x^\eps_t :=\int_0^t {\mathscr{Y}}_s (\mathfrak{A}_\theta \psi_\eps)\de s,
\end{equation*}
it only remains to show the convergence of the sequence $\{ x^\eps_t, t\in [0,T]\}_{\eps>0} \subseteq \mathscr{B}$ to the process $\{ x_t, t\in [0, T]\}$ given by $x_t:=\int_0^t \mathscr{Y}_s (\mathfrak{A}_\theta \psi) \de s$ as $\eps\to0$. Note that
\begin{align*}
	\mathfrak{A}_\theta \psi_\eps(u)&=\Phi'(\bar\rho_\theta(u))\partial_u^2\psi_\eps(u)
    \\&=\Phi'(\bar\rho_\theta(u))\phi_{\alpha,\beta}(u) \partial_u^2h_\eps(u) 
    \\&\phantom{=}+ \Phi'(\bar\rho_\theta(u))2\partial_u\phi_{\alpha,\beta}(u)\partial_u h_\eps(u)+\Phi'(\bar\rho_\theta(u))h_\eps(u)\partial_u^2 {\phi}_{\alpha,\beta}(u).
\end{align*} 
The function given by the bottom line of the last display converges in $L^2([0, 1])$ to $ \Phi'(\bar\rho_\theta)\partial_u^2\psi$ $=\mathfrak{A}_\theta\psi$ as $\eps\to0$, which is easy to see from the identity $\partial_u^2\psi(u)= 2\partial_u\phi_{\alpha,\beta} (u)+u\, \partial_u^2\phi_{\alpha,\beta}(u)$ for any $u \in[0,1]$. At the same time, we have that $\Phi'(\bar\rho_\theta)\phi_{\alpha,\beta}\, \partial_u^2h_\eps = \Phi'(\bar\rho_\theta)\iota^0_\eps$, and hence, by condition ii) of Proposition~\ref{prop:uniqueness_neg}, we have that
\begin{equation*}
	\mylim_{\eps\to0} \expected \left[ \mysup_{t \in [0,T]} \left(\int_0^t \mathscr{Y}_s(\Phi'(\bar\rho_\theta)\phi_{\alpha,\beta}\,\partial_u^2 h_\eps )\de s\right)^2\right] =0.
\end{equation*}
Hence, we get that $\{x^\eps\}_{\eps>0}$ converges to $x$ as $\eps\to 0$ in $\mathscr{B}$.
\end{proof}

From Lemmas~\ref{lemma:mg_problem_tilde_S} and \ref{lemma:mg_problem_psi}, we deduce the following corollary.

\begin{corollary} Let $\mathscr{Y}$ be an $\mathcal{S}'$-valued stochastic process satisfying \eqref{eq:martingale_problem} and items i) and ii) of Proposition~\ref{prop:uniqueness_neg} for each $H\in\mathcal{S}$. Then, the process $\{\mathscr{Y}_t, t\in[0, T]\}$ can be extended to a process $\{\mathscr{\tilde Y}_t, t\in[0, T]\}$ belonging to $\mathcal{C}([0,T], \mathcal{S}_{\theta, \Dir})$ and solution to the martingale problem of Proposition~\ref{prop:uniqueness} (with $\theta<0$ and with $\mathcal{S}_\theta$ substituted with $\mathcal{S}_{\theta, \Dir}$).
\end{corollary}

\begin{proof} For any $\beta\in (0,1)$ and $\alpha>\frac{1}{1-\beta}$, we can decompose any function $H$ in $\mathcal{S}_{\theta, \Dir}$ as
\begin{equation*}
	H=\partial_u H(0)\psi_{\alpha, \beta}-\partial_u H(1)\tilde{\psi}_{\alpha, \beta} + (H-\partial_u H(0)\psi_{\alpha, \beta}+\partial_u H(1)\tilde{\psi}_{\alpha, \beta}).
\end{equation*}
By the properties of $H$, $\psi_{\alpha, \beta}$ and $\tilde \psi_{\alpha, \beta}$, it is easy to check that the map
\begin{equation*}
    G_{\alpha, \beta}:=H-\partial_u H(0)\psi_{\alpha, \beta}+\partial_u H(1)\tilde{\psi}_{\alpha, \beta}
\end{equation*} 
belongs to $\mathcal{\tilde S}$: indeed, both $G_{\alpha, \beta}(0)=G_{\alpha, \beta}(1)=0$ and $\partial_u G_{\alpha, \beta}(0)=\partial_u G_{\alpha, \beta}(1)=0$ follow immediately. 

Also, by the linearity of the martingale problem, the maps $\partial_u H(0)\psi_{\alpha, \beta}$ and $\partial_u H(1)\tilde{\psi}_{\alpha, \beta}$ are suitable for Lemma~\ref{lemma:mg_problem_psi}. This implies that there exists a continuous martingale, which we denote by $\{\mathscr{\tilde M}^{\alpha, \beta}_t(H), t\in[0, T]\}$, and given by
\begin{equation*}
    \mathscr{\tilde M}^{\alpha, \beta}_t(H):=\partial_uH(0) \mathscr{\tilde M}_t (\psi_{\alpha, \beta}) - \partial_uH(1) \mathscr{\tilde M}_t (\tilde \psi_{\alpha, \beta}) + \mathscr{\tilde M}_t (G_{\alpha, \beta}),
\end{equation*}
(with $\mathscr{\tilde M}(\psi_{\alpha, \beta}), \mathscr{\tilde M} (\tilde \psi_{\alpha, \beta})$ and $\mathscr{\tilde M}(G_{\alpha, \beta})$ defined via \eqref{eq:tilde_M_psi} and \eqref{eq:tilde_M}) which has continuous trajectories and satisfies 
\begin{equation*}
	\mathscr{\tilde M}^{\alpha, \beta}_t(H)=\mathscr{Y}_t(H)-\mathscr{Y}_0(H)-\int_0^t \mathscr{Y}_s (\mathfrak{A}_\theta H)\de s.
\end{equation*}
Moreover, for any $H \in \mathcal{S}_{\theta, \Dir}$, the real-valued process $\{\mathscr{Y}_t(H), t\in [0,T]\}$ has continuous trajectories, which implies that the process $\{\mathscr{Y}_t(H), t\in [0,T]\}$ belongs to  $\mathcal{C}([0,T], \mathcal{S}_{\theta, \Dir}')$. 
	
We now claim that, by choosing $\beta=\frac{1}{\alpha}$, the sequence of martingales $\{ \mathscr{\tilde M}^{\alpha, \beta}(H)\}_{\alpha>0}$ converges in $\mathscr{B}$, as $\alpha\to\infty$, to a martingale $\mathscr{\tilde M}(H)$ whose quadratic variation is given by $\scal{\mathscr{\tilde M}(H)}_t=t\|\partial_u H\|_{L^{2, \theta}}^2$. It is easy to check that, if $\beta=\frac{1}{\alpha}$, then
\begin{equation*}
	\mylim_{\alpha \to \infty} \|\partial_u \psi_{\alpha, \beta}\|_{L^{2, \theta}}^2 = \mylim_{\alpha \to \infty} \|\partial_u \tilde \psi_{\alpha, \beta}\|_{L^{2, \theta}}^2 = 0,
\end{equation*}
since the supports of $\psi_{\alpha, \frac{1}{\alpha}}$ and $\tilde \psi_{\alpha, \frac{1}{\alpha}}$ converge to the empty set as $\alpha$ goes to infinity. This implies that
\begin{equation*} 
	\mylim_{\alpha \to \infty} \| \partial_u G_{\alpha, \beta} - \partial_u H \|_{L^{2, \theta}}^2 = 0.  
\end{equation*}
Moreover, recalling that $\scal{\mathscr{\tilde M} (\psi_{\alpha, \beta})}_t = t\|\partial_u \psi_{\alpha, \beta}\|_{L^{2, \theta}}^2$, $\scal{\mathscr{\tilde M} (\tilde\psi_{\alpha, \beta})}_t = t\| \partial_u \tilde\psi_{\alpha, \beta}\|_{L^{2, \theta}}^2$ and $\scal{\mathscr{\tilde M} (G_{\alpha, \beta} )}_t = t\|\partial_u G_{\alpha, \beta}\|_{L^{2, \theta}}^2$, by Doob's inequality we conclude that $\{\mathscr{\tilde M}^{\alpha, \beta}\}_{\alpha>0}$ is a Cauchy sequence in $\mathscr{B}$. Thus, this sequence converges to a martingale $\mathscr{\tilde M}(H)$ in $\mathscr{B}$ whose quadratic variation satisfies
\begin{equation*}
	\scal{\mathscr{\tilde M}(H)}_t = \mylim_{\alpha \to \infty} \scal{\mathscr{\tilde M}^{\alpha, \beta} (H)}_t = t \| \partial_u H \|_{L^{2, \theta}}^2, 
\end{equation*}
as required.
\end{proof}

This completes the proof of Proposition~\ref{prop:uniqueness_neg}.

%%%%%%%%%%%%%%%%%%%%%%%%%%%%%%%%%%%%%%%%%%%%%%%%%%
%%%REFERENCES%%%%%%%%%%%%%%%%%%%%%%%%%%%%%%%%%%%%%
%%%%%%%%%%%%%%%%%%%%%%%%%%%%%%%%%%%%%%%%%%%%%%%%%%
\bibliographystyle{Martin} 
\bibliography{references}

\end{document}